\documentclass[12pt,reqno]{amsart}
\usepackage{hyperref}
\usepackage{amsmath,amssymb,amscd,amsthm,amscd,amsthm,amsfonts,mathrsfs,array}
\usepackage{graphicx,color}
\usepackage[all]{xy}
\usepackage{tikz}
\usepackage{multirow}
\usepackage{here, latexsym, bm, ascmac, mathtools, multicol, tcolorbox, subfig}
\usetikzlibrary{intersections, calc, arrows.meta}
\usetikzlibrary{decorations, decorations.markings}
\usetikzlibrary{snakes,angles,quotes}
\usetikzlibrary{intersections,calc,arrows.meta}
\numberwithin{equation}{section}

\newtheorem{thm}{Theorem}[section]
\newtheorem{lem}{Lemma}[section]
\newtheorem{prop}{Proposition}[section]
\newtheorem{cor}{Corollary}[section]

\theoremstyle{definition}

\newtheorem{defn}{Definition}[section]
\newtheorem{rem}{Remark}[section]

\DeclareMathOperator{\grad}{grad}

\begin{document}

\title[Correspondences between Gradient Trees and  Holomorphic Disks]{Explicit correspondences between gradient trees in $\mathbb{R}$ and holomorphic disks in $T^{*}\mathbb{R}$}
\author[H. Suzuki]{Hidemasa SUZUKI}
\address{Department of Mathematics and Informatics, Graduate School of Science and Engineering, Chiba University,
Yayoicho 1-33, Inage, Chiba, 263-8522 Japan.}
\email{hsuzuki@g.math.s.chiba-u.ac.jp}
\date{}

\maketitle

\begin{abstract}
  Fukaya and Oh studied the correspondence between pseudoholomorphic disks in $T^{*}M$ which are bounded by Lagrangian sections $\{L_{i}^{\epsilon}\}$ and gradient trees in $M$ which consist of gradient curves of $\{f_{i}-f_{j}\}$. Here, $L_{i}^{\epsilon}$ is defined by $L_{i}^{\epsilon}=$\,graph$(\epsilon df_{i})$. They constructed approximate pseudoholomorphic disks in the case $\epsilon>0$ is sufficiently small. When $M=\mathbb{R}$ and Lagrangian sections are affine, pseudoholomorphic disks $w_{\epsilon}$ can be constructed explicitly. In this paper, we show that pseudoholomorphic disks $w_{\epsilon}$ converges to the gradient tree in the limit $\epsilon\to+0$ when the number of Lagrangian sections is three and four.
\end{abstract} 

\tableofcontents

\section{Introduction.}
Pseudoholomorphic disks are known as fundamental tools in modern symplectic geometry. Pseudoholomorphic disks are first introduced by Gromov \cite{Gro85} (see \cite{MS04}), and there are many applications of its theory. In Floer theory, counting of pseudoholomorphic disks bounded by Lagrangian submanifolds defines boundary operators of the Floer complexes. These disks are furthermore essential for defining the $A_{\infty}$-structures of the Fukaya categories. In certain special cases, combinatorial techniques can work to count pseudoholomorphic disks. However, counting pseudoholomorphic disks remains to be a challenging problem in general cases. 

Fukaya and Oh proved that the moduli space of pseudoholomorphic disks $\mathcal{M}_{J}(T^{*}M;\vec{L^{\epsilon}},\vec{x^{\epsilon}})$ bounded by Lagrangian sections $L_{1}^{\epsilon},L_{2}^{\epsilon},\dots,L_{k}^{\epsilon}$ of the cotangent bundle $T^{*}M$ and the moduli space of gradient trees $\mathcal{M}_{g}(M;\vec{f},\vec{p})$ constructed by functions $f_{1},f_{2},\dots,f_{k}$ on the Riemannian manifold $M$ are diffeomorphic to each other when $\epsilon>0$ is sufficiently small (\cite{FO97}). Here, $L_{i}^{\epsilon}$ is defined by $L_{i}^{\epsilon}\coloneqq\operatorname{graph}(\epsilon df_{i})$ for $i=1,2,\dots,k$. Pseudoholomorphic disks are pseudoholomorphic maps from the unit disk $D^{2}$ to $T^{*}M$. Fukaya and Oh constructed them approximately in $T^{*}M$ first, and show the existence of an exact solution in a neighborhood of the approximate one. Since we assume $M=\mathbb{R}$ in this paper, pseudoholomorphic disks are holomorphic maps from $D^{2}$. Hereafter, we call elements of $\mathcal{M}_{J}(T^{*}\mathbb{R};\vec{L^{\epsilon}},\vec{x^{\epsilon}})$ holomorphic disks. We consider holomorphic disks as maps from the closure of the upper half plane $\mathbb{H}$ instead of $D^{2}$ to make later analysis easier in this paper. When Lagrangian sections of $T^{*}\mathbb{R}=\mathbb{C}$ are affine, holomorphic disks can be described by Schwarz-Christoffel maps. Here, a Schwarz-Christoffel map is a conformal map from $\mathbb{H}$ to a polygonal domain in $\mathbb{C}$. On the other side, gradient trees are continuous maps from trees to $M$, and they map each edges to gradient curves of $f_{i}-f_{j}\,(i\neq j)$. Since each Lagrangian section $L_{i}^{\epsilon}$ of $T^{*}\mathbb{R}$ is affine, each function $f_{i}$ is described by $f_{i}(x)=a_{i}x^{2}+b_{i}x+c_{i}\,(a_{i},b_{i},c_{i}\in\mathbb{R})$ for $i=1,2,\dots,k$. In this case, gradient curves of $f_{i}-f_{j}$ are described by exponential functions if $a_{i}\neq a_{j}$, and by polynomials of degree $1$ if $a_{i}=a_{j}$. In this paper, we first describe holomorphic disks and gradient trees explicitly, and show the correspondence between gradient trees and pseudoholomorphic disks in the case $M=\mathbb{R}$ and $k=3,4$. In particular, we prove that the holomorphic disk $w_{\epsilon}$ we consider below converges to the gradient tree.

In the case $M=\mathbb{R},k=3$, we can consider a holomorphic disk which is the map from the upper half plane to a triangle $x_{1}^{\epsilon}x_{2}^{\epsilon}x_{3}^{\epsilon}$. Here, $x_{i}^{\epsilon}$ is the intersection point of $L_{i}^{\epsilon}$ and $L_{i+1}^{\epsilon}$. We set this holomorphic disk as the Schwarz-Christoffel $w_{\epsilon}$ satisfying $w_{\epsilon}(z_{i})=x_{i}^{\epsilon}$ for $i=1,2,3$, where $z_{1}=1,z_{2}=\infty,z_{3}=0$. On the other hand, when we take $f_{1},f_{2},f_{3}$ such that we can construct $w_{\epsilon}$ above, we obtain the unique gradient tree $I\in\mathcal{M}_{g}(\mathbb{R};\vec{f},\vec{p})$. The tree which is used to construct the gradient tree has three exterior edges and one interior vertex. Here, we denote by $I_{i}$ the restriction of $I$ to the exterior edge $e_{i}$ of tree for $i=1,2,3$. For each $i=1,2,3$, we take a neighborhood $D_{z_{i}}(\delta)$ of $z_{i}$, and set a transformation $\phi_{i,\delta}$ from the infinite stripe $\Theta=\{(\tau,\sigma)\mid\tau\in(-\infty,0),\sigma\in[0,1]\}$ to $D_{z_{i}}(\delta)$ for each $i=1,2,3$. When we regard $\mathbb{R}$ as the zero section of $T^{*}\mathbb{R}$, we obtain the uniformly convergence of $w_{\epsilon}\circ\phi_{i,\delta}(\tau,\sigma)-I_{i}(\epsilon\tau)$ to $0$. We also can prove the uniformly convergence of $w_{\epsilon}$ to the constant map to $p_{0}$ on $\overline{\mathbb{H}}\setminus(D_{z_{1}}(\delta)\cup D_{z_{2}}(\delta)\cup D_{z_{3}}(\delta))$. Here, the point $p_{0}$ is the intersection point of the gradient curves, which is also the image of the interior vertex of tree.

In the case $M=\mathbb{R},k=4$, we set $f_{1},f_{2},f_{3},f_{4}$ such that there exists a convex quadrilateral $x_{1}^{\epsilon}x_{2}^{\epsilon}x_{3}^{\epsilon}x_{4}^{\epsilon}$ which has vertices $x_{1}^{\epsilon},x_{2}^{\epsilon},x_{3}^{\epsilon},x_{4}^{\epsilon}$ in counterclockwise order. In this situation, we obtain the unique gradient tree. Let $w_{\epsilon}$ be the Schwarz-Christoffel map from the upper half plane with four marked points $z_{1},z_{2},z_{3},z_{4,\epsilon}$ to the convex quadrilateral $x_{1}^{\epsilon}x_{2}^{\epsilon}x_{3}^{\epsilon}x_{4}^{\epsilon}$ such that $w_{\epsilon}(z_{i})=x_{i}^{\epsilon}$ for $i=1,2,3$. Here, we fix $z_{1}=1,z_{2}=\infty,z_{3}=0\in\mathbb{R}$, and $z_{4,\epsilon}\in (0,1)$ moves as $\epsilon$ varies. We can also describe $w_{\epsilon}$ by a Schwarz-Christoffel map even if the quadrilateral $x_{1}^{\epsilon}x_{2}^{\epsilon}x_{3}^{\epsilon}x_{4}^{\epsilon}$ is not convex. In this case, the moduli space $\mathcal{M}_{g}(\mathbb{R};\vec{f},\vec{p})$ is not empty, but is of one-dimension. When the moduli space of holomorphic disks $\mathcal{M}_{J}(T^{*}\mathbb{R};\vec{L^{\epsilon}},\vec{x^{\epsilon}})$ is of zero-dimensional, the element of $\mathcal{M}_{J}(T^{*}\mathbb{R};\vec{L^{\epsilon}},\vec{x^{\epsilon}})$ are used to define $A_{\infty}$-structures of the Fukaya category and boundary operators of the Floer complexes. We discuss about these holomorphic disks only in this paper. Though we hope to prove a similar theorem as that in the case $k=3$, however there are two problems. The first problem is how we divide the upper half plane appropriately. In this paper, we divide the upper half plane into some regions which correspond to exterior edges or the interior vertex of the tree. In the case $k=3$, there is the only tree which has three external vertices. On the other hand, in the case $k\geq4$, there are many trees which have $k$ external vertices. This fact makes the study of correspondence between pseudoholomorphic disks and gradient trees harder. It is known that there is a correspondence between the behavior of $z_{4}^{\epsilon}$ and planer trees (see \cite{Dev12}). From this correspondence, we expect that the way to divide the upper half plane depends on the behavior of $z_{4,\epsilon}$. We verify it in subsection 4.5. Also, the second problem is how to study the behavior of $z_{4,\epsilon}$. In this paper, we examine the behavior of $z_{4,\epsilon}$ in the case the convex quadrilateral is generic in the sense that $L_{i}\cap L_{j}\neq \emptyset,L_{i}\neq L_{j}\,(i\neq j)$ and $p_{1}\neq p_{3},p_{2}\neq p_{4}$. We use conformal moduli which is a conformal invariant of quadrilaterals and the ratio $\left\lvert x_{1}^{\epsilon}-x_{2}^{\epsilon}\right\rvert/\left\lvert x_{3}^{\epsilon}-x_{4}^{\epsilon}\right\rvert$ or $\left\lvert x_{2}^{\epsilon}-x_{3}^{\epsilon}\right\rvert/\left\lvert x_{4}^{\epsilon}-x_{1}^{\epsilon}\right\rvert$ in the convex quadrilateral $x_{1}^{\epsilon}x_{2}^{\epsilon}x_{3}^{\epsilon}x_{4}^{\epsilon}$. After studying the behavior of $z_{4,\epsilon}$, we show the main theorems in the case $M=\mathbb{R},k=4$.

The outline of this paper is the following. In Section 2, we first recall definitions of gradient trees and pseudoholomorphic disks. In order to describe pseudoholomorphic disks explicitly, we next explain the notion of the Schwarz-Christoffel map. We next prepare hypergeometric functions for expanding the Schwarz-Christoffel map in power series. The Schwarz-Christoffel map is written by an integration of power functions. By comparing the integral representations of the Schwarz-Christoffel map and hypergeometric functions, the Schwarz-Christoffel map can be described by hypergeometric functions (or series). We also prepare one of the connection formulas of hypergeometric functions for describing pseudoholomorphic disks in the regions which are little far from preimage of vertices of polygons. We next recall the notion of conformal moduli of quadrilaterals which are conformal quantities for the unit disk (or the upper half plane) with marked points. We also discuss an inequation for conformal moduli which we use to calculate the limit value of $z_{4,\epsilon}$ at the end of this section. In Section 3 and Section 4, we first explain how to divide the upper half plane into some regions, and give conformal transformation between stripes and these regions. Thereafter, we give the main theorems in the case $M=\mathbb{R},k=3,4$. In order to show main theorems, we first consider when functions $f_{1},f_{2},\dots,f_{k}$ give us convex polygons bounded by affine Lagrangian sections. We then construct gradient trees associated to $f_{1},f_{2},\dots,f_{k}$. We next study power series representation of pseudoholomorphic disks by using integral representations and connection formulas of hypergeometric functions. In the case $k=4$, studying the behavior of $z_{4,\epsilon}$ is important to know how to divide the upper half plane into several regions. Especially, we study the limit value of $z_{4,\epsilon}$ and the principal part of $z_{4,\epsilon}$ at $\epsilon\to+0$ in the case the convex quadrilateral is generic. Then, we ensure the method of dividing the upper half plane, and we prove our main theorems.

\section*{Acknowledgements.}
The author is grateful to the advisor, Hiroshige Kajiura, for sharing his insights and for valuable advice. The author would also like to thank Manabu Akaho for discussion on pseudoholomorphic disks and Toshiyuki Sugawa and Tomoki Kawahira for explaining the estimation of conformal modulus. The author is also grateful to Masahiro Futaki, Yasunori Okada, Kazuki Hiroe, Shunya Adachi, Hayato Nakanishi and Azuna Nishida for helpful discussions and for valuable comments. This work was supported by JST SPRING, Grant Number JPMJSP2109.

\section{Preliminaries.}
\subsection{Gradient trees and pseudoholomorphic disks.}
In this subsection, we review two moduli spaces of our concern; one is of gradient trees and the other is of pseudoholomorphic disk (see \cite{FO97}). The moduli space of gradient trees is defined by the moduli space $Gr_{k}$ of metric ribbon trees.
\begin{defn}[\cite{FO97}]
  A \textit{ribbon tree} is a pair $(T, i)$ of a tree $T$ and an embedding $i:T\rightarrow D^{2}\subset \mathbb{C}$ which satisfies the following:\\
  (1) No vertex of $T$ has 2-edges.\\
  (2) If $v\in T$ is a vertex with one edge, then $i(v)\in \partial D^{2}$.\\
  (3) $i(T)\cap\partial D^{2}$ consists of vertices with one edge.
\end{defn}
We identify two pairs $(T, i)$ and $(T', i')$ if $T$ and $T'$ are isometric and $i$ and $i'$ are isotopic. Let $G_{k}$ be the set of all triples $(T, i, v_{1})$, where $(T, i)$ is as above, $v_{1}\in T\cap \partial D^{2}$ and $T\cap \partial D^{2}$ consists of $k$ points. We remark that choosing $v_{1}\in T\cap\partial D^{2}$ is equivalent to choosing an order of $T\cap \partial D^{2}$ which is compatible with the cyclic order of $\partial D^{2}$. 
\begin{defn}[\cite{FO97}]
  We call a vertex an \textit{internal vertex} if it has more than two edges attached to it and call it an \textit{external vertex} otherwise. We call an edge an \textit{internal edge} if both of its vertices are interior and call it an \textit{external edge} otherwise.
\end{defn}
For each $\mathfrak{t}=(T, i, v_{1})\in G_{k}$, let $C_{int}^{1}(T)$ be the set of all internal edges of $T$, and let $Gr(\mathfrak{t})$ be the set of all maps $l: C_{int}^{1}(T)\rightarrow \mathbb{R}^{+}$. We put $Gr_{k}=\bigcup_{\mathfrak{t}\in G_{k}}Gr(\mathfrak{t})$. Let $(T, i, v_{1}, l)\in Gr_{k}$ be a ribbon tree. We identify $T$ with $i(T)$ by the embedding $i$. In this paper, we denote the external edge $e_{i}$ if one of its vertices is $v_{i}$. We define a metric on $T$ such that the exterior edge $e_{i}\,(i=1,\dots,k-1)$ is isometric to $(-\infty , 0]$, exterior edge $e_{k}$ is isometric to $[0,\infty)$ and the interior edge $e$ is isometric to $[0, l(e)]$. We call $l(e)$ the \textit{length} of the interior edge $e$. The unit disk $D^{2}$ is separated into $k$ connected components by $i(T)$. We write one of connected components of $D^{2}\setminus i(T)$ as $D_{i}$ if the closure $\bar{D_{i}}$ contains $v_{i}$ and $v_{i+1}$. Note that, for each edge $e$, there are two subsets $D_{i},D_{j}$ such that its closure contains $e$. We define the integers $lef(e)$ and $rig(e)$ so that the closure of $D_{lef(e)}$ contains $e$ and $D_{lef(e)}$ is on the left side of $e$ with respect to the orientation of $e$ and $\mathbb{R}^{2}$. We define $rig(e)$ in the same way as $lef(e)$. In this paper, we define the orientation of edge $e$ such that $lef(e)=\min\{i,j\},rig(e)=\max\{i,j\}$. In order to define a gradient curve at the external edge $e_{k}$ in a similarly way as at other external edges, we consider the orientation of $e_{k}$ such that $lef(e_{k})=k,rig(e_{k})=1$, and define its metric so that $e_{k}$ is isometric to $(-\infty,0]$.
\begin{figure}[tb]
  \centering
  \begin{tikzpicture}
    \fill[black] (-6,{sqrt(3)}) circle (0.06);
    \fill[black] (-{sqrt(2)}-7,{sqrt(2)}) circle (0.06);
    \fill[black] (-9,0) circle (0.06);
    \fill[black] (-63/5+4,-6/5) circle (0.06);
    \fill[black] (-7,-2) circle (0.06);
    \fill[black] ({sqrt(3)-7},-1) circle (0.06);
    \draw[thick] ({sqrt(3)-7},-1)--(0.5-7,-0.5)--(-7,-2)--(-6.5,-0.5)--(-6.5,0.7)--(-6,{sqrt(3)})--(-
    6.5,0.7)--(-7.3,0.2)--(-{sqrt(2)}-7,{sqrt(2)})--(-7.3,0.2)--(-9,0)--(-7.3,0.2)--(-8/5-7,-6/5);
    \draw[->,>={Stealth[scale=1.25]}] ({0-7},-2)--({0.25-7},-1.25);
    \draw[->,>={Stealth[scale=1.25]}] ({sqrt(3)-7},-1)--({(sqrt(3)+0.5)/2-7},{(-1-0.5)/2});
    \draw[->,>={Stealth[scale=1.25]}] ({0.5-7},-0.5)--({0.5-7},0.1);
    \draw[->,>={Stealth[scale=1.25]}] ({1-7},{sqrt(3)})--({(1+0.5)/2-7},{(sqrt(3)+0.7)/2});
    \draw[->,>={Stealth[scale=1.25]}] ({0.5-7},0.7)--({(0.5-0.3)/2-7},{(0.7+0.2)/2});
    \draw[->,>={Stealth[scale=1.25]}] ({-sqrt(2)-7},{sqrt(2)})--({(-sqrt(2)-0.3)/2-7},{(sqrt(2)+0.2)/2});
    \draw[->,>={Stealth[scale=1.25]}] ({-2-7},0)--({(-2-0.3)/2-7},{(0+0.2)/2});
    \draw[->,>={Stealth[scale=1.25]}] ({-8/5-7},-6/5)--({(-8/5-0.3)/2-7},{(-6/5+0.2)/2});
    \fill[black] (-6.5,-0.5) circle (0.06);
    \fill[black] (-6.5,0.7) circle (0.06);
    \fill[black] (-7.3,0.2) circle (0.06);
    \draw[->,>=stealth,semithick] (-4.5,0)--(-3,0);
    \draw (-3.75,0) node[above]{$i$};
    \draw[thick] (0,0) circle (2);
    \fill[black] (1,{sqrt(3)}) circle (0.06);
    \fill[black] (-{sqrt(2)},{sqrt(2)}) circle (0.06);
    \fill[black] (-2,0) circle (0.06);
    \fill[black] (-8/5,-6/5) circle (0.06);
    \fill[black] (0,-2) circle (0.06);
    \fill[black] ({sqrt(3)},-1) circle (0.06);
    \draw[thick] ({sqrt(3)},-1)--(0.5,-0.5)--(0,-2)--(0.5,-0.5)--(0.5,0.7)--(1,{sqrt(3)})--(0.5,0.7)--(-0.3,0.2)--(-{sqrt(2)},{sqrt(2)})--(-0.3,0.2)--(-2,0)--(-0.3,0.2)--(-8/5,-6/5);
    \draw[->,>={Stealth[scale=1.25]}] (0,-2)--(0.25,-1.25);
    \draw[->,>={Stealth[scale=1.25]}] ({sqrt(3)},-1)--({(sqrt(3)+0.5)/2},{(-1-0.5)/2});
    \draw[->,>={Stealth[scale=1.25]}] (0.5,-0.5)--(0.5,0.1);
    \draw[->,>={Stealth[scale=1.25]}] (1,{sqrt(3)})--({(1+0.5)/2},{(sqrt(3)+0.7)/2});
    \draw[->,>={Stealth[scale=1.25]}] (0.5,0.7)--({(0.5-0.3)/2},{(0.7+0.2)/2});
    \draw[->,>={Stealth[scale=1.25]}] (-{sqrt(2)},{sqrt(2)})--({(-sqrt(2)-0.3)/2},{(sqrt(2)+0.2)/2});
    \draw[->,>={Stealth[scale=1.25]}] (-2,0)--({(-2-0.3)/2},{(0+0.2)/2});
    \draw[->,>={Stealth[scale=1.25]}] (-8/5,-6/5)--({(-8/5-0.3)/2},{(-6/5+0.2)/2});
    \fill[black] (0.5,-0.5) circle (0.06);
    \fill[black] (0.5,0.7) circle (0.06);
    \fill[black] (-0.3,0.2) circle (0.06);
    \draw (0,-2)node[below]{$z_{1}$};
    \draw ({sqrt(3)},-1)node[below right]{$z_{2}$};
    \draw (1,{sqrt(3)})node[above right]{$z_{3}$};
    \draw (-8/5,-6/5)node[below left]{$z_{6}$};
    \draw (-{sqrt(2)},{sqrt(2)})node[above left]{$z_{4}$};
    \draw (-2,0)node[left]{$z_{5}$};
    \draw (-7,-2)node[below]{$v_{1}$};
    \draw ({sqrt(3)-7},-1)node[below right]{$v_{2}$};
    \draw (-6,{sqrt(3)})node[above right]{$v_{3}$};
    \draw (-8/5-7,-6/5)node[below left]{$v_{6}$};
    \draw (-{sqrt(2)}-7,{sqrt(2)})node[above left]{$v_{4}$};
    \draw (-9,0)node[left]{$v_{5}$};
    \draw (-0.3,-0.6)node{$D_{1}$};
    \draw (0.8,-1.3)node{$D_{2}$};
    \draw (1.2,0.3)node{$D_{3}$};
    \draw (-0.25,1.2)node{$D_{4}$};
    \draw (-1.2,0.5)node{$D_{5}$};
    \draw (-1.3,-0.3)node{$D_{6}$};
  \end{tikzpicture}
  \caption{The ribbon tree ($k=6$).}
\end{figure}
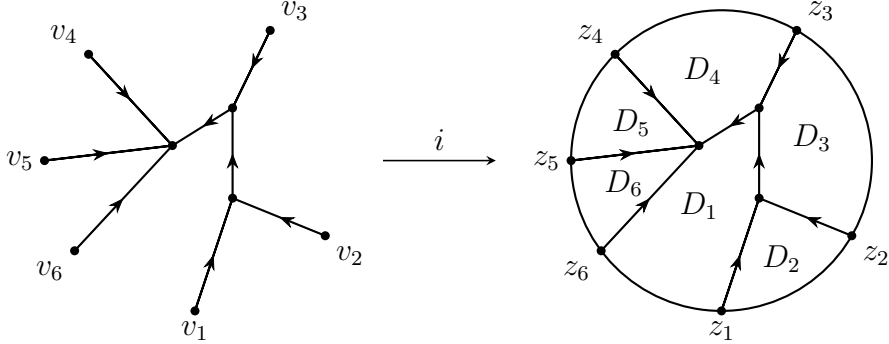
Now we review the moduli space of gradient trees.
\begin{defn}[\cite{FO97}]
  Let $M$ be a Riemannian manifold, and fix a Riemannian metric $g$ of $M$. Let $f_{1}, \cdots , f_{k}$ be $C^{\infty}$-functions on $M$ such that $f_{i+1}-f_{i}$ is a Morse function for each $i=1, \cdots , k$. Here we put $f_{k+1}=f_{1}$. Let $p_{i}$ be one of critical points of $f_{i+1}-f_{i}$. An element of the moduli space $\mathcal{M}_{g}(M, \vec{f}, \vec{p})$ of gradient trees is a pair $((T, i, v_{1}, l), I)$ of $(T, i, v_{1}, l)\in Gr_{k}$ and a map $I: T\rightarrow M$ which satisfies the following conditions:\\
  (1) $I$ is continuous, $I(v_{i})=p_{i}$.\\
  (2) For each exterior edge $e_{i}$, identify $e_{i}\simeq (-\infty, 0]$, we have
  \[
    \dfrac{dI \left\lvert_{e_{i}} \right.}{dt}=-\grad(f_{i+1}-f_{i}).
  \]
  (3) For each interior edge $e$, identify $e\simeq [0, l(e)]$, then
  \[
    \dfrac{dI \left\lvert_{e} \right.}{dt}=-\grad(f_{rig(e)}-f_{lef(e)}).
  \]
\end{defn}
We call a element $((T, i, v_{1}, l), I)$ of the moduli space $\mathcal{M}_{g}(M;\vec{f},\vec{p})$ or the map $I$ the \textit{gradient tree}. The \textit{Morse function} is the function whose critical points are non-degenerate (see \cite{AD14}). Next, we recall the moduli space of pseudoholomorphic disks. This moduli space is defined by the moduli space $\mathfrak{T}_{0,k}$ of disks with $k$ marked points which is defined as follows:
\[
  \mathfrak{T}_{0,k}\coloneqq\left\{(z_{1},\cdots ,z_{k})\in (\partial D^{2})^{k}\left\lvert
      \begin{array}{l}
        z_{i}\neq z_{j}(i\neq j),\\
        z_{1},\cdots ,z_{k}\ \text{respects the cyclic order of}\  \partial D^{2}
      \end{array}
    \right\}\right. \biggl{/}\sim\,.
\]
Here, we use the counterclockwise cyclic ordering for $\partial D^{2}$. Also, we denote $(z_{1},\cdots ,z_{n})\sim (z_{1}',\cdots ,z_{n}')$ if and only if there exists a biholomorphic map $\varphi :D^{2}\rightarrow D^{2}$ such that $\varphi (z_{i})=z_{i}'$. We next prepare definitions of pseudoholomorphic disks.
\begin{defn}[\cite{FO97}]
  Let $(M,g)$ be a Riemannian manifold. Let $\omega$ be the standard symplectic form of the cotangent manifold $T^{*}M$, and let $J$ be the almost complex structure of $T^{*}M$ which is compatible with $\omega$. Let $L_{1},\cdots ,L_{k}$ be Lagrangian submanifolds of $T^{*}M$, and let $x_{i}$ be one of intersection points of $L_{i}$ and $L_{i+1}$. Here, we put $L_{k+1}=L_{1}$. An element of the moduli space $\mathcal{M}_{J}(T^{*}M,\vec{L},\vec{x})$ of pseudoholomorphic disks is a pair $([z_{1},\cdots ,z_{k}],w)$ of elements $[z_{1},\cdots ,z_{k}]\in \mathfrak{T}_{0,k}$ and a smooth map $w: D^{2}\rightarrow T^{*}M$ satisfying the following conditions:\\
  (1) $w(z_{i})=x_{i}$\\
  (2) Denote $\partial_{i}$ the connected component of $\partial D^{2}\setminus\{z_{1},\cdots ,z_{k}\}$ which vertices are $z_{i-1}$ and $z_{i}$. Then, $w(\partial_{i})\subset L_{i}$.\\
  (3) The differential map of $w$ is compatible with $J$, it means $J\circ dw=dw\circ J$.
\end{defn}
We call an element $([z_{1},\cdots ,z_{k}],w)$ of the moduli space $\mathcal{M}_{J}(T^{*}M;\vec{L},\vec{x})$ itself a \textit{pseudoholomorphic disk}. We denote the exact Lagrangian submanifold $L_{i}$ as the graph of $df_{i}$, i.e.,
\[
  L_{i}=\operatorname{graph}(df)\coloneqq\left\{(p,df_{i}(p)):p\in M\right\}.
\]
Here, we denote $L_{i}^{\epsilon}\coloneqq\operatorname{graph}(\epsilon df_{i})$, and denote $x_{i}^{\epsilon}$ one of points of $L_{i}^{\epsilon}\cap L_{i+1}^{\epsilon}$. There exists a correspondence between two moduli spaces. 
\begin{thm}[\cite{FO97}]
  Let $\pi$ be the cotangent bundle $\pi:T^{*}M\rightarrow M$. We set $p_{i}=\pi(x_{i}^{\epsilon})$. Let $J=J_{g}$ be the canonical almost complex structure on $T^{*}M$ associated to the metric $g$ on $M$. For each generic $\vec{f}=(f_{i})$ and for sufficiently small $\epsilon$, we have an oriented diffeomorphism $\mathcal{M}_{g}(M:\vec{f},\vec{p})\simeq \mathcal{M}_{J}(T^{*}M:\vec{L^{\epsilon}},\vec{x^{\epsilon}})$.
\end{thm}
In this paper, we will construct gradient trees and pseudoholomorphic disks in the case of $M=\mathbb{R}$ and that Lagrangian sections are affine. Then, we discuss a correspondence between them. It is known that the unit disk is conformal to the upper half plane. Let $[z_{1},z_{2},\dots,z_{n}]\in\mathfrak{T}_{0,n}$ be a unit disk with $n$ points. Let $\varphi$ be the conformal map from the unit disk to the upper half plane such that $\varphi(z_{1})=1,\varphi(z_{2})=\infty,\varphi(z_{3})=0$. When $n\geq 4$, we set $\xi_{1},\xi_{2},\dots,\xi_{n-3}\in\mathbb{R}$ such that $\varphi(z_{i+3})=\xi_{i}$ for $i=1,2,\dots,n-3$. We thus identify pseudoholomorphic disks $w$ from a unit disk with $n$ marked points $z_{1},z_{2},\dots,z_{n}$ with the holomorphic map $w\circ\varphi^{-1}$ from the upper half plane with marked points $1,\infty, 0, \xi_{1},\xi_{2},\dots,\xi_{n-3}$ in this paper. We hereafter call such map $w$ as a holomorphic disk instead of a pseudoholomorphic disk since pseudoholomorphic disks we treat in this paper are holomorphic.

\subsection{Schwarz-Christoffel mapping, hypergeometric function and their properties.}
Since Lagrangian sections we treat are affine, we study about holomorphic disks which map the upper half plane to bounded polygonal domains. 
\begin{thm}[\cite{DT02},\cite{Ahl-ca},\cite{Neh75}]
  \label{SCmap}
  Let $P$ be the interior of a polygon $\Gamma$ having vertices $w_{1},w_{2},...,w_{n}$ and interior angles $\alpha_{1}\pi, \alpha_{2}\pi,..., \alpha_{n}\pi$ in counterclockwise order. Let $f$ be a conformal map from the upper half plane $H^{+}$ to $P$ with $f(\infty)=w_{n}$. Then we obtain 
  \[
    f(z)=A+B\int_{}^{z}\prod_{k=1}^{n-1}(\zeta-z_{k})^{\alpha_{k}-1}d\zeta
  \]
  for some complex constants $A$ and $C$, where $w_{k}=f(z_{k})$ for $k=1,...,n-1$. We here take branches of $(\zeta-z_{i})^{\alpha_{i}}$ at each $i=1,2,\dots,n-1$ such that 
  \[
    (\zeta-z_{i})^{\alpha_{i}}=
    \begin{cases*}
      (\zeta-z_{i})^{\alpha_{i}}\,(\zeta\in \mathbb{R},\zeta>z_{i})\\
      \left\lvert \zeta-z_{i}\right\rvert^{\alpha_{i}}e^{\pi\alpha_{i}i}\,(\zeta\in \mathbb{R},\zeta<z_{i})
    \end{cases*}.
  \]
\end{thm}
This map $f$ is called the \textit{Schwarz-Christoffel map} from the upper half plane. As we mention at subsection 3.4 and 4.4, we can locally describe a Schwarz-Christoffel map $f$ by power series. Now, let us study integral representations and connection formulas of the hypergeometric series (and functions). We use the Gauss hypergeometric function $_{2}F_{1}$ in the case $k=3$, and we use Appell's $F_{1}$ function and Horn's $G_{2}$ function in the case $k=4$. The Gauss hypergeometric function $_{2}F_{1}(a,b,c;x)$ is defined as follows.
\begin{defn}[\cite{AAR99}]
  The \textit{Gauss hypergeometric function} $_{2}F_{1}(a,b,c;x)$ is defined by the series
  \[
    \sum_{n=0}^{\infty}\dfrac{(a)_{n}(b)_{n}}{(c)_{n}n!}x^{n}
  \]
  for $\left\lvert x\right\rvert<1$, and by continuation elsewhere. Here $a,b,c$ are complex numbers, and $c$ is not a negative integer or zero. 
\end{defn}
This hypergeometric function $_{2}F_{1}$ has the following integral representation.
\begin{prop}[\cite{AAR99}]
  \label{2F1}
  If $\operatorname{Re} c>\operatorname{Re} b>0$, then
  \[
    _{2}F_{1}(a,b,c;x)=\dfrac{\Gamma(c)}{\Gamma(b)\Gamma(c-b)}\int_{0}^{1}t^{b-1}(1-t)^{c-b-1}(1-xt)^{-a}\,dt
  \]
  in the $x$ plane cut along the real axis from $1$ to $\infty$. Here it is understood that $\arg t=\arg(1-t)=0$ and $(1-xt)^{-a}$ has its principal value.
\end{prop}
Appell's hypergeometric function $F_{1}$ is defined as follows. 
\begin{defn}[\cite{Mim22},\cite{Erd50}]
  \textit{Appell's hypergeometric function} $F_{1}$ is the analytic continuation of Appell's hypergeometric series
  \[
    F_{1}(a,b_{1},b_{2},c;x,y)=\sum_{m,n\geq 0}\dfrac{(a)_{m+n}(b_{1})_{m}(b_{2})_{n}}{(c)_{m+n}m!n!}x^{m}y^{n}, \left\lvert x\right\rvert<1, \left\lvert y\right\rvert<1 
  \]
  where $a,b_{1},b_{2},c$ are complex numbers, and $c$ is not a negative integer or zero. 
\end{defn}
This hypergeometric function $F_{1}$ also has the following integral representation.
\begin{prop}[\cite{Mim22},\cite{Erd50}]
  \label{AppellF1}
  Let $a,b_{1},b_{2},c$ be the complex number such that $0<\operatorname{Re}a<\operatorname{Re}c$. Then, we obtain 
  \[
    F_{1}(a,b_{1},b_{2},c;x,y)=\dfrac{\Gamma(c)}{\Gamma(a)\Gamma(c-a)}\int_{0}^{1}t^{a-1}(1-t)^{c-a-1}(1-xt)^{-b_{1}}(1-yt)^{-b_{2}}\,dt
  \]
  when $\left\lvert x\right\rvert<1,\left\lvert y\right\rvert<1$.
\end{prop}
We need other power series in the case $(x,y)$ does not satisfy $\left\lvert x\right\rvert<1,\left\lvert y\right\rvert<1$. This can be realized by connection formulas of $F_{1}$ and $G_{2}$. We first review Horn's hypergeometric function $G_{2}$.
\begin{defn}[\cite{Mim22},\cite{Ols64}]
  Let $\alpha,\beta,\gamma,\delta$ be complex numbers. \textit{Horn's hypergeometric function} $G_{2}$ is the analytic continuation of following hypergeometric series
  \[
    G_{2}(\alpha,\beta,\gamma,\delta;x,y)\coloneqq\sum_{m,n\geq 0}(\alpha)_{m}(\beta)_{n}(\gamma)_{n-m}(\delta)_{m-n}\dfrac{x^{m}}{m!}\dfrac{y^{n}}{n!}, \left\lvert x\right\rvert<1, \left\lvert y\right\rvert<1,
  \]
  where $(a)_{m-n}=\Gamma(a+m-n)/\Gamma(a)$.
\end{defn}
In this subsection, we recall one of the connection formulas.
\begin{lem}[\cite{Mim22},\cite{Ols64}]
  \label{F1conn1}
  If $\gamma,\beta-\alpha,\beta-\gamma\notin\mathbb{Z}$, then the following holds:
  \begin{align*}
    &F_{1}(\alpha,\beta',\beta,\gamma;y,x)\\
    &=\dfrac{\Gamma(\beta-\alpha)\Gamma(\gamma)}{\Gamma(\beta)\Gamma(\gamma-\alpha)}(-x)^{-\alpha}F_{1}\left(\alpha,1+\alpha-\gamma,\beta',1+\alpha-\beta;\dfrac{1}{x},\dfrac{y}{x}\right)\\
    &+\dfrac{\Gamma(\alpha-\beta)\Gamma(\gamma)}{\Gamma(\alpha)\Gamma(\gamma-\beta)}(-x)^{-\beta}G_{2}\left(\beta,\beta',\alpha-\beta,1+\beta-\gamma;-\dfrac{1}{x},-y\right).
  \end{align*}
  Here the arguments of $-x$ and $-y$ of the factors $(-x)^{*}$ and $(-y)^{*}$ are assigned to be zero on the real region $\infty<x<y<0$ (see \cite{Mim22}).
\end{lem}
We use this lemma in subsection 4.4.
\subsection{Conformal Modulus of Quadrilaterals.}
In this subsection, we review a conformal invariant, called the conformal modulus, for general quadrilaterals. We also discuss some monotonicity properties of the invariant. 
\begin{defn}[\cite{Hen86}]
  A \textit{quadrilateral} is a system $(Q;a_{1}, a_{2}, a_{3}, a_{4})$, where $Q$ is a Jordan region and where the $a_{i}$ are four distinct points on the boundary of $Q$, arranged in the sense of increasing parameter values.
  \begin{figure}[h]
    \centering
    \begin{tikzpicture}
      \draw [thick] (6,3) to [out=-80, in=90] (8,2);
      \draw [thick] (8,2) to [out=-90, in=0] (5,1);
      \draw [thick] (5,1) to [out=180, in=-90] (3,2);
      \draw [thick] (3,2) to [out=90, in=-120] (4,5);
      \draw [thick] (4,5) to [out=60, in=100] (6,3);
      \fill[black] (5,1) circle (0.06) node[below]{$a_{1}$};
      \fill[black] (8,2) circle (0.06) node[right]{$a_{2}$};
      \fill[black] (6,3) circle (0.06) node[above right]{$a_{3}$};
      \fill[black] (4,5) circle (0.06) node[above left]{$a_{4}$};
      \draw (4.5,3) node{$Q$};
    \end{tikzpicture}
    \caption{quadrilateral}
  \end{figure}
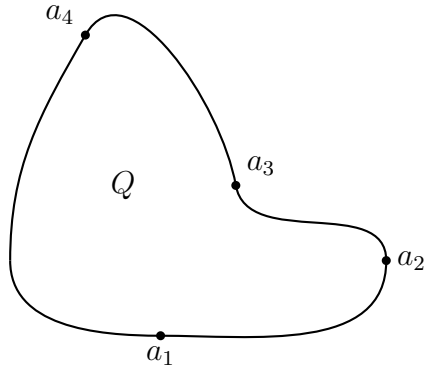
\end{defn}
The word ``quadrilateral'' usually means a polygon which is bounded by four straight lines. Since we consider a quadrilateral $(Q;a_{1}, a_{2}, a_{3}, a_{4})$ as a polygon in Section 4, we use the term ``quadrilateral'' in the sense above only in this subsection. The conformal modulus of the quadrilateral is defined as follows.
\begin{defn}[\cite{Hen86}]
  Let $(Q;a,b,c,d)$ be a quadrilateral. Let $w=f(z)$ be a one-to-one conformal map from the domain $Q$ onto the rectangle $0<u<1,0<v<M\,(w=u+iv)$ which maps $a,b,c,d$ to $0,1,1+iM,iM$, respectively. The number $M$ is called the \textit{conformal modulus} of the quadrilateral $(Q;a,b,c,d)$, and we denote it by $M(Q;a,b,c,d)$.
\end{defn}
Note that $M(Q;a,b,c,d)=1/M(Q;b,c,d,a)$. There is a popular inequality of the conformal modulus which contains geometric quantities.
\begin{lem}[cf. \cite{LV73}]
  \label{rengel}
  Let $(Q;z_{1},z_{2},z_{3},z_{4})$ be a quadrilateral, and $m(Q)$ be the Euclidean area of $Q$. Let $l(C)$ be the Euclidean length of a curve $C$. Let $C_{a}$ be the set of open Jordan arcs which lies in the interior of $Q$ and has one point on each arcs $z_{1}z_{2}$ and $z_{3}z_{4}$, and let $C_{b}$ be the set of open Jordan curves which lies in the interior of $Q$ and has one point on each arcs $z_{2}z_{3}$ and $z_{4}z_{1}$, respectively. We set $s_{a}(Q)$ and $s_{b}(Q)$ as follows:
  \[
    s_{a}(Q)\coloneqq \inf_{C\in C_{a}}l(C),\ s_{b}(Q)\coloneqq \inf_{C\in C_{b}}l(C).
  \] 
  Then, the module of a quadrilateral $Q$ satisfies the double inequality
  \[
    \dfrac{(s_{a}(Q))^{2}}{m(Q)}\leq M(Q)\leq \dfrac{m(Q)}{(s_{b}(Q))^{2}}.
  \]
  Equality holds in both cases if and only if when $Q$ is a rectangle.
\end{lem}
This inequation above is called \textit{Rengel's inequation}. We use this inequality to evaluate conformal moduli of polygonal quadrilaterals.
\section{The case of $M=\mathbb{R},k=3$.}
\subsection{Main result.}
We first define a neighborhood $D_{p}(\delta)$ of $p\in\mathbb{R}\cup\{\infty\}$ and the transformation $\phi_{p,\delta}$ from a stripe $\Theta_{ext}\coloneqq\{(\tau,\sigma)\in(-\infty,0)\times[0,1]\}$ to $D_{p}(\delta)\setminus\{p\}$ as follows.
\begin{align*}
  &D_{p}(\delta)\coloneqq
  \begin{cases*}
    \{\left\lvert z-p\right\rvert<\delta\mid\text{Im}z\geq0\}\,(p\in\mathbb{R})\\
    \{\left\lvert z\right\rvert^{-1}<\delta\mid\text{Im}z\geq0\}\,(p=\infty)
  \end{cases*}
  \,(\delta>0)\\
  &\phi_{p,\delta}(\tau,\sigma)\coloneqq
  \begin{cases*}
    p+\delta\exp[\pi(\tau+i\sigma)]\,(p\neq\infty)\\
    -\delta^{-1}\exp[-\pi(\tau+i\sigma)]\,(p=\infty)
  \end{cases*}
\end{align*}
We use these domains and transformations in this section and Section $4$. 
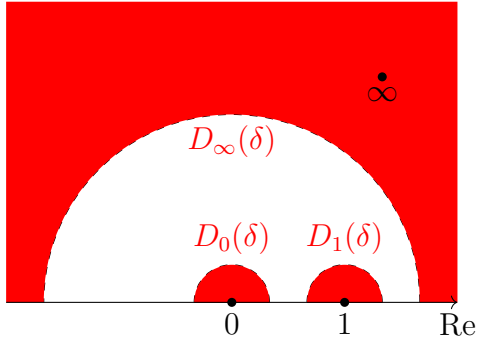
\begin{figure}[bt]
  \centering
  \begin{tikzpicture}
      \clip (-1.5,-0.5) rectangle (5.5,4);
      \draw[dashed] (2.5,0) arc (0:180:0.5);
      \draw[dashed] (4,0) arc (0:180:0.5);
      \draw[dashed] (4.5,0) arc (0:180:2.5);
      \begin{scope}
        \clip (-1,0) rectangle (5,4);
        \fill[red,opacity=0.3] (2,0) circle [radius=0.5];
        \fill[red,opacity=0.3] (3.5,0) circle [radius=0.5];
      \end{scope}
      \fill[red,opacity=0.3] (5,4)--(5,0)--(4.5,0) arc (0:180:2.5)--(-0.5,0)--(-1,0)--(-1,4)--cycle;
      \draw[->] (-1,0)--(5,0)node[below]{Re};
      \fill[black] (2,0) circle (0.06) node[below]{$0$};
      \fill[black] (3.5,0) circle (0.06) node[below]{$1$};
      \fill[black] (4,3) circle (0.06) node[below]{$\infty$};
      \draw[red] (2,0.5) node[above]{$D_{0}(\delta)$};
      \draw[red] (3.5,0.5) node[above]{$D_{1}(\delta)$};
      \draw[red] (2,2.5) node[below]{$D_{\infty}(\delta)$};
  \end{tikzpicture}
  \caption{The figure of the upper half plane in the case $k=3$.}
  \label{3markedpts}
\end{figure}
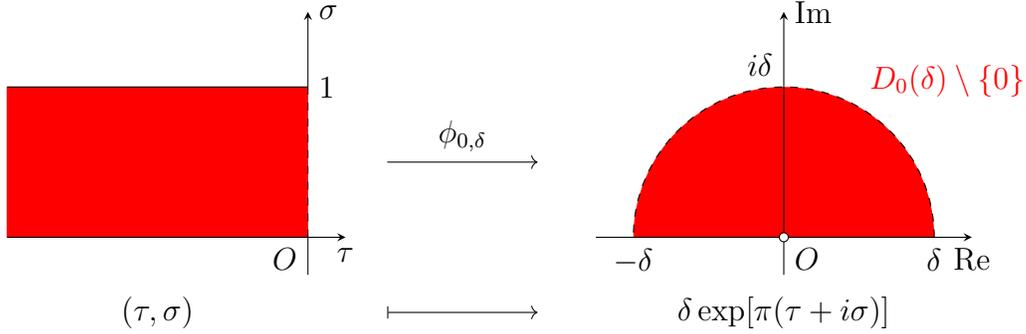
\begin{figure}[bt]
  \centering
  \begin{tikzpicture}
    \draw (0,0) node[below left]{$O$};
    \draw (0,2) node[right]{$1$};
    \fill[red,opacity=0.3] (-4,0)--(0,0)--(0,2)--(-4,2);
    \draw (0,2)--(-4,2);
    \draw[red](-2,1)node{$\Theta_{ext}$};
    \draw[->,,>=stealth] (-4,0)--(0.5,0)node[below]{$\tau$};
    \draw[->,>=stealth] (0,2)--(0,3)node[right]{$\sigma$};
    \draw[dashed] (0,0)--(0,2);
    \draw (0,-0.5)--(0,0);
    \draw (-2,-1)node{$(\tau,\sigma)$};
    \begin{scope}[xshift=30]
      \draw[->] (0,1)--(2,1);
      \draw (1,1)node[above]{$\phi_{0,\delta}$};
      \draw[|->] (0,-1)--(2,-1);
    \end{scope}
    \begin{scope}[xshift=180]
      \fill[red,opacity=0.3] (2,0) arc (0:180:2) --(-2,0)--cycle;
      \draw[dashed] (2,0) arc (0:180:2) --(-2,0);
      \draw[->,>=stealth] (-2.5,0)--(2.5,0)node[below]{Re};
      \draw[->,>=stealth] (0,-0.5)--(0,3)node[right]{Im};
      \draw (0,2) node[above left]{$i\delta$};
      \draw (2,0) node[below]{$\delta$};
      \draw (-2,0) node[below]{$-\delta$};
      \draw (0,0) node[below right]{$O$};
      \draw (0,-1) node{$\delta\exp[\pi(\tau+i\sigma)]$};
      \draw[red] (60:2)node[above right]{$D_{0}(\delta)\setminus\{0\}$};
      \fill[white] (0,0)circle(0.06);
      \draw (0,0) circle (0.06);
    \end{scope}
  \end{tikzpicture}
  \caption{The transformation $\phi_{0,\delta}$.}
  \label{k3trans}
\end{figure}
In the case $k=3$, we divide $\overline{\mathbb{H}}$ into four regions $D_{0}(\delta),D_{1}(\delta),D_{\infty}(\delta),D(\delta)$ (see Figure \ref{3markedpts}). Here, we define $D(\delta)$ by
\[
  D(\delta)\coloneqq\overline{\mathbb{H}}\setminus\bigcup_{i=1}^{3} D_{z_{i}}(\delta).
\]
We then discuss the convergence of holomorphic disks in each regions. The following is the main statement in the case $k=3$. We hereafter regard $\mathbb{R}$ as the zero section of $T^{*}\mathbb{R}$, namely, we identify $p\in\mathbb{R}$ with $(p,0)\in T^{*}\mathbb{R}$.
\begin{thm}
  \label{k3corr}
  Let $L_{1},L_{2},L_{3}$ be affine Lagrangian sections in $T^{*}\mathbb{R}$ which satisfy $L_{i}\cap L_{j}\neq\emptyset,L_{i}\neq L_{i+1}\,(i\neq j)$. Let $f_{1},f_{2},f_{3}$ be functions on $\mathbb{R}$ such that $L_{i}=\operatorname{graph}(df_{i})$ for $i=1,2,3$. Let $x_{i}$ be the intersection point of $L_{i}$ and $L_{i+1}$. We assume that there exists a triangle which has distinct vertices $x_{1},x_{2},x_{3}$ in counterclockwise order. Let $I\in\mathcal{M}_{g}(\mathbb{R};\vec{f},\vec{p})$ be the gradient tree, and $I_{i}$ the restriction of $I$ to the external edge which has the external vertex $v_{i}$ such that $I(v_{i})=p_{i}$. We define the Lagrangian section $L_{i}^{\epsilon}$ as $L_{i}^{\epsilon}\coloneqq\operatorname{graph}(\epsilon df_{i})$, and the point $x_{i}^{\epsilon}$ as $x_{i}^{\epsilon}\in L_{i}^{\epsilon}\cap L_{i+1}^{\epsilon}$ for $i=1,2,3$. Let $w_{\epsilon}\in\mathcal{M}_{J}(T^{*}\mathbb{R};\vec{L^{\epsilon}},\vec{x^{\epsilon}})$ be the Schwarz-Christoffel map from the upper half plane to the bounded triangle $x_{1}^{\epsilon}x_{2}^{\epsilon}x_{3}^{\epsilon}$ such that $w_{\epsilon}(z_{i})=x_{i}^{\epsilon}$ for $i=1,2,3$. Here, we set $z_{1}=1,z_{2}=\infty,z_{3}=0$. If we take a real number $\delta>0$ such that $D_{z_{i}}(\delta)\cup D_{z_{j}}(\delta)=\emptyset$ for $i\neq j$, then we obtain 
  \begin{gather}
    \lim_{\epsilon\to+0}\sup_{z\in D(\delta)}\left\lvert w_{\epsilon}(z)-p_{0}\right\rvert=0, \label{eq:k3p0}\\
    \lim_{\epsilon\to+0}\sup_{(\tau,\sigma)\in\Theta_{ext}}\left\lvert w_{\epsilon}\circ\phi_{z_{i},\delta}(\tau,\sigma)-I_{i}(\epsilon\tau)\right\rvert=0 \,(i=1,2,3,I_{i}(0)=p_{0}), \label{eq:k3ext}
  \end{gather}
  where $p_{0}$ is the critical point of $f_{2}-f_{1},f_{3}-f_{2}$ or $f_{1}-f_{3}$ such that the Morse index of $p_{0}$ is zero.
\end{thm}
The \textit{Morse index} of a critical point $p$ of $f$ is defined by the number of negative eigenvalues of the Hessian matrix of $f$ (see \cite{AD14} et al). Note that the above $p_{0}$ is determined uniquely as we see in Lemma \ref{tri-dom} and Lemma \ref{k3gradtree}.
\subsection{Holomorphic disks and Gradient trees.}
We first discuss when three affine Lagrangian sections $L_{1},L_{2},L_{3}$ enable us to obtain the Schwarz-Christoffel map which maps the upper half plane to the triangular domain bounded by $L_{1},L_{2},L_{3}$.
\begin{lem}
  \label{tri-dom}
  Let $L_{1},L_{2},L_{3}$ be affine Lagrangian sections of $T^{*}\mathbb{R}$ which satisfy $L_{i}\cap L_{j}$ and $L_{i}\neq L_{j}$ for $i\neq j$. By identifying $T^{*}\mathbb{R}$ with $\mathbb{R}^{2}$ and writing $(x,y)$ as the coordinate of $\mathbb{R}^{2}$, there exist $a_{i},b_{i}$ such that $L_{i}=\{(x,y)\mid y=a_{i}x+b_{i}\}$ for each $i=1,2,3$. Let $x_{i}=(p_{i},q_{i})$ be the intersection point of $L_{i}$ and $L_{i+1}$. Then $x_{1},x_{2},x_{3}$ forms a triangle if and only if $p_{i}\neq p_{i+1}$ for $i=1,2,3$. Furthermore, the vertices $x_{1},x_{2},x_{3}$ of the triangle $x_{1}x_{2}x_{3}$ is in counterclockwise order if and only if $a_{1},a_{2},a_{3}$ satisfy one of the following inequations:
  \begin{equation}
    \label{k3tri}
    \begin{cases}
      a_{1}<a_{3}<a_{2},\\
      a_{2}<a_{1}<a_{3},\\
      a_{3}<a_{2}<a_{1}.
     \end{cases}
  \end{equation}
\end{lem}
The proof is shown by elementary calculations. By the definition of the ribbon tree, the ribbon tree $(T,i)$ is the element of $Gr_{3}$ if and only if $T$ is isometric to the following tree in Figure \ref{k3tree2}.
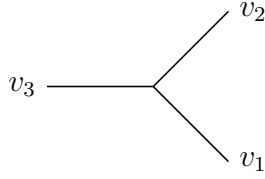
\begin{figure}[tb]
  \begin{tikzpicture}
    \draw[semithick] (1,1)--(0,0);
    \draw[semithick] ({-sqrt(2)},0)--(0,0);
    \draw[semithick] (1,-1)--(0,0);
    \draw (1,-1) node[right]{$v_{1}$};
    \draw (1,1) node[right]{$v_{2}$};
    \draw ({-sqrt(2)},0) node[left]{$v_{3}$};
  \end{tikzpicture} 
  \caption{The tree $T$ of the ribbon tree $(T,i)\in Gr_{3}$.}
  \label{k3tree2}
\end{figure}
We next show the existence and uniqueness of gradient trees.
\begin{lem}
  \label{k3gradtree}
  Let $L_{1},L_{2},L_{3}$ be affine Lagrangian sections of $T^{*}\mathbb{R}$ which satisfy one of conditions in \eqref{k3tri}. Let $f_{1},f_{2},f_{3}$ be functions defined on $\mathbb{R}$ such that $L_{i}=\operatorname{graph}(df_{i})$ for $i=1,2,3$. Then the moduli space $\mathcal{M}_{g}(\mathbb{R};\vec{f},\vec{p})$ of gradient trees is a one-point set.
\end{lem}
\begin{proof}
  Since $L_{1},L_{2},L_{3}$ are affine, $f_{1},f_{2},f_{3}$ can be written as
  \[
    f_{i}(x)=\dfrac{1}{2}a_{i}x^{2}+b_{i}x+c_{i}\,(a_{i},b_{i},c_{i}\in\mathbb{R}).
  \]
  Here, $x$ is the coordinate of $\mathbb{R}$. Since $a_{i}\neq a_{i+1}$, $f_{i+1}-f_{i}$ is a Morse function and has a unique critical point $p_{i}$. We discuss the case $a_{2}<a_{1}<a_{3}$ only since other cases can be proved by the same way. Let $I\in\mathcal{M}_{g}(\mathbb{R};\vec{f},\vec{p})$ be a gradient tree, and $I_{i}$ the restriction of $I$ to the external edge $e_{i}$. This $I_{i}$ satisfies the following differential equation:
  \[
    \dfrac{dI_{i}}{dt}=-\grad(f_{i+1}-f_{i}),\lim_{t\rightarrow-\infty}I_{i}(t)=p_{i}.
  \]
  Thus, we obtain $I_{i}(t)=A_{i}e^{-(a_{i+1}-a_{i})t}+p_{i}$. Here, $A_{i}$ is the real number which is determined later. By $a_{3}-a_{2}$ and $I_{2}(t)\rightarrow p_{2}\,(t\rightarrow-\infty)$, we obtain $A_{2}=0$. Then, we have $I_{1}(0)=I_{3}(0)=p_{2}$ to ensure the continuity of gradient trees. This follows $A_{2}=p_{2}-p_{1}$ and $A_{3}=p_{2}-p_{3}$. In conclusion, we obtain the unique gradient tree in the case $a_{2}<a_{1}<a_{3}$.
\end{proof}
We use the following lemma to calculate a interior angle at each vertex in the triangle $x_{1}x_{2}x_{3}$. We hereafter sometimes consider a point on the Euclidean space $\mathbb{R}$ as a real number in this paper. The following lemma can be proved by elementary calculations.
\begin{lem}
  \label{critk3}
  We assume the setting of Lemma \ref{k3gradtree}. Let $p_{i}$ be the critical point of $f_{i+1}-f_{i}$. For $i=1,2,3$, if we have $a_{i+2}<a_{i+1}<a_{i}$, then we obtain $p_{i+1}<p_{i+2}<p_{i}$ or $p_{i}<p_{i+2}<p_{i+1}$, where we identify $p_{k+3}=p_{k}$ and $a_{k+3}=a_{k}$.
\end{lem}
\subsection{Power series representations of holomorphic disks.}
In this subsection, we express the Schwarz-Christoffel map from the upper half plane to interior of the triangle in power series. We apply the following lemma to the holomorphic disk $w_{\epsilon}$ in Theorem \ref{k3corr}.
\begin{lem}\label{trisc}
  Let $P$ be the interior of a triangle having vertices $x_{1}, x_{2}, x_{3}$ in counterclockwise order, and let $\alpha_{1}\pi, \alpha_{2}\pi, \alpha_{3}\pi$ be the corresponding interior angles. Let $w$ be the Schwarz-Christoffel map from the upper half plane to $P$ such that $w(0)=x_{1}, w(1)=x_{2}, w(\infty)=x_{3}$. Then, $w$ can be described by the following power series:\\
  $(\rm{i})$ on $\left\lvert z\right\rvert<1$
  \[
    w(z)=x_{1}+\dfrac{\Gamma(\alpha_{1}+\alpha_{2})}{\Gamma(\alpha_{1}+1)\Gamma(\alpha_{2})}\cdot (x_{2}-x_{1})\cdot z^{\alpha_{1}}\,_{2}F_{1}(\alpha_{1},1-\alpha_{2},\alpha_{1}+1;z),
  \]
  $(\rm{ii})$ on $\left\lvert 1-z\right\rvert<1$
  \[
    w(z)=x_{2}+\dfrac{\Gamma(\alpha_{1}+\alpha_{2})}{\Gamma(\alpha_{2}+1)\Gamma(\alpha_{1})}\cdot (x_{1}-x_{2})\cdot (1-z)^{\alpha_{2}}\,_{2}F_{1}(\alpha_{2},1-\alpha_{1},\alpha_{2}+1;1-z),
  \]
  $(\rm{iii})$ on $\left\lvert z^{-1}\right\rvert<1$
  \begin{align*}
    &w(z)=x_{3}+\dfrac{\Gamma(1-\alpha_{1})}{\Gamma(2-\alpha_{1}-\alpha_{2})\Gamma(\alpha_{2})}\cdot (x_{2}-x_{3})\cdot z^{\alpha_{1}+\alpha_{2}-1}\\
    &\hspace{6cm}\cdot\,_{2}F_{1}(1-\alpha_{1}-\alpha_{2},1-\alpha_{2},2-\alpha_{1}-\alpha_{2};z^{-1}).
  \end{align*}
\end{lem}
\begin{proof}
  By using Proposition \ref{2F1}, when we take one of the singular points of the integrand as the starting point of the contour, we can obtain power series representations of $w$ in the neighborhood of each of singular points respectively. At first, we set $0$ to the starting point of the contour. By Proposition \ref{2F1}, for $\left\lvert z\right\rvert<1$, we get
  \begin{align*}
    \int_{0}^{z}\zeta^{\alpha_{1}-1}(\zeta-1)^{\alpha_{2}-1}\,d\zeta&=e^{\pi(\alpha_{2}-1)i}\int_{0}^{1}(zt)^{\alpha_{1}-1}(1-zt)^{\alpha_{2}-1}z\,dt\\
    &=e^{\pi(\alpha_{2}-1)i}z^{\alpha_{1}}\,_{2}F_{1}(\alpha_{1},1-\alpha_{2},\alpha_{1}+1;z)B(\alpha_{1},1)\,.
  \end{align*}
  If we bring $z$ close to $1$ within the domain $\left\{z\in\overline{\mathbb{H}}\mid\left\lvert z\right\rvert<1\right\}$, then we obtain
  \begin{align*}
    \lim_{z\rightarrow 1}\int_{0}^{z}\zeta^{\alpha_{1}-1}(\zeta-1)^{\alpha_{2}-1}\,d\zeta&=e^{\pi(\alpha_{2}-1)i}\int_{0}^{1}\zeta^{\alpha_{1}-1}(1-\zeta)^{\alpha_{2}-1}\,d\zeta\\
    &=e^{\pi(\alpha_{2}-1)i}B(\alpha_{1},\alpha_{2})\,.
  \end{align*}
  Since we have $w(0)=x_{1}$ and $w(1)=x_{2}$, $w(z)$ has the following power series representation in $\left\lvert z\right\rvert<1$:
  \begin{align*}
    w(z)&=x_{1}+\dfrac{e^{\pi(\alpha_{2}-1)i}z^{\alpha_{1}}\,_{2}F_{1}(\alpha_{1},1-\alpha_{2},\alpha_{1}+1;z)B(\alpha_{1},1)}{e^{\pi(\alpha_{2}-1)i}B(\alpha_{1},\alpha_{2})}\cdot (x_{2}-x_{1})\\
    &=x_{1}+\dfrac{\Gamma(\alpha_{1}+\alpha_{2})}{\Gamma(\alpha_{1}+1)\Gamma(\alpha_{2})}\cdot (x_{2}-x_{1})\cdot z^{\alpha_{1}}\,_{2}F_{1}(\alpha_{1},1-\alpha_{2},\alpha_{1}+1;z).
  \end{align*}
  If we put $1$ into the starting point of contour and apply a similar argument as above, then we also obtain the representation of $w(z)$ in $\left\lvert 1-z\right\rvert<1$. When we set $t=1/\zeta$ in the integral representation of the Schwarz-Christoffel map and set $0$ to the starting point of the contour, we also obtain the representation of $w(z)$ in $\left\lvert z\right\rvert^{-1}<1$.
\end{proof}
\subsection{Interior angles in convex polygons.}
In order to discuss the convergence of holomorphic disks $w_{\epsilon}$, we prepare the explicit formula of the angles $\pi\alpha_{i}^{\epsilon}$ at $x_{i}^{\epsilon}$ and their properties. The arguments in this subsection are also elementary.
\begin{lem}
  \label{angles}
  Let $k$ be a positive integer. Let $L_{i}^{\epsilon}=\{(x,y)\in\mathbb{R}^{2}\,|\,y=(a_{i}x+b_{i})\epsilon\}$ be straight lines in $\mathbb{R}^{2}$ for $i=1,2,\dots,k$. We assume that $L_{i}^{\epsilon}\cap L_{i+1}^{\epsilon}\neq \emptyset,L_{i}^{\epsilon}\neq L_{i+1}^{\epsilon}, L_{i}^{\epsilon}\cap L_{i+1}\cap L_{i+2}^{\epsilon}=\emptyset$ for $i=1,2,\dots,k$. We denote by $x_{i}^{\epsilon}=(p_{i},q_{i})$ the intersection point of $L_{i}^{\epsilon}$ and $L_{i+1}^{\epsilon}$. Let $P_{\epsilon}$ be the convex polygon having vertices $x_{1}^{\epsilon},x_{2}^{\epsilon},\dots,x_{k}^{\epsilon}$ in counterclockwise order, and let $\pi\alpha_{i}^{\epsilon}$ be the interior angle at $x_{i}^{\epsilon}$. Then we obtain
  \begin{gather*}
    \pi\alpha_{i}^{\epsilon}=\left\lvert \arctan a_{i+1}\epsilon-\arctan a_{i}\epsilon\right\rvert\ \ \ \text{if $(p_{i+1}-p_{i})(p_{i-1}-p_{i})>0$,}\\
    \pi\alpha_{i}^{\epsilon}=\pi-\left\lvert \arctan a_{i+1}\epsilon-\arctan a_{i}\epsilon\right\rvert\ \ \ \text{if $(p_{i+1}-p_{i})(p_{i-1}-p_{i})<0$.}
  \end{gather*}
\end{lem}
\begin{cor}
  \label{angleslim}
  We assume the setting in Lemma \ref{angles}. Then we obtain
  \[
    \lim_{\epsilon\to+0}\alpha_{i}^{\epsilon}=0, \lim_{\epsilon\to+0}\dfrac{\pi\alpha_{i}^{\epsilon}}{\epsilon\left\lvert a_{i+1}-a_{i}\right\rvert}=1,\lim_{\epsilon\to+0}\epsilon\Gamma(\alpha_{i}^{\epsilon})=\dfrac{\pi}{\left\lvert a_{i+1}-a_{i}\right\rvert}
  \]
  in the case $(p_{i+1}-p_{i})(p_{i-1}-p_{i})>0$, and 
  \[
    \lim_{\epsilon\to+0}\alpha_{i}^{\epsilon}=1, \lim_{\epsilon\to+0}\dfrac{\pi(1-\alpha_{i}^{\epsilon})}{\epsilon\left\lvert a_{i+1}-a_{i}\right\rvert}=1,\lim_{\epsilon\to+0}\epsilon\Gamma(1-\alpha_{i}^{\epsilon})=\dfrac{\pi}{\left\lvert a_{i+1}-a_{i}\right\rvert}
  \]
  in the case $(p_{i+1}-p_{i})(p_{i-1}-p_{i})<0$.
\end{cor}
Since $\arctan x-x$ is monotonically decreasing, we obtain the following lemma.
\begin{lem}
  \label{anglesineq}
  If we assume the setting in Lemma \ref{angles}, then we obtain $\pi\alpha_{i}^{\epsilon}-\left\lvert (a_{i+1}-a_{i})\epsilon\right\rvert<0$ in the case $(p_{i+1}-p_{i})(p_{i-1}-p_{i})>0$.
\end{lem}
We use Lemma \ref{anglesineq} to show the uniform convergence of holomorphic disks to gradient trees at the external edges of a tree.
\subsection{Proof of main theorem.}
  We first prove the existence of $\delta>0$. The conditions $D_{z_{1}}(\delta)\cap D_{z_{2}}(\delta)=\emptyset$ holds if and only if $\delta<1-\delta$. Then, we have $\delta<1/2$. As a similar way, we obtain $\delta<-1/2+\sqrt{5}/2$ from $D_{z_{2}}(\delta)\cap D_{z_{3}}(\delta)$, and $\delta<1$ from $D_{z_{3}}(\delta)\cap D_{z_{1}}(\delta)$. Since $1/2<-1/2+\sqrt{5}/2$ holds, we obtain $\delta<1/2$, and we have proved the existence of $\delta>0$. We now show the uniformly convergence of the holomorphic disks $w_{\epsilon}$. We prove this in the case $a_{2}<a_{1}<a_{3}$ only since other cases can be proved in a similar way as below. We first prove \eqref{eq:k3p0}. By the proof of Lemma \ref{trisc} $(\rm{iii})$, we have 
  \[
    w_{\epsilon}(z)=x_{3}^{\epsilon}+(x_{1}^{\epsilon}-x_{3}^{\epsilon})e^{\pi(1-\alpha_{1}^{\epsilon})}\dfrac{\Gamma(\alpha_{3}^{\epsilon}+\alpha_{1}^{\epsilon})}{\Gamma(\alpha_{3}^{\epsilon})\Gamma(\alpha_{1}^{\epsilon})}\int_{0}^{z}\zeta^{\alpha_{3}^{\epsilon}-1}(\zeta-1)^{\alpha_{1}^{\epsilon}-1}\,d\zeta.
  \]
  Since the Morse index of $p_{2}$ is zero, we have $p_{0}=p_{2}$. For $z\in D(\delta)$ and $\left\lvert z\right\rvert \leq 1$, we have
  \begin{align}
    \label{eq1-1}
    &\left\lvert w_{\epsilon}(z)-p_{2}\right\rvert\notag\\
    &\leq \left\lvert x_{3}^{\epsilon}+(x_{1}^{\epsilon}-x_{3}^{\epsilon})e^{\pi(1-\alpha_{1}^{\epsilon})}\dfrac{\Gamma(\alpha_{3}^{\epsilon}+\alpha_{1}^{\epsilon})}{\Gamma(\alpha_{3}^{\epsilon})\Gamma(\alpha_{1}^{\epsilon})}\int_{0}^{\delta z/\left\lvert z\right\rvert}\zeta^{\alpha_{3}^{\epsilon}-1}(\zeta-1)^{\alpha_{1}^{\epsilon}-1}\,d\zeta-p_{2}\right\rvert\notag\\
    &\hspace{4cm}+\left\lvert (x_{1}^{\epsilon}-x_{3}^{\epsilon})\dfrac{\Gamma(\alpha_{3}^{\epsilon}+\alpha_{1}^{\epsilon})}{\Gamma(\alpha_{3}^{\epsilon})\Gamma(\alpha_{1}^{\epsilon})}\int_{\delta z/\left\lvert z\right\rvert }^{z}\zeta^{\alpha_{3}^{\epsilon}-1}(\zeta-1)^{\alpha_{1}^{\epsilon}-1}\,d\zeta\right\rvert.
  \end{align}
  We also have 
  \begin{align*}
    &(x_{1}^{\epsilon}-x_{3}^{\epsilon})e^{\pi(1-\alpha_{1}^{\epsilon})}\dfrac{\Gamma(\alpha_{3}^{\epsilon}+\alpha_{1}^{\epsilon})}{\Gamma(\alpha_{3}^{\epsilon})\Gamma(\alpha_{1}^{\epsilon})}\int_{0}^{\delta z/\left\lvert z\right\rvert}\zeta^{\alpha_{3}^{\epsilon}-1}(\zeta-1)^{\alpha_{1}^{\epsilon}-1}\,d\zeta\\
    &=(x_{1}^{\epsilon}-x_{3}^{\epsilon})\dfrac{\Gamma(\alpha_{3}^{\epsilon}+\alpha_{1}^{\epsilon})}{\Gamma(\alpha_{3}^{\epsilon}+1)\Gamma(\alpha_{1}^{\epsilon})}\left(\delta\dfrac{z}{\left\lvert z\right\rvert }\right)^{\alpha_{3}^{\epsilon}}\,_{2}F_{1}(\alpha_{3}^{\epsilon},1-\alpha_{1}^{\epsilon},\alpha_{3}^{\epsilon};\delta z/\left\lvert z\right\rvert),
  \end{align*}
  which yields
  \begin{align}
    \label{eq1-2}
    &\left\lvert x_{3}^{\epsilon}+(x_{1}^{\epsilon}-x_{3}^{\epsilon})e^{\pi(1-\alpha_{1}^{\epsilon})}\dfrac{\Gamma(\alpha_{3}^{\epsilon}+\alpha_{1}^{\epsilon})}{\Gamma(\alpha_{3}^{\epsilon})\Gamma(\alpha_{1}^{\epsilon})}\int_{0}^{\delta z/\left\lvert z\right\rvert}\zeta^{\alpha_{3}^{\epsilon}-1}(\zeta-1)^{\alpha_{1}^{\epsilon}-1}\,d\zeta-p_{2}\right\rvert\notag\\
    &\leq \left\lvert x_{3}^{\epsilon}+(x_{1}^{\epsilon}-x_{3}^{\epsilon})\dfrac{\Gamma(\alpha_{3}^{\epsilon}+\alpha_{1}^{\epsilon})}{\Gamma(\alpha_{3}^{\epsilon}+1)\Gamma(\alpha_{1}^{\epsilon})}-p_{2}\right\rvert+\left\lvert (x_{1}^{\epsilon}-x_{3}^{\epsilon})\dfrac{\Gamma(\alpha_{3}^{\epsilon}+\alpha_{1}^{\epsilon})}{\Gamma(\alpha_{3}^{\epsilon}+1)\Gamma(\alpha_{1}^{\epsilon})}\right\rvert \left\lvert \left(\delta\dfrac{z}{\left\lvert z\right\rvert }\right)^{\alpha_{3}^{\epsilon}}-1\right\rvert\notag\\
    &\hspace{0.5cm}+\left\lvert (x_{1}^{\epsilon}-x_{3}^{\epsilon})\dfrac{\Gamma(\alpha_{3}^{\epsilon}+\alpha_{1}^{\epsilon})}{\Gamma(\alpha_{3}^{\epsilon}+1)\Gamma(\alpha_{1}^{\epsilon})}\right\rvert \left\lvert \left(\delta\dfrac{z}{\left\lvert z\right\rvert }\right)^{\alpha_{3}^{\epsilon}}\right\rvert \left\lvert \,_{2}F_{1}\left(\alpha_{3}^{\epsilon},1-\alpha_{1}^{\epsilon},\alpha_{3}^{\epsilon};\delta\dfrac{z}{\left\lvert z\right\rvert}\right)-1\right\rvert\notag\\
    &\leq \left\lvert x_{3}^{\epsilon}+(x_{1}^{\epsilon}-x_{3}^{\epsilon})\dfrac{\Gamma(\alpha_{3}^{\epsilon}+\alpha_{1}^{\epsilon})}{\Gamma(\alpha_{3}^{\epsilon}+1)\Gamma(\alpha_{1}^{\epsilon})}-p_{2}\right\rvert\notag\\
    &\hspace{0.5cm}+\left\lvert x_{1}^{\epsilon}-x_{3}^{\epsilon}\right\rvert\dfrac{\Gamma(\alpha_{3}^{\epsilon}+\alpha_{1}^{\epsilon})}{\Gamma(\alpha_{3}^{\epsilon}+1)\Gamma(\alpha_{1}^{\epsilon})}\left\{\delta^{\alpha_{3}^{\epsilon}} \left(\sum_{n=1}^{\infty}\dfrac{(\alpha_{3}^{\epsilon})_{n}(1-\alpha_{1}^{\epsilon})_{n}}{(\alpha_{3}^{\epsilon}+1)_{n}n!}\delta^{n}\right)+\left\lvert \delta^{\alpha_{3}^{\epsilon}}e^{i\pi\alpha_{3}^{\epsilon}}-1\right\rvert\right\}.
  \end{align}
  We also have
  \begin{align}
    \label{eq1-3}
    \left\lvert \int_{\delta z/\left\lvert z\right\rvert }^{z}\zeta^{\alpha_{3}^{\epsilon}-1}(\zeta-1)^{\alpha_{1}^{\epsilon}-1}\,d\zeta\right\rvert&=\left\lvert \int_{\delta}^{\left\lvert z\right\rvert }\left(\dfrac{z}{\left\lvert z\right\rvert}t\right)^{\alpha_{3}^{\epsilon}-1}\left(\dfrac{z}{\left\lvert z\right\rvert}t-1\right)^{\alpha_{1}^{\epsilon}-1}\dfrac{z}{\left\lvert z\right\rvert}\,dt\right\rvert\notag\\
    &\leq \int_{\delta}^{\left\lvert z\right\rvert }t^{\alpha_{3}^{\epsilon}-1}\left\lvert \dfrac{z}{\left\lvert z\right\rvert}t-1\right\rvert^{\alpha_{1}^{\epsilon}-1}\,dt\notag\\
    &\leq\int_{\delta}^{\left\lvert z\right\rvert }\delta^{\alpha_{3}^{\epsilon}-1}\delta^{\alpha_{1}^{\epsilon}-1}\,dt\notag\\
    &=\delta^{\alpha_{3}^{\epsilon}+\alpha_{1}^{\epsilon}-2}(\left\lvert z\right\rvert-\delta)\leq\delta^{\alpha_{3}^{\epsilon}+\alpha_{1}^{\epsilon}-2}(1-\delta).
  \end{align}
  By \eqref{eq1-1}-\eqref{eq1-3}, we obtain 
  \begin{align}
    \label{eq1-4}
    &\left\lvert w_{\epsilon}(z)-p_{2}\right\rvert\notag\\
    &\leq \left\lvert x_{3}^{\epsilon}+(x_{1}^{\epsilon}-x_{3}^{\epsilon})\dfrac{\Gamma(\alpha_{3}^{\epsilon}+\alpha_{1}^{\epsilon})}{\Gamma(\alpha_{3}^{\epsilon}+1)\Gamma(\alpha_{1}^{\epsilon})}-p_{2}\right\rvert\notag\\
    &\hspace{1cm}+\left\lvert x_{1}^{\epsilon}-x_{3}^{\epsilon}\right\rvert\dfrac{\Gamma(\alpha_{3}^{\epsilon}+\alpha_{1}^{\epsilon})}{\Gamma(\alpha_{3}^{\epsilon}+1)\Gamma(\alpha_{1}^{\epsilon})}\left\{\delta^{\alpha_{3}^{\epsilon}} \left(\sum_{n=1}^{\infty}\dfrac{(\alpha_{3}^{\epsilon})_{n}(1-\alpha_{1}^{\epsilon})_{n}}{(\alpha_{3}^{\epsilon}+1)_{n}n!}\delta^{n}\right)+\left\lvert \delta^{\alpha_{3}^{\epsilon}}e^{i\pi\alpha_{3}^{\epsilon}}-1\right\rvert\right.\notag\\
    &\hspace{10cm}\biggl.+\delta^{\alpha_{3}^{\epsilon}+\alpha_{1}^{\epsilon}-2}(1-\delta)\biggr\}
  \end{align}
  in the case $z\in D$ and $\left\lvert z\right\rvert \leq 1$. We denote the right hand side of the above inequation by $M_{1,\epsilon}$. On the other hand, we also have 
  \[
    w_{\epsilon}(z)=x_{2}^{\epsilon}-(x_{1}^{\epsilon}-x_{2}^{\epsilon})\dfrac{\Gamma(\alpha_{1}^{\epsilon}+\alpha_{2}^{\epsilon})}{\Gamma(\alpha_{1}^{\epsilon})\Gamma(\alpha_{2}^{\epsilon})}\int_{\infty}^{z}\zeta^{\alpha_{3}^{\epsilon}-1}(\zeta-1)^{\alpha_{1}^{\epsilon}-1}\,d\zeta.
  \]
  If one has $z\in D(\delta)$ and $\left\lvert z\right\rvert\geq 1$, then we have
  \begin{align}
    \label{eq2-1}
    \left\lvert w_{\epsilon}(z)-p_{2}\right\rvert&\leq \left\lvert x_{2}^{\epsilon}-p_{2}\right\rvert+\left\lvert -(x_{1}^{\epsilon}-x_{2}^{\epsilon})\dfrac{\Gamma(\alpha_{1}^{\epsilon}+\alpha_{2}^{\epsilon})}{\Gamma(\alpha_{1}^{\epsilon})\Gamma(\alpha_{2}^{\epsilon})}\int_{\infty}^{z}\zeta^{\alpha_{3}^{\epsilon}-1}(\zeta-1)^{\alpha_{1}^{\epsilon}-1}\,d\zeta\right\rvert\notag\\
    &\leq \left\lvert x_{2}^{\epsilon}-p_{2}\right\rvert+\left\lvert -(x_{1}^{\epsilon}-x_{2}^{\epsilon})\dfrac{\Gamma(\alpha_{1}^{\epsilon}+\alpha_{2}^{\epsilon})}{\Gamma(\alpha_{1}^{\epsilon})\Gamma(\alpha_{2}^{\epsilon})}\int_{\infty}^{\delta^{-1}z/\left\lvert z\right\rvert }\zeta^{\alpha_{3}^{\epsilon}-1}(\zeta-1)^{\alpha_{1}^{\epsilon}-1}\,d\zeta\right\rvert\notag\\
    &\hspace{2cm}+\left\lvert -(x_{1}^{\epsilon}-x_{2}^{\epsilon})\dfrac{\Gamma(\alpha_{1}^{\epsilon}+\alpha_{2}^{\epsilon})}{\Gamma(\alpha_{1}^{\epsilon})\Gamma(\alpha_{2}^{\epsilon})}\int_{\delta^{-1}z/\left\lvert z\right\rvert }^{z}\zeta^{\alpha_{3}^{\epsilon}-1}(\zeta-1)^{\alpha_{1}^{\epsilon}-1}\,d\zeta\right\rvert.
  \end{align}
  Here, we have
  \begin{align*}
    &-(x_{1}^{\epsilon}-x_{2}^{\epsilon})\dfrac{\Gamma(\alpha_{1}^{\epsilon}+\alpha_{2}^{\epsilon})}{\Gamma(\alpha_{1}^{\epsilon})\Gamma(\alpha_{2}^{\epsilon})}\int_{\infty}^{\delta^{-1}z/\left\lvert z\right\rvert }\zeta^{\alpha_{3}^{\epsilon}-1}(\zeta-1)^{\alpha_{1}^{\epsilon}-1}\,d\zeta\\
    &\hspace{1cm}=(x_{1}^{\epsilon}-x_{2}^{\epsilon})\dfrac{\Gamma(\alpha_{1}^{\epsilon}+\alpha_{2}^{\epsilon})}{\Gamma(\alpha_{1}^{\epsilon})\Gamma(\alpha_{2}^{\epsilon}+1)}\left(\delta^{-1}\dfrac{z}{\left\lvert z\right\rvert}\right)^{-\alpha_{2}^{\epsilon}}\,_{2}F_{1}\left(\alpha_{2}^{\epsilon},1-\alpha_{1}^{\epsilon},\alpha_{2}^{\epsilon}+1;\left(\delta^{-1}\dfrac{z}{\left\lvert z\right\rvert}\right)^{-1}\right),
  \end{align*}
  which leads to the following:
  \begin{align}
    \label{eq2-2}
    &\left\lvert -(x_{1}^{\epsilon}-x_{2}^{\epsilon})\dfrac{\Gamma(\alpha_{1}^{\epsilon}+\alpha_{2}^{\epsilon})}{\Gamma(\alpha_{1}^{\epsilon})\Gamma(\alpha_{2}^{\epsilon})}\int_{\infty}^{\delta^{-1}z/\left\lvert z\right\rvert }\zeta^{\alpha_{3}^{\epsilon}-1}(\zeta-1)^{\alpha_{1}^{\epsilon}-1}\,d\zeta\right\rvert\notag\\
    &\hspace{1cm}\leq \left\lvert x_{1}^{\epsilon}-x_{2}^{\epsilon}\right\rvert \dfrac{\Gamma(\alpha_{1}^{\epsilon}+\alpha_{2}^{\epsilon})}{\Gamma(\alpha_{1}^{\epsilon})\Gamma(\alpha_{2}^{\epsilon}+1)}\delta^{\alpha_{2}^{\epsilon}}\,_{2}F_{1}\left(\alpha_{2}^{\epsilon},1-\alpha_{1}^{\epsilon},\alpha_{2}^{\epsilon}+1;\delta\right).
  \end{align}
  We also have
  \begin{align}
    \label{eq2-3}
    \left\lvert \int_{\delta^{-1}z/\left\lvert z\right\rvert }^{z}\zeta^{\alpha_{3}^{\epsilon}-1}(\zeta-1)^{\alpha_{1}^{\epsilon}-1}\,d\zeta\right\rvert&=\left\lvert \int_{\delta^{-1}}^{\left\lvert z\right\rvert }\left(\dfrac{z}{\left\lvert z\right\rvert}t\right)^{\alpha_{3}^{\epsilon}-1}\left(\dfrac{z}{\left\lvert z\right\rvert}t-1\right)^{\alpha_{1}^{\epsilon}-1}\dfrac{z}{\left\lvert z\right\rvert}\,dt\right\rvert\notag\\
    &\leq\int_{\left\lvert z\right\rvert}^{\delta^{-1}}t^{\alpha_{3}^{\epsilon}-1}\left\lvert \dfrac{z}{\left\lvert z\right\rvert}t-1\right\rvert^{\alpha_{1}^{\epsilon}-1}\,dt\notag\\
    &\leq\int_{\left\lvert z\right\rvert}^{\delta^{-1}}\delta^{\alpha_{1}^{\epsilon}-1}\,dt\notag\\
    &=\delta^{\alpha_{1}^{\epsilon}-1}(\delta^{-1}-\left\lvert z\right\rvert)\leq\delta^{\alpha_{1}^{\epsilon}-1}(\delta^{-1}-1).
  \end{align}
  By \eqref{eq2-1}-\eqref{eq2-3}, we have 
  \begin{align}
    \label{eq2-4}
    \left\lvert w_{\epsilon}(z)-p_{2}\right\rvert\leq \left\lvert x_{2}^{\epsilon}-p_{2}\right\rvert+\left\lvert x_{1}^{\epsilon}-x_{2}^{\epsilon}\right\rvert \dfrac{\Gamma(\alpha_{1}^{\epsilon}+\alpha_{2}^{\epsilon})}{\Gamma(\alpha_{1}^{\epsilon})\Gamma(\alpha_{2}^{\epsilon}+1)}\delta^{\alpha_{2}^{\epsilon}}\,_{2}F_{1}\left(\alpha_{2}^{\epsilon},1-\alpha_{1}^{\epsilon},\alpha_{2}^{\epsilon}+1;\delta\right)\notag\\
    +\left\lvert x_{1}^{\epsilon}-x_{2}^{\epsilon}\right\rvert \dfrac{\Gamma(\alpha_{1}^{\epsilon}+\alpha_{2}^{\epsilon})}{\Gamma(\alpha_{1}^{\epsilon})\Gamma(\alpha_{2}^{\epsilon})}\delta^{\alpha_{1}^{\epsilon}-1}(\delta^{-1}-1).
  \end{align}
  We denote the right hand side of the above inequation by $M_{2,\epsilon}$. We obtain
  \[
    \sup_{z\in D(\delta)}\left\lvert w_{\epsilon}(z)-p_{2}\right\rvert\leq\max\{M_{1,\epsilon},M_{2,\epsilon}\}.
  \]
  By Lemma \ref{critk3}, \ref{angles} and Corollary \ref{angleslim}, one has
  \begin{equation}
    \label{eq3-1}
    \lim_{\epsilon\to+0}\alpha_{1}^{\epsilon}=0,\,\lim_{\epsilon\to+0}\alpha_{2}^{\epsilon}=1,\,\lim_{\epsilon\to+0}\alpha_{3}^{\epsilon}=0
  \end{equation}
  and 
  \begin{equation}
    \label{eq3-2}
    \lim_{\epsilon\to+0}\epsilon\Gamma(\alpha_{1}^{\epsilon})=\dfrac{\pi}{a_{1}-a_{2}},\,\lim_{\epsilon\to+0}\epsilon\Gamma(1-\alpha_{2}^{\epsilon})=\dfrac{\pi}{a_{3}-a_{2}},\,\lim_{\epsilon\to+0}\epsilon\Gamma(\alpha_{1}^{\epsilon})=\dfrac{\pi}{a_{3}-a_{1}}.
  \end{equation}
  By \eqref{eq3-1} and \eqref{eq3-2}, we obtain
  \begin{gather}
    \label{eq3-3}
    \lim_{\epsilon\to+0}\left\lvert x_{1}^{\epsilon}-x_{3}^{\epsilon}\right\rvert\dfrac{\Gamma(\alpha_{3}^{\epsilon}+\alpha_{1}^{\epsilon})}{\Gamma(\alpha_{3}^{\epsilon})\Gamma(\alpha_{1}^{\epsilon})}=\lim_{\epsilon\to+0}\left\lvert x_{1}^{\epsilon}-x_{3}^{\epsilon}\right\rvert\dfrac{\epsilon\cdot\epsilon\Gamma(1-\alpha_{2}^{\epsilon})}{\epsilon\Gamma(\alpha_{3}^{\epsilon})\cdot\epsilon\Gamma(\alpha_{1}^{\epsilon})}=0,\\
    \lim_{\epsilon\to+0}\left\lvert x_{1}^{\epsilon}-x_{2}^{\epsilon}\right\rvert \dfrac{\Gamma(\alpha_{1}^{\epsilon}+\alpha_{2}^{\epsilon})}{\Gamma(\alpha_{1}^{\epsilon})\Gamma(\alpha_{2}^{\epsilon})}=\lim_{\epsilon\to+0}\left\lvert x_{1}^{\epsilon}-x_{2}^{\epsilon}\right\rvert \dfrac{\epsilon\Gamma(\alpha_{1}^{\epsilon}+\alpha_{2}^{\epsilon})}{\epsilon\Gamma(\alpha_{1}^{\epsilon})\cdot\Gamma(\alpha_{2}^{\epsilon})}=0
  \end{gather}
  and
  \begin{align}
    \label{eq3-4}
    &\lim_{\epsilon\to+0}\left(x_{3}^{\epsilon}+(x_{1}^{\epsilon}-x_{3}^{\epsilon})\dfrac{\Gamma(\alpha_{3}^{\epsilon}+\alpha_{1}^{\epsilon})}{\Gamma(\alpha_{3}^{\epsilon}+1)\Gamma(\alpha_{1}^{\epsilon})}-p_{2}\right)\notag\\
    =\,&\lim_{\epsilon\to+0}\left(x_{3}^{\epsilon}+(x_{1}^{\epsilon}-x_{3}^{\epsilon})\dfrac{\epsilon\Gamma(1-\alpha_{2}^{\epsilon})}{\Gamma(\alpha_{3}^{\epsilon}+1)\cdot\epsilon\Gamma(\alpha_{1}^{\epsilon})}-p_{2}\right)\notag\\
    =\,&(p_{3}-p_{2})+(p_{1}-p_{3})\cdot\dfrac{a_{1}-a_{2}}{a_{3}-a_{2}}\notag\\
    =\,&p_{3}\cdot\dfrac{a_{3}-a_{1}}{a_{3}-a_{2}}-p_{2}+p_{1}\cdot\dfrac{a_{1}-a_{2}}{a_{3}-a_{2}}\notag\\
    =\,&\dfrac{b_{1}-b_{3}}{a_{3}-a_{2}}+\dfrac{b_{3}-b_{2}}{a_{3}-a_{2}}+\dfrac{b_{2}-b_{1}}{a_{3}-a_{2}}=0.
  \end{align}
  By \eqref{eq3-1} and \eqref{eq3-3}-\eqref{eq3-4}, we obtain $M_{1,\epsilon},M_{2,\epsilon}\rightarrow0\,(\epsilon\to+0)$ and thus obtain \eqref{eq:k3p0}. We next prove \eqref{eq:k3ext}. We first have
  \begin{align}
    \label{eq4-1}
    &\left\lvert w_{\epsilon}(\phi_{z_{3},\delta}(\tau,\sigma))-I_{3}(\epsilon\tau)\right\rvert=\left\lvert w_{\epsilon}(\phi_{z_{3},\delta}(\tau,\sigma))-(p_{3}+(p_{2}-p_{3})e^{-(a_{1}-a_{3})\epsilon\tau})\right\rvert\notag\\
    &<\left\lvert x_{3}^{\epsilon}-p_{3}\right\rvert+\left\lvert (x_{1}^{\epsilon}-x_{3}^{\epsilon})\dfrac{\Gamma(\alpha_{3}^{\epsilon}+\alpha_{1}^{\epsilon})}{\Gamma(\alpha_{3}^{\epsilon}+1)\Gamma(\alpha_{1}^{\epsilon})}-(p_{2}-p_{3})\right\rvert\notag\\
    &\hspace{6cm}\cdot\left\lvert[\phi_{z_{3},\delta}(\tau,\sigma)]^{\alpha_{3}}\,_{2}F_{1}(\alpha_{3}^{\epsilon},1-\alpha_{1}^{\epsilon},\alpha_{3}^{\epsilon}+1;\phi_{z_{3},\delta}(\tau,\sigma))\right\rvert\notag\\
    &\hspace{2cm}+\left\lvert p_{2}-p_{3}\right\rvert\left\lvert [\phi_{z_{3},\delta}(\tau,\sigma)]^{\alpha_{3}}\right\rvert \left\lvert\,_{2}F_{1}(\alpha_{3}^{\epsilon},1-\alpha_{1}^{\epsilon},\alpha_{3}^{\epsilon}+1;\phi_{z_{3},\delta}(\tau,\sigma))-1\right\rvert\notag\\
    &\hspace{6cm}+\left\lvert p_{2}-p_{3}\right\rvert\left\lvert [\phi_{z_{3},\delta}(\tau,\sigma)]^{\alpha_{3}}-e^{-(a_{1}-a_{3})\epsilon\tau}\right\rvert.
  \end{align}
  Since we have $\phi_{z_{3},\delta}(\tau,\sigma)=\delta\exp[\pi(\tau+i\sigma)]$, one has 
  \begin{align}
    \label{eq4-2}
    \left\lvert[\phi_{z_{3},\delta}(\tau,\sigma)]^{\alpha_{3}}\right\rvert\left\lvert\,_{2}F_{1}(\alpha_{3}^{\epsilon},1-\alpha_{1}^{\epsilon},\alpha_{3}^{\epsilon}+1;\phi_{z_{3},\delta}(\tau,\sigma))\right\rvert\leq\delta^{\alpha_{3}^{\epsilon}}\sum_{n=0}^{\infty}\dfrac{(\alpha_{3}^{\epsilon})_{n}(1-\alpha_{1}^{\epsilon})_{n}}{(\alpha_{3}^{\epsilon}+1)_{n}n!}\delta^{n}\notag\\
    \hspace{5cm}=\delta^{\alpha_{3}^{\epsilon}}\,_{2}F_{1}(\alpha_{3}^{\epsilon},1-\alpha_{1}^{\epsilon},\alpha_{3}^{\epsilon}+1;\delta).
  \end{align}
  The following also holds:
  \begin{align}
    \label{eq4-3}
    &\sup_{(\tau,\sigma)\in\Theta_{ext}}\left\lvert\,_{2}F_{1}(\alpha_{3}^{\epsilon},1-\alpha_{1}^{\epsilon},\alpha_{3}^{\epsilon}+1;\phi_{z_{3},\delta}(\tau,\sigma))-1\right\rvert\notag\\
    &=\sup_{(\tau,\sigma)\in\Theta_{ext}}\left\lvert\sum_{n=1}^{\infty}\dfrac{(\alpha_{3}^{\epsilon})_{n}(1-\alpha_{1}^{\epsilon})_{n}}{(\alpha_{3}^{\epsilon}+1)_{n}n!}[\phi_{z_{3},\delta}(\tau,\sigma)]^{n}\right\rvert\notag\\
    &\leq\sup_{(\tau,\sigma)\in\Theta_{ext}}\sum_{n=1}^{\infty}\dfrac{(\alpha_{3}^{\epsilon})_{n}(1-\alpha_{1}^{\epsilon})_{n}}{(\alpha_{3}^{\epsilon}+1)_{n}n!}\delta^{n}=\,_{2}F_{1}(\alpha_{3}^{\epsilon},1-\alpha_{1}^{\epsilon},\alpha_{3}^{\epsilon}+1;\delta)-1.
  \end{align}
  We have 
  \begin{align}
    \label{eq4-4}
    \left\lvert [\phi_{z_{3},\delta}(\tau,\sigma)]^{\alpha_{3}}-e^{-(a_{1}-a_{3})\epsilon\tau}\right\rvert&\leq \left\lvert e^{\pi\alpha_{3}^{\epsilon}\tau-i\sigma\pi\alpha_{3}^{\epsilon}}-e^{\pi\alpha_{3}^{\epsilon}\tau}\right\rvert+\left\lvert e^{\pi\alpha_{3}^{\epsilon}\tau}-e^{-(a_{1}-a_{3})\epsilon\tau}\right\rvert\notag\\
    &\leq \left\lvert e^{-i\pi\alpha_{3}^{\epsilon}}-1\right\rvert+\left\lvert e^{\pi\alpha_{3}^{\epsilon}\tau}-e^{-(a_{1}-a_{3})\epsilon\tau}\right\rvert
  \end{align}
  When we put $h_{\epsilon}(\tau)\coloneqq e^{\pi\alpha_{3}^{\epsilon}\tau}-e^{-(a_{1}-a_{3})\epsilon\tau}$, one has the following:
  \[
    \dfrac{dh_{\epsilon}}{d\tau}=\pi\alpha_{3}^{\epsilon}e^{\pi\alpha_{3}^{\epsilon}\tau}-(a_{3}-a_{1})\epsilon e^{-(a_{1}-a_{3})\epsilon\tau}.
  \]
  We also obtain 
  \[
    \tau<\dfrac{1}{\pi\alpha_{3}^{\epsilon}-(a_{3}-a_{1})\epsilon}\log\dfrac{(a_{3}-a_{1})\epsilon}{\pi\alpha_{3}^{\epsilon}}\,(<0)
  \]
  if and only if $dh_{\epsilon}/d\tau>0$ by using Lemma \ref{anglesineq}. Because of $h_{\epsilon}(\tau)\rightarrow 0\,(\tau\rightarrow-\infty)$, we have $h_{\epsilon}(\tau)>0$ in $\tau\in(-\infty,0)$. We therefore obtain 
  \begin{align*}
    \sup_{\tau\in(-\infty,0)}\left\lvert e^{\pi\alpha_{3}^{\epsilon}\tau}-e^{-(a_{1}-a_{3})\epsilon\tau}\right\rvert=h_{\epsilon}\left(\dfrac{1}{\pi\alpha_{3}^{\epsilon}-(a_{3}-a_{1})\epsilon}\log\dfrac{(a_{3}-a_{1})\epsilon}{\pi\alpha_{3}^{\epsilon}}\right).
  \end{align*}
  By using Corollary \ref{angleslim} and the property of Euler's constant $e$
  \[
    \lim_{x\rightarrow0}(1+x)^{1/x}=e,
  \]
  we obtain 
  \begin{equation}
    \label{eq4-5}
    \lim_{\epsilon\to+0}\sup_{\tau\in(-\infty,0)}\left\lvert e^{\pi\alpha_{3}^{\epsilon}\tau}-e^{-(a_{1}-a_{3})\epsilon\tau}\right\rvert=\left\lvert e^{-1}-e^{-1}\right\rvert=0.
  \end{equation}
  By \eqref{eq4-1}-\eqref{eq4-5}, we finally obtain
  \[
    \lim_{\epsilon\to+0}\sup_{(\tau,\sigma)\in\Theta_{ext}}\left\lvert w_{\epsilon}(\phi_{z_{3},\delta}(\tau,\sigma))-I_{3}(\epsilon\tau)\right\rvert=0.
  \]
  We can prove the uniformly convergence of $w_{\epsilon}(\phi_{z_{1},\delta}(\tau,\sigma))$ similarly. We next prove that of $w_{\epsilon}(\phi_{z_{2},\delta}(\tau,\sigma))$. We first have 
  \begin{align}
    \label{eq5-1}
    &\sup_{(\tau,\sigma)\in\Theta_{ext}}\left\lvert w_{\epsilon}(\phi_{z_{2},\delta}(\tau,\sigma))-I_{2}(\epsilon\tau)\right\rvert\notag\\
    &\leq\left\lvert \dfrac{(x_{1}^{\epsilon}-x_{2}^{\epsilon})\Gamma(1-\alpha_{3}^{\epsilon})}{\Gamma(2-\alpha_{3}^{\epsilon}-\alpha_{1}^{\epsilon})}\right\rvert \sup_{(\tau,\sigma)\in\Theta_{ext}}\left\lvert [\phi_{z_{2},\delta}(\tau,\sigma)]^{-\alpha_{2}^{\epsilon}}\,_{2}F_{1}(\alpha_{2}^{\epsilon},1-\alpha_{1}^{\epsilon},\alpha_{2}^{\epsilon}+1;[\phi_{z_{2},\delta}(\tau,\sigma)]^{-1})\right\rvert\notag\\
    &\hspace{10cm}+\left\lvert x_{2}^{\epsilon}-p_{2}\right\rvert.
  \end{align}
  We also have 
  \begin{align}
    \label{eq5-2}
    &\sup_{(\tau,\sigma)\in\Theta_{ext}}\left\lvert [\phi_{z_{2},\delta}(\tau,\sigma)]^{-\alpha_{2}^{\epsilon}}\,_{2}F_{1}(\alpha_{2}^{\epsilon},1-\alpha_{1}^{\epsilon},\alpha_{2}^{\epsilon}+1;[\phi_{z_{2},\delta}(\tau,\sigma)]^{-1})\right\rvert\notag\\
    &\hspace{5cm}\leq\delta^{\alpha_{2}^{\epsilon}}\,_{2}F_{1}(\alpha_{2}^{\epsilon},1-\alpha_{1}^{\epsilon},\alpha_{2}^{\epsilon}+1;\delta)
  \end{align}
  and
  \begin{equation}
    \label{eq5-3}
    \lim_{\epsilon\to+0}\dfrac{\Gamma(1-\alpha_{3}^{\epsilon})}{\Gamma(\alpha_{1}^{\epsilon})\Gamma(2-\alpha_{3}^{\epsilon}-\alpha_{1}^{\epsilon})}=\lim_{\epsilon\to+0}\dfrac{\Gamma(1-\alpha_{3}^{\epsilon})}{\epsilon\Gamma(\alpha_{1}^{\epsilon})\Gamma(\alpha_{2}^{\epsilon}+1)}\epsilon=0.
  \end{equation}
  By \eqref{eq3-1}, \eqref{eq5-1}-\eqref{eq5-3} and $\left\lvert\,_{2}F_{1}(1,1,2;\delta)\right\rvert<\infty$, we obtain
  \[
    \lim_{\epsilon\to+0}\sup_{(\tau,\sigma)\in\Theta_{ext}}\left\lvert w_{\epsilon}(\phi_{z_{2},\delta}^{\epsilon}(\tau,\sigma))-I_{2}(\tau)\right\rvert=0.
  \]
\section{The case of $M=\mathbb{R},k=4$.}
\subsection{Main result.}
Let $L_{i}^{\epsilon}\,(i=1,2,3,4)$ be affine Lagrangian sections of $T^{*}\mathbb{R}$, and $f_{i}$ the function on $\mathbb{R}$ such that $L_{i}^{\epsilon}\coloneqq\operatorname{graph}(\epsilon df_{i})$. We here assume the following two conditions. First, $L_{i}^{\epsilon}\cap L_{j}^{\epsilon}$ is a one-point set when $i\neq j$ holds. Second, there exists a convex quadrilateral which has vertices $x_{1}^{\epsilon},x_{2}^{\epsilon},x_{3}^{\epsilon},x_{4}^{\epsilon}$ in counterclockwise order, where $x_{i}^{\epsilon}\in L_{i}\cap L_{i+1}$ for $i=1,2,3,4$. Under these assumptions, $f_{i+1}-f_{i}$ is a Morse function, and there is a unique critical point $p_{i}$ of $f_{i+1}-f_{i}$ for $i=1,2,3,4$. The above assumptions also implies that the moduli space $\mathcal{M}_{g}(\mathbb{R};\vec{f},\vec{p})$ of gradient trees is a one-point set. We prove this in subsection 4.3. Let $(I,(T,i,v_{1},l))$ be the unique element of $\mathcal{M}_{g}(\mathbb{R};\vec{f},\vec{p})$. By the definition of trees, $T\in Gr_{4}$ has at most one internal edge. We therefore consider $l\geq0$ as the length of the internal edge of $T$. Let $w_{\epsilon}$ be the Schwarz-Christoffel map from the upper half plane to the convex quadrilateral $x_{1}^{\epsilon}x_{2}^{\epsilon}x_{3}^{\epsilon}x_{4}^{\epsilon}$ such that $w_{\epsilon}(z_{i})=x_{i}^{\epsilon}$ for $i=1,2,3$, and denote $z_{4,\epsilon}=w_{\epsilon}^{-1}(x_{4}^{\epsilon})\in(0,1)$. When we set $z_{1}=1,z_{2}=\infty,z_{3}=0$, we obtain $\lim_{\epsilon\to+0}z_{4,\epsilon}=0$ in the case $(p_{3}-p_{1})(p_{4}-p_{2})>0$, and $\lim_{\epsilon\to+0}z_{4,\epsilon}=1$ in the case $(p_{3}-p_{1})(p_{4}-p_{2})<0$. We prove these in subsection 4.5.\\
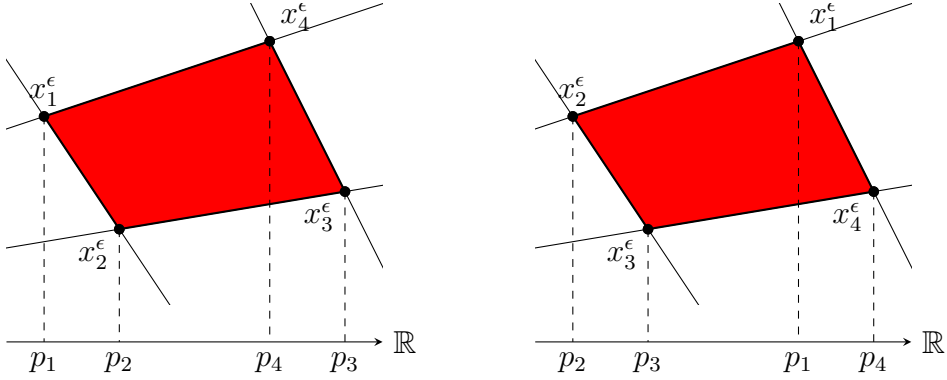
\begin{figure}[tb]
  \centering
  \begin{tikzpicture}
    \begin{scope}
      \clip (-0.5,0.5) rectangle (4.5,4.5);
      \draw [name path=L1] plot(\x,\x/3+3);
      \draw [name path=L2] plot(\x,{\x*(-3/2)+3});
      \draw [name path=L3] plot(\x,\x/6+4/3);
      \draw [name path=L4] plot(\x,-2*\x+10);
      \fill [black, name intersections={of=L1 and L2, by={x1}}] (x1) circle (2pt) node [above] at (x1) {$x_{1}^{\epsilon}$};
      \fill [black, name intersections={of=L2 and L3, by={x2}}] (x2) circle (2pt) node [below left] at (x2) {$x_{2}^{\epsilon}$};
      \fill [black, name intersections={of=L3 and L4, by={x3}}] (x3) circle (2pt) node [below left] at (x3) {$x_{3}^{\epsilon}$};
      \fill [black, name intersections={of=L4 and L1, by={x4}}] (x4) circle (2pt) node [above right] at (x4) {$x_{4}^{\epsilon}$};
      \fill[red,opacity=0.5] (x1)--(x2)--(x3)--(x4)--cycle;
      \draw[thick] (x1)--(x2)--(x3)--(x4)--cycle;
      \fill[black] (x1) circle (2pt);
      \fill[black] (x2) circle (2pt);
      \fill[black] (x3) circle (2pt);
      \fill[black] (x4) circle (2pt);
    \end{scope}
    \draw[->,>=stealth] (-0.5,0)--(4.5,0) node[right]{$\mathbb{R}$};
    \draw[dashed] (0,3)--(0,0) node[below]{$p_{1}$};
    \draw[dashed] (1,1.5)--(1,0) node[below]{$p_{2}$};
    \draw[dashed] (4,2)--(4,0) node[below]{$p_{3}$};
    \draw[dashed] (3,4)--(3,0) node[below]{$p_{4}$};
    \begin{scope}[xshift=200]
      \begin{scope}
        \clip (-0.5,0.5) rectangle (4.5,4.5);
        \draw [name path=L1] plot(\x,\x/3+3);
        \draw [name path=L2] plot(\x,{\x*(-3/2)+3});
        \draw [name path=L3] plot(\x,\x/6+4/3);
        \draw [name path=L4] plot(\x,-2*\x+10);
        \fill [black, name intersections={of=L1 and L2, by={x1}}] (x1) circle (2pt) node [above] at (x1) {$x_{2}^{\epsilon}$};
        \fill [black, name intersections={of=L2 and L3, by={x2}}] (x2) circle (2pt) node [below left] at (x2) {$x_{3}^{\epsilon}$};
        \fill [black, name intersections={of=L3 and L4, by={x3}}] (x3) circle (2pt) node [below left] at (x3) {$x_{4}^{\epsilon}$};
        \fill [black, name intersections={of=L4 and L1, by={x4}}] (x4) circle (2pt) node [above right] at (x4) {$x_{1}^{\epsilon}$};
        \fill[red,opacity=0.5] (x1)--(x2)--(x3)--(x4)--cycle;
        \draw[thick] (x1)--(x2)--(x3)--(x4)--cycle;
        \fill[black] (x1) circle (2pt);
        \fill[black] (x2) circle (2pt);
        \fill[black] (x3) circle (2pt);
        \fill[black] (x4) circle (2pt);
      \end{scope}
      \draw[->,>=stealth] (-0.5,0)--(4.5,0) node[right]{$\mathbb{R}$};
      \draw[dashed] (0,3)--(0,0) node[below]{$p_{2}$};
      \draw[dashed] (1,1.5)--(1,0) node[below]{$p_{3}$};
      \draw[dashed] (4,2)--(4,0) node[below]{$p_{4}$};
      \draw[dashed] (3,4)--(3,0) node[below]{$p_{1}$};
    \end{scope}
  \end{tikzpicture}
  \caption{Figure of quadrilaterals $x_{1}^{\epsilon}x_{2}^{\epsilon}x_{3}^{\epsilon}x_{4}^{\epsilon}$ corresponding to the case $(p_{3}-p_{1})(p_{4}-p_{2})>0$ (the left hand side) and $(p_{3}-p_{1})(p_{4}-p_{2})<0$ (the right hand side).}
  \label{const-1}
\end{figure}
We now discuss how to divide the upper half plane in the case $(p_{3}-p_{1})(p_{4}-p_{2})>0$. In this case, we divided $\overline{\mathbb{H}}$ into seven regions (see Figure \ref{fig:k4upperhalf}).
\begin{figure}[tb]
  \centering
  \begin{tikzpicture}
    \clip (-2,-0.75) rectangle (9,7);
    \draw[->]  (-2,0)--(8,0) node [below]{Re};
    \fill[blue,opacity=0.3] (0.15,0) arc (0:180:0.15)--(-0.15,0)--cycle;
    \fill[blue,opacity=0.3] (0.65,0) arc (0:180:0.15)--(-0.35,0)--cycle;
    \fill[red,opacity=0.3] (5.5,0) arc (0:180:1.5)--(4,0)--cycle;
    \fill[red,opacity=0.3] (8,0)--(8,9)--(-6.5,9)--(-6.5,0) arc (180:0:6.5)--(6.5,0)--cycle;
    \fill[green,opacity=0.3] (1.5,0) arc (0:180:1.5)--(-1.5,0)--(-1,0) arc (180:0:1)--(1,0)--cycle;
    \draw[dashed] (0.15,0) arc (0:180:0.15);
    \draw[dashed] (0.65,0) arc (0:180:0.15);
    \fill[black] (0,0) circle (0.05) node[below]{$z_{3}$};
    \fill[black] (0.5,0) circle (0.05) node[below]{$z_{4,\epsilon}$};
    \draw[dashed] (1,0) arc (0:180:1);
    \draw[dashed] (1.5,0) arc (0:180:1.5);
    \draw[dashed] (5.5,0) arc (0:180:1.5);
    \fill[black] (4,0) circle (0.05) node[below]{$z_{1}$};
    \draw[dashed] (6.5,0) arc (0:180:6.5);
  \end{tikzpicture}
  \caption{The figure of the upper half plane divided by seven regions.}
  \label{fig:k4upperhalf}
\end{figure}
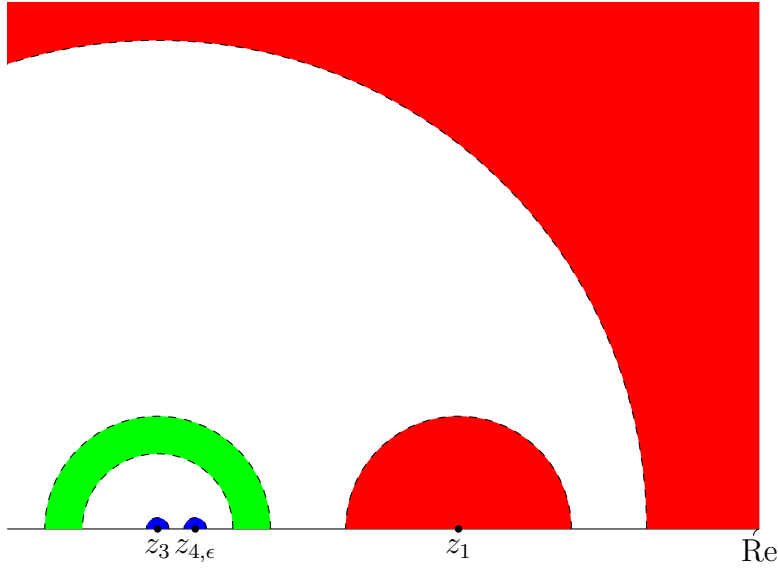
In the Figure \ref{fig:k4upperhalf}, the white regions correspond to internal vertices, the red regions and the blue regions correspond to external edges, and the green region corresponds to the internal edge. We first divide the upper half plane as follows:
\[
  \overline{\mathbb{H}}=D_{z_{3}}(\delta)\cup D_{z_{1}}(\delta)\cup D_{z_{2}}(\delta)\cup D(\delta).
\]
Here, the definition of $D_{p}(\delta)$ is the same as that in subsection 3.1, and we take a real positive number $\delta>0$ such that $D_{z_{i}}(\delta)\cap D_{z_{j}}(\delta)=\emptyset$ for $i,j\in\{1,2,3\},i\neq j$. The set $D(\delta)$ is then given by 
\[
  D(\delta)\coloneqq\overline{\mathbb{H}}\setminus(D_{z_{3}}(\delta)\cup D_{z_{1}}(\delta)\cup D_{z_{2}}(\delta)).
\]
Two regions $D_{z_{1}}(\delta)$ and $D_{z_{2}}(\delta)$ correspond to the red regions in Figure \ref{fig:k4upperhalf}. In order to make the analysis in the region $D_{z_{3}}(\delta)$ easier, we transform conformally the upper half plane by $\psi_{\epsilon}(z)\coloneqq z/z_{4,\epsilon}$ (see Figure \ref{4markedpts}). After that, we divide the upper half plane $\psi_{\epsilon}(\overline{\mathbb{H}})$ into four regions $D_{0}(\delta),D_{1}(\delta),D_{\infty}(\delta),D(\delta)$ again:
\begin{align*}
  \psi_{\epsilon}(\overline{\mathbb{H}})&=D_{0}(\delta)\cup D_{1}(\delta)\cup D_{\infty}(\delta)\cup D(\delta)\\
  &=D_{\psi_{\epsilon}(z_{3})}(\delta)\cup D_{\psi_{\epsilon}(z_{4,\epsilon})}(\delta)\cup D_{\psi_{\epsilon}(z_{2})}(\delta)\cup D(\delta).
\end{align*}
Here, we define two regions $D_{z_{3}}^{\epsilon}(\delta)$ and $D_{z_{4,\epsilon}}^{\epsilon}(\delta)$ as 
\[
  D_{z_{3}}^{\epsilon}(\delta)\coloneqq \psi_{\epsilon}^{-1}(D_{\psi_{\epsilon}(z_{3})}(\delta))=D_{z_{3}}(z_{4,\epsilon}\delta),D_{z_{4},\epsilon}^{\epsilon}(\delta)\coloneqq \psi_{\epsilon}^{-1}(D_{\psi_{\epsilon}(z_{4,\epsilon})}(\delta))=D_{z_{4,\epsilon}}(z_{4,\epsilon}\delta),
\]
which are subsets of $D_{z_{3}}(\delta)(\subset \overline{\mathbb{H}})$ (see Figure \ref{4markedpts}). These two regions correspond to the blue regions in Figure \ref{fig:k4upperhalf}. Since we have $(p_{3}-p_{1})(p_{4}-p_{2})>0$, $z_{4,\epsilon}$ converges to zero as $\epsilon\to+0$. Thus, $z_{4,\epsilon}<\delta^{2}$ holds for any sufficiently small $\epsilon>0$, which implies that we can define the region $D_{int}^{\epsilon}(\delta)$ as follows (see Figure \ref{4markedpts}):
\begin{align*}
  D_{int}^{\epsilon}(\delta)&\coloneqq D_{z_{3}}(\delta)\cup \psi_{\epsilon}^{-1}(D_{z_{2}}(\delta))\\
  &=D_{z_{3}}(\delta)\cup D_{z_{2}}(\delta/z_{4,\epsilon})=\{z\in\overline{\mathbb{H}}\mid z_{4,\epsilon}/\delta<\left\lvert z\right\rvert<\delta\}.
\end{align*}
\begin{figure}[tb]
  \centering
  \begin{tikzpicture}
    \begin{scope}
      \clip (-1,-1) rectangle (6,5);
      \draw[dashed] (0.5,0) arc (0:180:0.5);
      \draw[dashed] (2.5,0) arc (0:180:0.5);
      \draw[dashed] (3.5,0) arc (0:180:3.5);
      \fill[red,opacity=0.3] (0.5,0) arc (0:180:0.5)--(-0.5,0)--cycle;
      \fill[red,opacity=0.3] (2.5,0) arc (0:180:0.5)--(1.5,0)--cycle;
      \draw[red] (2,2)node{$D_{z_{2}}(\delta)$};
      \fill[red,opacity=0.3] (3.5,0)--(5.5,0)--(5.5,6)--(-1,6)--(-3.5,0) arc (180:0:3.5)--cycle;
      \draw[->] (-1,0)--(5.5,0)node[below]{Re};
      \fill[black] (0,0) circle (0.06) node[below]{$0$};
      \fill[black] (2,0) circle (0.06) node[below]{$1$};
      \draw (3.5,0)node[below]{$1/\delta$};
      \draw[red] (0,0.5) node[above]{$D_{z_{3}}(\delta)$};
      \draw[red] (2,0.5) node[above]{$D_{z_{1}}(\delta)$};
    \end{scope}
    \begin{scope}[xshift=200]
      \draw[->] (-1,2.5)--(1,2.5);
      \draw (0,2.5)node[above]{$\psi_{\epsilon}$};
      \draw (0,2.5)node[below]{$\psi_{\epsilon}(z)=z/z_{4,\epsilon}$};
    \end{scope}
    \begin{scope}[xshift=270]
      \clip (-1,-1) rectangle (6,5);
      \draw[dashed] (0.5,0) arc (0:180:0.5);
      \draw[dashed] (2.5,0) arc (0:180:0.5);
      \draw[dashed] (3.5,0) arc (0:180:3.5);
      \draw[dashed] (4.5,0) arc (0:180:4.5);
      \fill[blue,opacity=0.3] (0.5,0) arc (0:180:0.5)--(-0.5,0)--cycle;
      \fill[blue,opacity=0.3] (2.5,0) arc (0:180:0.5)--(1.5,0)--cycle;
      \draw[teal] (4,4)node{$\psi_{\epsilon}(D_{int}^{\epsilon}(\delta))$};
      \fill[green,opacity=0.3] (3.5,0)--(4.5,0) arc (0:180:4.5)--(-4.5,0)--(-3.5,0) arc (180:0:3.5)--cycle;
      \draw[->] (-1,0)--(5.5,0)node[below]{Re};
      \fill[black] (0,0) circle (0.06) node[below]{$0$};
      \fill[black] (2,0) circle (0.06) node[below]{$1$};
      \draw (3.5,0)node[below]{$1/\delta$};
      \draw (4.5,0)node[below]{$\delta/z_{4,\epsilon}$};
      \draw[blue] (0,0.5) node[above]{$D_{0}(\delta)$};
      \draw[blue] (2,0.5) node[above]{$D_{1}(\delta)$};
    \end{scope}
  \end{tikzpicture}
  \caption{The figure of nearby $0$ in the upper half plane in the case $z_{4,\epsilon}\rightarrow0\,(\epsilon\to+0)$.}
  \label{4markedpts}
\end{figure}
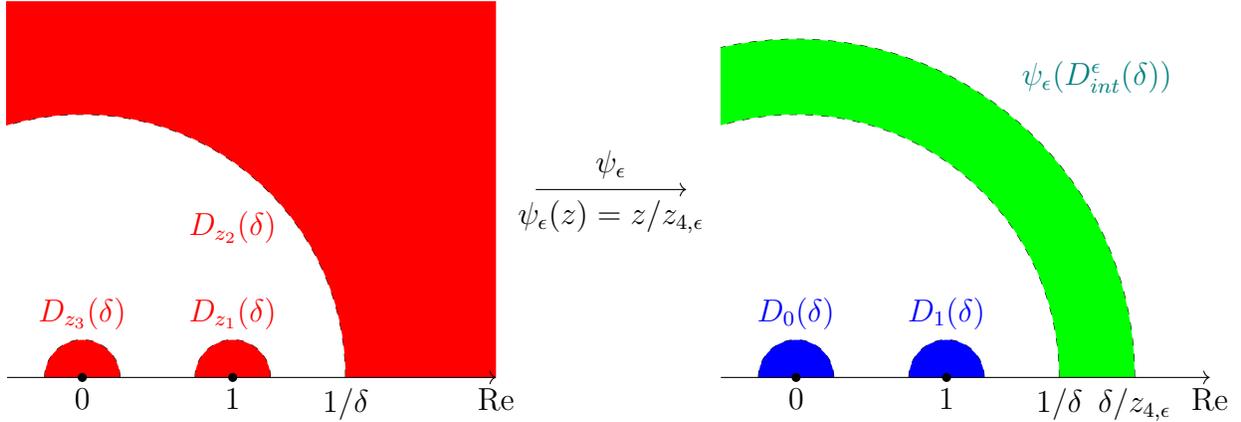
This region corresponds to the green region in Figure \ref{fig:k4upperhalf}. Also, two regions $D(\delta)$ and $\psi_{\epsilon}^{-1}(D(\delta))$ correspond to the white regions in Figure \ref{fig:k4upperhalf}. We therefore obtain the partition of the upper half plane. We next give the transformations between stripes and regions which correspond to edges of tree. Here, we sometime omit $\epsilon$ from $z_{4,\epsilon}$ in order to treat $z_{4,\epsilon}$ as $z_{1},z_{2},z_{3}$. We define the transformation $\phi_{z_{i},\delta}^{\epsilon}$ from the stripe $\Theta_{ext}$ to the region $D_{z_{i}}^{\epsilon}(\delta)$ for each $i=1,2,3,4$ as follows:
\[
  D_{z_{i}}^{\epsilon}(\delta)\coloneqq
  \begin{cases*}
    D_{z_{i}}(\delta)\,(i=1,2)\\
    D_{z_{i}}(z_{4,\epsilon}\delta)\,(i=3,4)
  \end{cases*}
  ,\phi_{z_{i},\delta}^{\epsilon}\coloneqq
  \begin{cases*}
    \phi_{z_{i},\delta}\,(i=1,2)\\
    \psi_{\epsilon}^{-1}\circ\phi_{\psi_{\epsilon}(z_{i}),\delta}\,(i=3,4)
  \end{cases*}
  \,.
\]
Here, the definition of $\phi_{p,\delta}$ and $\Theta_{ext}$ is the same as that in subsection 3.1. We next define the transformation $\phi_{int}$ from the stripe $\Theta_{int}^{\epsilon}(\delta)$ to the region $D_{int}^{\epsilon}(\delta)$ analogously to $\phi_{z_{i},\delta}^{\epsilon}$ as follows (see Figure \ref{fig-intedge}):
\begin{align*}
  \phi_{int}(\tau,\sigma)&\coloneqq\exp\left[-\pi\tau+i\pi(1-\sigma)\right]\\
  \Theta_{int}^{\epsilon}(\delta)&\coloneqq\left.\left\{(\tau,\sigma)\,\middle| \,\tau\in\left(-\dfrac{1}{\pi}\log \delta,-\dfrac{1}{\pi}\log z_{4,\epsilon}+\dfrac{1}{\pi}\log \delta\right),\sigma\in[0,1]\right\}\right.\,.
\end{align*}
\begin{figure}[tb]
  \centering
  \begin{tikzpicture}
    \begin{scope}
      \clip (-1,-1) rectangle (6,5);
      \fill[green,opacity=0.3] (0.5,0)--(4.5,0)--(4.5,3)--(0.5,3)--cycle;
      \draw[->] (-0.5,0)--(5,0)node[above]{$\tau$};
      \draw[->] (0,-0.5)--(0,4.5)node[left]{$\sigma$};
      \draw (0.5,3)--(4.5,3);
      \draw[dashed] (0.5,0)--(0.5,3);
      \draw[dashed] (4.5,0)--(4.5,3);
      \draw[teal] (2.5,1.5)node{$\Theta_{int}^{\epsilon}(\delta)$};
      \draw (0,0)node[below left]{$O$};
      \draw (0.5,0)node[below]{$\tau_{1,\delta}$};
      \draw (4.5,0)node[below]{$\tau_{2,\delta,\epsilon}$};
      \draw (0,3)node[left]{$1$};
    \end{scope}
    \begin{scope}[xshift=180]
      \draw[->] (-1,2)--(1,2);
      \draw (0,2)node[above]{$\phi_{int}$};
    \end{scope}
    \begin{scope}[xshift=250]
      \clip (-1,-1.25) rectangle (6,6);
      \draw[dashed] (1.5,0) arc (0:180:1.5);
      \draw[dashed] (4,0) arc (0:180:4);
      \draw[teal] (3.7,3.7)node{$\phi_{int}(\Theta_{int}^{\epsilon}(\delta))$};
      \draw[teal] (3.9,3.2)node{$(=D_{int}^{\epsilon}(\delta))$};
      \fill[green,opacity=0.3] (1.5,0)--(4,0) arc (0:180:4)--(-4,0)--(-1.5,0) arc (180:0:1.5)--cycle;
      \draw[->] (-1,0)--(5.5,0)node[below]{Re};
      \draw (1.5,0)node[below]{$\phi_{int}(\tau_{2,\delta,\epsilon},1)$};
      \draw (1.5,-0.5)node[below]{$(=z_{4,\epsilon}/\delta)$};
      \draw (4,0)node[below]{$\phi_{int}(\tau_{1,\delta},1)$};
      \draw (4,-0.5)node[below]{$(=\delta)$};
    \end{scope}
  \end{tikzpicture}
  \caption{The figure of $\phi_{int}$ and $D_{int}^{\epsilon}(\delta)$ in the case $z_{4,\epsilon}\rightarrow+0$ (Here, $\tau_{1,\epsilon}\coloneqq-(\log\delta)/\pi$ and $\tau_{2,\delta,\epsilon}\coloneqq-(\log(z_{4,\epsilon}/\delta))/\pi$).}
  \label{fig-intedge}
\end{figure}
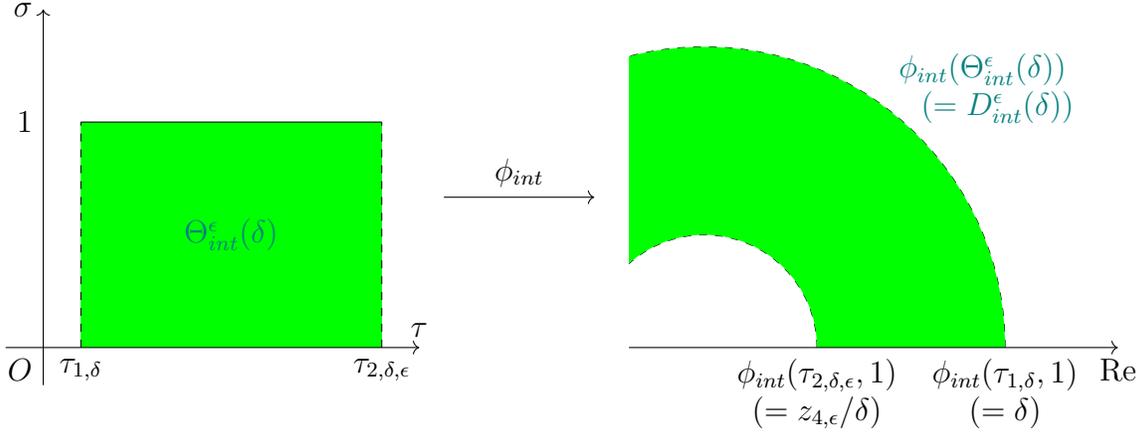
The following is our main theorem.
\begin{thm}
  \label{k4corr1}
  Let $\delta>0$ be a real number such that $D_{i}(\delta)\cap D_{j}(\delta)$ for $i\neq j$ and $i,j\in\{0,1,\infty\}$. The point $p_{0}\in \mathbb{R}$ is either $p_{1}$ or $p_{2}$ such that the Morse index of $p_{0}$ is zero, and the point $p_{0}'\in\mathbb{R}$ is either $p_{3}$ or $p_{4}$ satisfying the same condition as $p_{0}$. We set $I_{i}$ as the restriction of $I$ to the external edge $e_{i}$ of $T$, and set $I_{int}$ as the restriction of $I$ to the internal edge of $T$. If $(p_{3}-p_{1})(p_{4}-p_{2})>0$, then we obtain 
  \begin{gather}
    \lim_{\epsilon\to+0}\sup_{z\in D(\delta)}\left\lvert w_{\epsilon}(z)-p_{0}\right\rvert=0,\label{eq:k4p0-1-1}\\
    \lim_{\epsilon\to+0}\sup_{z\in D(\delta)}\left\lvert w_{\epsilon}\circ\psi_{\epsilon}^{-1}(z)-p_{0}'\right\rvert=0,\label{eq:k4p0-1-2}\\
    \lim_{\epsilon\to+0}\sup_{(\tau,\sigma)\in\Theta_{ext}}\left\lvert w_{\epsilon}\circ \phi_{z_{i},\delta}^{\epsilon}(\tau,\sigma)-I_{i}(\epsilon\tau)\right\rvert=0\,(i=1,2,3,4),\label{eq:k4ext-1}\\
    \lim_{\epsilon\to+0}\sup_{(\tau,\sigma)\in\Theta_{int}^{\epsilon}(\delta)}\left\lvert w_{\epsilon}\circ\phi_{int}(\tau,\sigma)-I_{int}(\epsilon\tau)\right\rvert=0.\label{eq:k4int-1}
  \end{gather}
\end{thm}
\begin{rem}
  We do not say the uniformly convergence of holomorphic disks in \eqref{eq:k4int-1}. This is because the domain $\Theta_{int}^{\epsilon}(\delta)$ in \eqref{eq:k4int-1} moves as $\epsilon>0$ varies.
\end{rem}
The above $p_{0}$ and $p_{0}'$ are determined uniquely as we see in Lemma \ref{square} and Lemma \ref{k4gradtree}. Note that the statement of this theorem The case $(p_{3}-p_{1})(p_{4}-p_{2})<0$ is reduced to the case $(p_{3}-p_{1})(p_{4}-p_{2})>0$ by using cyclic action of $SL(2;\mathbb{R})$ to $z_{1},z_{2},z_{3},z_{4,\epsilon}$ appropriately. From this fact and the theorem above, we obtain the statement of this theorem in the case the counterclockwise convex quadrilateral $x_{1}^{\epsilon}x_{2}^{\epsilon}x_{3}^{\epsilon}x_{4}^{\epsilon}$ is generic.
\subsection{Holomorphic disks and gradient trees.}
We first consider conditions when four affine Lagrangian sections $L_{1},L_{2},L_{3},L_{4}$ enable us to obtain the Schwarz-Christoffel map which maps the upper half plane to the interior of the quadrilateral bounded by $L_{1},L_{2},L_{3},L_{4}$.
\begin{lem}
  \label{square}
  Let $L_{1},L_{2},L_{3},L_{4}$ be affine Lagrangian sections in $T^{*}\mathbb{R}$ which satisfy $L_{i}\cap L_{i+1}\neq\emptyset,L_{i}\neq L_{i+1}$, for $i=1,2,3,4$. By identifying $T^{*}\mathbb{R}$ with $\mathbb{R}^{2}$, we express each $L_{i}$ as a line $\{(x,y)\mid y=a_{i}x+b_{i}\}$, where $a_{i}$ and $b_{i}$ are real numbers. Let $x_{i}=(p_{i},q_{i})$ be the intersection point between $L_{i}$ and $L_{i+1}$. Then, $x_{1},x_{2},x_{3},x_{4}$ forms a quadrilateral if and only if $p_{1},p_{2},p_{3},p_{4}$ satisfy one of conditions in $(B)$ in Table \ref{square-list}. Furthermore, the quadrilateral $x_{1}x_{2}x_{3}x_{4}$ is convex and has vertices $x_{1},x_{2},x_{3},x_{4}$ in counterclockwise order if and only if $a_{1},a_{2},a_{3},a_{4}$ satisfy one of conditions in $(A)$ in Table \ref{square-list}. Here, $(a,b)$ in Table \ref{square-list} is a open interval, and we set $(a,b)=\emptyset$ if $a\geq b$ holds.
  \begin{table}[htbp]
    \centering
    \begin{tabular}{c|c}
      $(A)$ & $(B)$\\ \hline\hline
      \multirow{2}{*}{$a_{2},a_{4}\in(a_{1},a_{3})$} & $p_{4}<p_{1}<p_{2}<p_{3}$\\
      & $p_{3}<p_{2}<p_{1}<p_{4}$\\ \hline
      \multirow{2}{*}{$a_{2},a_{4}\in(a_{3},a_{1})$} & $p_{1}<p_{4}<p_{3}<p_{2}$\\
      & $p_{2}<p_{3}<p_{4}<p_{1}$\\ \hline
      \multirow{2}{*}{$a_{1},a_{3}\in(a_{2},a_{4})$} & $p_{1}<p_{2}<p_{3}<p_{4}$\\
      & $p_{4}<p_{3}<p_{2}<p_{1}$\\ \hline
      \multirow{2}{*}{$a_{1},a_{3}\in(a_{4},a_{2})$} & $p_{3}<p_{4}<p_{1}<p_{2}$\\
      & $p_{2}<p_{1}<p_{4}<p_{3}$ \\ \hline
      \multirow{6}{*}{$\max\{a_{1},a_{3}\}<\min\{a_{2},a_{4}\}$} & $p_{4}<p_{1}<p_{3}<p_{2}$\\
      & $p_{2}<p_{3}<p_{1}<p_{4}$ \\
      & $p_{4}<p_{3}<p_{1}<p_{2}$ \\
      & $p_{2}<p_{1}<p_{3}<p_{4}$ \\
      & $p_{4}<p_{1}=p_{3}<p_{2}$ \\
      & $p_{2}<p_{3}=p_{1}<p_{4}$ \\ \hline
      \multirow{6}{*}{$\max\{a_{2},a_{4}\}<\min\{a_{1},a_{3}\}$} & $p_{1}<p_{4}<p_{2}<p_{3}$\\
      & $p_{3}<p_{2}<p_{4}<p_{1}$ \\
      & $p_{3}<p_{4}<p_{2}<p_{1}$\\
      & $p_{1}<p_{2}<p_{4}<p_{3}$ \\ 
      & $p_{1}<p_{4}=p_{2}<p_{3}$ \\
      & $p_{3}<p_{2}=p_{4}<p_{1}$\\ \hline
    \end{tabular}
    \caption{All of conditions which enable us to obtain convex quadrilaterals.}
    \label{square-list}
  \end{table}
\end{lem}
We can prove Lemma \ref{square} by elementary calculations. Let $f_{1},f_{2},f_{3},f_{4}$ be functions such that $L_{i}=\operatorname{graph}(df_{i})$ for $i=1,2,3,4$. We next study the moduli space $\mathcal{M}_{g}(\mathbb{R};\vec{f},\vec{p})$ in the case when $L_{1},L_{2},L_{3},L_{4}$ satisfy one of conditions in Table \ref{square-list}. By the definition of ribbon trees, the ribbon tree $(T,i)$ is the element of $Gr_{4}$ if and only if $T$ is isometric to one of the following trees in Figure \ref{k4trees2}. 
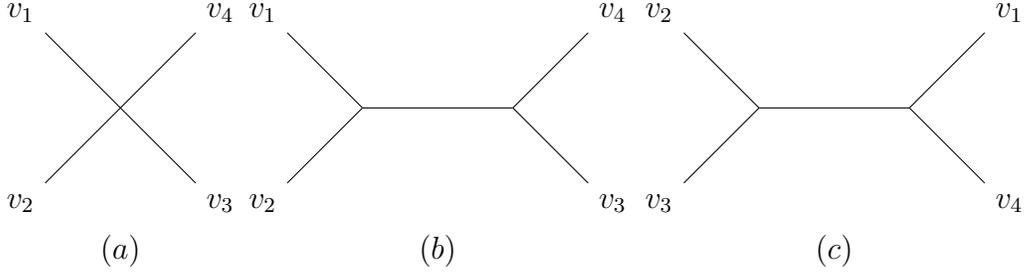
\begin{figure}[h]
  \center
  \begin{tikzpicture}
    \begin{scope}
      \draw[thin] (1,1)--(0,0);
      \draw[thin] (-1,1)--(0,0);
      \draw[thin] (1,-1)--(0,0);
      \draw[thin] (-1,-1)--(0,0);
      \draw (-1,1) node[above left]{$v_{1}$};
      \draw (-1,-1) node[below left]{$v_{2}$};
      \draw (1,-1) node[below right]{$v_{3}$};
      \draw (1,1) node[above right]{$v_{4}$};
      \draw (0,-1.9) node{$(a)$};
    \end{scope}
    \begin{scope}[xshift=120]
      \draw[thin] (2,1)--(1,0);
      \draw[thin] (2,-1)--(1,0);
      \draw[thin] (-1,0)--(1,0);
      \draw[thin] (-2,1)--(-1,0);
      \draw[thin] (-2,-1)--(-1,0);
      \draw (-2,1) node[above left]{$v_{1}$};
      \draw (-2,-1) node[below left]{$v_{2}$};
      \draw (2,-1) node[below right]{$v_{3}$};
      \draw (2,1) node[above right]{$v_{4}$};
      \draw (0,-1.9) node{$(b)$};
    \end{scope}
    \begin{scope}[xshift=270]
      \draw[thin] (2,1)--(1,0);
      \draw[thin] (2,-1)--(1,0);
      \draw[thin] (-1,0)--(1,0);
      \draw[thin] (-2,1)--(-1,0);
      \draw[thin] (-2,-1)--(-1,0);
      \draw (-2,1) node[above left]{$v_{2}$};
      \draw (-2,-1) node[below left]{$v_{3}$};
      \draw (2,-1) node[below right]{$v_{4}$};
      \draw (2,1) node[above right]{$v_{1}$};
      \draw (0,-1.9) node{$(c)$};
    \end{scope}
  \end{tikzpicture}
  \caption{Candidates of the tree $T$ of the ribbon tree $(T,i)\in Gr_{4}$.}
  \label{k4trees2}
\end{figure}
\begin{lem}
  \label{k4gradtree}
  The moduli space $\mathcal{M}_{g}(\mathbb{R};\vec{f},\vec{p})$ of gradient trees is a one-point set. The gradient tree is constructed by the tree which is isomorphic to the tree in column $(C)$ of Table \ref{tree-critpt2}. Here, $(a),(b),(c)$ correspond to Figure \ref{k4trees2}. 
  \begin{table}[htbp]
    \centering
    \begin{tabular}{c|c|c}
      $(A)$ & $(B)$ & $(C)$ \\ \hline\hline
      \multirow{2}{*}{$a_{2},a_{4}\in(a_{1},a_{3})$} & $p_{4}<p_{1}<p_{2}<p_{3}$ & \multirow{2}{*}{$(c)$} \\
      & $p_{3}<p_{2}<p_{1}<p_{4}$ & \\ \hline
      \multirow{2}{*}{$a_{2},a_{4}\in(a_{3},a_{1})$} & $p_{1}<p_{4}<p_{3}<p_{2}$ & \multirow{2}{*}{$(c)$} \\
      & $p_{2}<p_{3}<p_{4}<p_{1}$ & \\ \hline
      \multirow{2}{*}{$a_{1},a_{3}\in(a_{2},a_{4})$} & $p_{1}<p_{2}<p_{3}<p_{4}$ & \multirow{2}{*}{$(b)$} \\
      & $p_{4}<p_{3}<p_{2}<p_{1}$ & \\ \hline
      \multirow{2}{*}{$a_{1},a_{3}\in(a_{4},a_{2})$} & $p_{3}<p_{4}<p_{1}<p_{2}$ & \multirow{2}{*}{$(b)$} \\
      & $p_{2}<p_{1}<p_{4}<p_{3}$ & \\ \hline
      \multirow{6}{*}{$\max\{a_{1},a_{3}\}<\min\{a_{2},a_{4}\}$} & $p_{4}<p_{1}<p_{3}<p_{2}$ & \multirow{2}{*}{$(c)$} \\
      & $p_{2}<p_{3}<p_{1}<p_{4}$ &\\ \cline{2-3}
      & $p_{4}<p_{3}<p_{1}<p_{2}$ & \multirow{2}{*}{$(b)$} \\
      & $p_{2}<p_{1}<p_{3}<p_{4}$ &\\ \cline{2-3}
      & $p_{4}<p_{1}=p_{3}<p_{2}$ & \multirow{2}{*}{$(a)$} \\
      & $p_{2}<p_{3}=p_{1}<p_{4}$ &\\ \hline
      \multirow{6}{*}{$\max\{a_{2},a_{4}\}<\min\{a_{1},a_{3}\}$} & $p_{1}<p_{4}<p_{2}<p_{3}$ & \multirow{2}{*}{$(c)$} \\
      & $p_{3}<p_{2}<p_{4}<p_{1}$ &\\ \cline{2-3}
      & $p_{3}<p_{4}<p_{2}<p_{1}$ & \multirow{2}{*}{$(b)$} \\
      & $p_{1}<p_{2}<p_{4}<p_{3}$ &\\ \cline{2-3}
      & $p_{1}<p_{4}=p_{2}<p_{3}$ & \multirow{2}{*}{$(a)$} \\
      & $p_{3}<p_{2}=p_{4}<p_{1}$ &\\ \hline
    \end{tabular}
    \caption{All of conditions when we obtain a unique gradient tree.}
    \label{tree-critpt2}
  \end{table}
\end{lem}
\begin{proof}
  We prove the statement in the case $a_{2},a_{4}\in(a_{1},a_{3})$ only since other cases can be proved similarly. We first calculate gradient curves $I_{j}$ for $j=1,2,3,4$. This $I_{j}$ satisfies
  \[
    \dfrac{dI_{j}}{dt}=-\grad(f_{j+1}-f_{j}), I_{j}(t)\rightarrow p_{j}(t\rightarrow -\infty),
  \]
  we thus have
  \[
    I_{1}(t)=p_{1}, I_{2}(t)=p_{2}, I_{3}(t)=A_{3}e^{-(a_{4}-a_{3})t}+p_{3}, I_{4}(t)=A_{4}e^{-(a_{1}-a_{4})t}+p_{4}.
  \]
  Here, $A_{3},A_{4}$ are real numbers which are determined later. By the continuity of the gradient tree and $p_{2}\neq p_{1}(=0)$, the tree which construct the gradient tree $I\in\mathcal{M}_{g}(\mathbb{R};\vec{f},\vec{p})$ is homeomorphic to $(c)$ in Figure \ref{k4trees2}. The gradient tree $I$ holds $I_{3}(0)=p_{2},I_{4}(0)=p_{1}=0$, we therefore obtain 
  \[
    I_{3}(t)=(p_{2}-p_{3})e^{-(a_{4}-a_{3})t}+p_{3}, I_{4}(t)=(p_{1}-p_{4})e^{-(a_{1}-a_{4})t}+p_{4}.
  \]
  Let $I_{int}$ be the restriction of $I$ to the internal edge. By definition of gradient trees, $I_{int}$ satisfies
  \[
    \dfrac{dI_{int}}{dt}=-\grad(f_{4}-f_{2}), I_{int}(0)=p_{2}.
  \]
  There exists $l>0$ such that $I_{int}(l)=p_{1}$ if and only if the gradient tree $I$ can be constructed. Hereafter, we prove the existence of $l$. We first discuss the case $a_{2}\neq a_{4}$. In this case, we have 
  \[
    I_{int}(t)=\left(\dfrac{b_{4}-b_{2}}{a_{4}-a_{2}}+p_{2}\right)e^{-(a_{4}-a_{2})t}-\dfrac{b_{4}-b_{2}}{a_{4}-a_{2}}.
  \]
  The existence of $l>0$ such that $I_{int}(l)=p_{1}$ is equivalent to the existence of $l>0$ such that
  \[
    e^{-(a_{4}-a_{2})l}=\left.\left(p_{1}+\dfrac{b_{4}-b_{2}}{a_{4}-a_{2}}\right)\middle/\left(p_{2}+\dfrac{b_{4}-b_{2}}{a_{4}-a_{2}}\right)\right..
  \]
  We now set $a_{4}-a_{2}<0$. This assumption leads to $e^{-(a_{4}-a_{2})l}\in(1,\infty)$, and we have
  \[
    -\dfrac{b_{4}-b_{2}}{a_{4}-a_{2}}<p_{2}<p_{1}\,\text{or}\,p_{1}<p_{2}<-\dfrac{b_{4}-b_{2}}{a_{4}-a_{2}}.
  \]
  We also have 
  \begin{align*}
    p_{2}+\dfrac{b_{4}-b_{2}}{a_{4}-a_{2}}&=-\dfrac{(b_{3}-b_{2})(a_{4}-a_{2})-(b_{4}-b_{2})(a_{3}-a_{2})}{(a_{3}-a_{2})(a_{4}-a_{2})}\\
    &=-\dfrac{(b_{3}-b_{2})(a_{4}-a_{3})+(b_{3}-b_{2})(a_{3}-a_{2})}{(a_{3}-a_{2})(a_{4}-a_{2})}\\
    &\hspace{3cm}+\dfrac{(b_{4}-b_{3})(a_{3}-a_{2})+(b_{3}-b_{2})(a_{3}-a_{2})}{(a_{3}-a_{2})(a_{4}-a_{2})}\\
    &=\dfrac{a_{4}-a_{3}}{a_{4}-a_{2}}\left(-\dfrac{b_{3}-b_{2}}{a_{3}-a_{2}}+\dfrac{b_{4}-b_{3}}{a_{4}-a_{3}}\right)=\dfrac{a_{4}-a_{3}}{a_{4}-a_{2}}\left(p_{2}-p_{3}\right).
  \end{align*}
  Since $a_{2},a_{4}\in(a_{1},a_{3})$, we obtain $p_{3}<p_{2}$ (resp. $p_{3}>p_{2}$) when $p_{2}>-(b_{4}-b_{2})/(a_{4}-a_{2})$ (resp. $p_{2}<-(b_{4}-b_{2})/(a_{4}-a_{2})$) holds. By a similar argument as above, we also obtain $p_{4}>p_{1}$ (resp. $p_{4}<p_{1}$) when one has $p_{1}>-(b_{4}-b_{2})/(a_{4}-a_{2})$ (resp. $p_{1}<-(b_{4}-b_{2})/(a_{4}-a_{2})$). Hence the positive real number $l>0$ exists if and only if we have either $p_{4}<p_{1}<p_{2}<p_{3}$ or $p_{4}>p_{1}>p_{2}>p_{3}$. Since we assume $a_{2},a_{4}\in(a_{1},a_{3})$, the real number $l>0$ exists in the case $a_{4}-a_{2}>0$. We can also prove in the case $a_{4}-a_{2}<0$ by the same way. We next discuss the case $a_{2}=a_{4}$. By the differential equation which $I_{int}$ satisfies, we have $I_{int}(t)=-(b_{4}-b_{2})t+p_{2}$. We obtain the gradient tree $I\in\mathcal{M}_{g}(\mathbb{R};\vec{f},\vec{p})$ if and only if there exists $l>0$ such that $-(b_{4}-b_{2})l+p_{2}=p_{1}$. We have
  \begin{align*}
    b_{4}-b_{2}&=(a_{4}-a_{3})(a_{3}-a_{2})\left(\dfrac{b_{4}-b_{3}}{a_{4}-a_{3}}\dfrac{1}{a_{3}-a_{2}}+\dfrac{b_{3}-b_{2}}{a_{3}-a_{2}}\dfrac{1}{a_{4}-a_{3}}\right)\\
    &=(a_{4}-a_{3})(a_{3}-a_{2})\dfrac{p_{3}-p_{2}}{a_{4}-a_{3}}=(a_{3}-a_{2})(p_{3}-p_{2}),
  \end{align*}
  and this leads
  \[
    \dfrac{p_{2}-p_{1}}{b_{4}-b_{2}}=\dfrac{p_{2}-p_{1}}{p_{3}-p_{1}}\dfrac{1}{a_{3}-a_{2}}.
  \]
  Since one has $p_{4}<p_{1}<p_{2}<p_{3}$ or $p_{3}<p_{2}<p_{1}<p_{4}$, we obtain $(p_{2}-p_{1})/(b_{4}-b_{2})>0$. Therefore, the real number $l>0$ exists, and the gradient tree exists uniquely in the case $a_{2}=a_{4}$, too. 
\end{proof}
\subsection{Power series representations of holomorphic disks.}
We here study the power series representation of a Schwarz-Christoffel map from the upper half plane to a convex quadrilateral. 
\begin{lem}
  \label{sqr-rep1}
  Let $P$ be the interior of the convex quadrilateral $x_{1}x_{2}x_{3}x_{4}$ having vertices $x_{1},x_{2},x_{3},x_{4}$ in counterclockwise order and let $\pi\alpha_{1},\pi\alpha_{2},\pi\alpha_{3},\pi\alpha_{4}$ be the corresponding interior angles. We assume that $\alpha_{i}+\alpha_{i+1}\neq 1$ for $i=1,\dots,4$, namely, there are no pairs of parallel sides in the convex quadrilateral $P$. Let $x_{13}$ denote the intersection point of the straight line through $x_{4},x_{1}$ and the straight line through $x_{2},x_{3}$. Let $w$ be the Schwarz-Christoffel map from the upper half plane to $P$, and take $\xi\in(0,1)$ such that $w(0)=x_{1},w(\xi)=x_{2},w(1)=x_{3},w(\infty)=x_{4}$. Under these assumptions, we obtain the power series representation of $w$ as follows:\\
  $(\rm{i})$\,on $\{z\in\overline{\mathbb{H}}\mid \left\lvert z\right\rvert<\xi\}$
  \begin{align*}
    &w(z)=x_{1}+\dfrac{\Gamma(\alpha_{1}+\alpha_{2})}{\Gamma(\alpha_{1}+1)\Gamma(\alpha_{2})}\dfrac{x_{2}-x_{1}}{_{2}F_{1}(\alpha_{1},1-\alpha_{3},\alpha_{1}+\alpha_{2};\xi)}\left(\dfrac{z}{\xi}\right)^{\alpha_{1}}\\
    &\hspace{6cm}\cdot F_{1}\left(\alpha_{1},1-\alpha_{2},1-\alpha_{3},\alpha_{1}+1;\dfrac{z}{\xi},z\right),
  \end{align*}
  $(\rm{ii})$\,on $\{z\in\overline{\mathbb{H}}\mid \left\lvert z\right\rvert^{-1}<1\}$
  \begin{align*}
    &w(z)=x_{4}+\dfrac{\Gamma(2-\alpha_{1}-\alpha_{2})}{\Gamma(3-\alpha_{1}-\alpha_{2}-\alpha_{3})\Gamma(\alpha_{3})}\dfrac{x_{3}-x_{4}}{_{2}F_{1}(2-\alpha_{1}-\alpha_{2}-\alpha_{3},1-\alpha_{2},2-\alpha_{1}-\alpha_{2};\xi)}\\
    &\hspace{2cm}\cdot z^{\alpha_{1}+\alpha_{2}+\alpha_{3}-2} F_{1}\left(2-\alpha_{1}-\alpha_{2}-\alpha_{3},1-\alpha_{2},1-\alpha_{3},3-\alpha_{1}-\alpha_{2}-\alpha_{3};\dfrac{\xi}{z},\dfrac{1}{z}\right),
  \end{align*}
  $(\rm{iii})$\,on $\{z\in\overline{\mathbb{H}}\mid\xi<\left\lvert z\right\rvert<1\}$
  \begin{align*}
    &w(z)=x_{13}-e^{-\pi\alpha_{2}i}\dfrac{\Gamma(\alpha_{1}+\alpha_{2}-1)}{\Gamma(\alpha_{1})\Gamma(\alpha_{2})}\dfrac{x_{2}-x_{1}}{_{2}F_{1}(\alpha_{1},1-\alpha_{3},\alpha_{1}+\alpha_{2};\xi)}\\
    &\hspace{3cm}\cdot\left(\dfrac{z}{\xi}\right)^{\alpha_{1}+\alpha_{2}-1}G_{2}\left(1-\alpha_{2},1-\alpha_{3},\alpha_{1}+\alpha_{2}-1,1-\alpha_{1}-\alpha_{2};-\dfrac{\xi}{z},-z\right),
  \end{align*}
  $(\rm{iv})$\,on $\{z\in\overline{\mathbb{H}}\mid\left\lvert 1-z\right\rvert<1-\xi\}$
  \begin{align*}
    &w(z)=x_{3}+\dfrac{x_{2}-x_{3}}{_{2}F_{1}(\alpha_{3},1-\alpha_{1},\alpha_{2}+\alpha_{3};1-\xi)}\dfrac{\Gamma(\alpha_{2}+\alpha_{3})}{\Gamma(\alpha_{2})\Gamma(\alpha_{3})}\left(\dfrac{1-z}{1-\xi}\right)^{\alpha_{3}}\\
    &\hspace{3cm}\cdot F_{1}\left(\alpha_{3},1-\alpha_{1},1-\alpha_{2},\alpha_{3}+1;1-z,\dfrac{1-z}{1-\xi}\right).
  \end{align*}
\end{lem}
\begin{proof}
  We can prove $(\rm{i}),(\rm{ii}),(\rm{iv})$ by using the integral representation of Appell's $F_{1}$ function (Proposition \ref{AppellF1}). We only remark how we calculate special values of $F_{1}$. For example, in $(\rm{i})$, we have
  \begin{align*}
    \int_{0}^{z}\zeta^{\alpha_{1}-1}(\zeta-\xi)^{\alpha_{2}-1}(\zeta-1)^{\alpha_{3}-1}\,d\zeta=\int_{0}^{1}(zt)^{\alpha_{1}-1}(zt-\xi)^{\alpha_{2}-1}(zt-1)^{\alpha_{3}-1}z\,dt\\
    =e^{\pi(\alpha_{2}+\alpha_{3})i}z^{\alpha_{1}}\xi^{\alpha_{2}-1}\int_{0}^{1}t^{\alpha_{1}-1}(1-zt/\xi)^{\alpha_{2}-1}(1-zt)^{\alpha_{3}-1}\,dt.
  \end{align*}
  As a result, we obtain the special value as follows:
  \begin{align*}
    \lim_{z\to\xi}\int_{0}^{z}\zeta^{\alpha_{1}-1}(\zeta-\xi)^{\alpha_{2}-1}(\zeta-1)^{\alpha_{3}-1}\,d\zeta=e^{\pi(\alpha_{2}+\alpha_{3})i}\xi^{\alpha_{1}+\alpha_{2}-1}\int_{0}^{1}t^{\alpha_{1}-1}(1-t)^{\alpha_{2}-1}(1-\xi t)^{\alpha_{3}-1}\,dt\\
    =e^{\pi(\alpha_{2}+\alpha_{3})i}\xi^{\alpha_{1}+\alpha_{2}-1}B(\alpha_{1},\alpha_{2})\,_{2}F_{1}(\alpha_{1},1-\alpha_{3},\alpha_{1}+\alpha_{2};\xi).
  \end{align*}
  We can prove $(\rm{iii})$ by adapting the connection formula for $F_{1}$ (Lemma \ref{F1conn1}) to $(\rm{i})$. By $(\rm{i})$, we have 
  \begin{align*}
    &w(z)=x_{1}+\dfrac{\Gamma(\alpha_{1}+\alpha_{2})}{\Gamma(\alpha_{1}+1)\Gamma(\alpha_{2})}\dfrac{x_{2}-x_{1}}{_{2}F_{1}(\alpha_{1},1-\alpha_{3},\alpha_{1}+\alpha_{2};\xi)}\left(\dfrac{z}{\xi}\right)^{\alpha_{1}}\\
    &\hspace{6cm}\cdot F_{1}\left(\alpha_{1},1-\alpha_{2},1-\alpha_{3},\alpha_{1}+1;\dfrac{z}{\xi},z\right)
  \end{align*}
  in $\{z\in\overline{\mathbb{H}}\mid\left\lvert z\right\rvert<\xi\}$. By Lemma \ref{F1conn1}, we also have
  \begin{align*}
    &F_{1}\left(\alpha_{1},1-\alpha_{2},1-\alpha_{3},\alpha_{1}+1;\dfrac{z}{\xi},z\right)=F_{1}\left(\alpha_{1},1-\alpha_{3},1-\alpha_{2},\alpha_{1}+1;z,\dfrac{z}{\xi}\right)\\
    &=\dfrac{\Gamma(1-\alpha_{1}-\alpha_{2})\Gamma(\alpha_{1}+1)}{\Gamma(1-\alpha_{2})\Gamma(1)}\left(-\dfrac{z}{\xi}\right)^{-\alpha_{1}}F_{1}\left(\alpha_{1},0,1-\alpha_{3},\alpha_{1}+\alpha_{2};\dfrac{\xi}{z},\xi\right)\\
    &\hspace{2cm}+\dfrac{\Gamma(\alpha_{1}+\alpha_{2}-1)\Gamma(\alpha_{1}+1)}{\Gamma(\alpha_{1})\Gamma(\alpha_{1}+\alpha_{2})}\left(-\dfrac{z}{\xi}\right)^{\alpha_{2}-1}\\
    &\hspace{3cm}\cdot G_{2}\left(1-\alpha_{2},1-\alpha_{3},\alpha_{1}+\alpha_{2}-1,1-\alpha_{1}-\alpha_{2};-\dfrac{\xi}{z},-z\right)\\
    &=\dfrac{\Gamma(1-\alpha_{1}-\alpha_{2})\Gamma(\alpha_{1}+1)}{\Gamma(1-\alpha_{2})}\left(-\dfrac{z}{\xi}\right)^{-\alpha_{1}}\,_{2}F_{1}\left(\alpha_{1},1-\alpha_{3},\alpha_{1}+\alpha_{2};\xi\right)\\
    &\hspace{2cm}+\dfrac{\Gamma(\alpha_{1}+\alpha_{2}-1)\Gamma(\alpha_{1}+1)}{\Gamma(\alpha_{1})\Gamma(\alpha_{1}+\alpha_{2})}\left(-\dfrac{z}{\xi}\right)^{\alpha_{2}-1}\\
    &\hspace{3cm}\cdot G_{2}\left(1-\alpha_{2},1-\alpha_{3},\alpha_{1}+\alpha_{2}-1,1-\alpha_{1}-\alpha_{2};-\dfrac{\xi}{z},-z\right).
  \end{align*}
  Here, the argument of $-z/\xi$ of the factors $(-z/\xi)^{*}$ is assigned to be zero if $z$ is a real number which satisfies $-\infty<z/\xi<z<0$. Since we have $\arg z=\arg (z/\xi)=0$ when $z$ is positive real number and $0<z<z/\xi<\infty$ holds, we set $-z/\xi=\exp[-i\pi]\cdot z/\xi$. We therefore obtain
  \begin{align*}
    &w(z)=x_{1}+\dfrac{\Gamma(\alpha_{1}+\alpha_{2})}{\Gamma(\alpha_{1}+1)\Gamma(\alpha_{2})}\dfrac{x_{2}-x_{1}}{_{2}F_{1}(\alpha_{1},1-\alpha_{3},\alpha_{1}+\alpha_{2};\xi)}\left(\dfrac{z}{\xi}\right)^{\alpha_{1}}\\
    &\hspace{1cm}\cdot \left[\dfrac{\Gamma(1-\alpha_{1}-\alpha_{2})\Gamma(\alpha_{1}+1)}{\Gamma(1-\alpha_{2})}\left(-\dfrac{z}{\xi}\right)^{-\alpha_{1}}\,_{2}F_{1}\left(\alpha_{1},1-\alpha_{3},\alpha_{1}+\alpha_{2};\xi\right)\right.\\
    &\hspace{4cm}+\dfrac{\Gamma(\alpha_{1}+\alpha_{2}-1)\Gamma(\alpha_{1}+1)}{\Gamma(\alpha_{1})\Gamma(\alpha_{1}+\alpha_{2})}\left(-\dfrac{z}{\xi}\right)^{\alpha_{2}-1}\\
    &\hspace{4cm}\left.\cdot G_{2}\left(1-\alpha_{2},1-\alpha_{3},\alpha_{1}+\alpha_{2}-1,1-\alpha_{1}-\alpha_{2};-\dfrac{\xi}{z},-z\right)\right]\\
    &=x_{1}+\dfrac{\Gamma(\alpha_{1}+\alpha_{2})\Gamma(1-\alpha_{1}-\alpha_{2})}{\Gamma(\alpha_{2})\Gamma(1-\alpha_{2})}e^{\pi\alpha_{1}i}(x_{2}-x_{1})\\
    &\hspace{3cm}+\dfrac{\Gamma(\alpha_{1}+\alpha_{2}-1)}{\Gamma(\alpha_{1})\Gamma(\alpha_{2})}e^{-\pi(\alpha_{2}-1)i}\dfrac{x_{2}-x_{1}}{_{2}F_{1}(\alpha_{1},1-\alpha_{3},\alpha_{1}+\alpha_{2};\xi)}\\
    &\hspace{3cm}\cdot\left(\dfrac{z}{\xi}\right)^{\alpha_{1}+\alpha_{2}-1}G_{2}\left(1-\alpha_{2},1-\alpha_{3},\alpha_{1}+\alpha_{2}-1,1-\alpha_{1}-\alpha_{2};-\dfrac{\xi}{z},-z\right).
  \end{align*}
  Horn's hypergeometric series $G_{2}$ above converges when $\xi<\left\lvert z\right\rvert<1$ holds. For $\alpha_{1}+\alpha_{2}<1$, we have the triangle $x_{1}x_{2}x_{13}$ which has interior angles $\pi\alpha_{1},\pi\alpha_{2}$ and $\pi(1-\alpha_{1}-\alpha_{2})$ in counterclockwise order. Then we obtain
  \begin{align*}
    &x_{1}+\dfrac{\Gamma(\alpha_{1}+\alpha_{2})\Gamma(1-\alpha_{1}-\alpha_{2})}{\Gamma(\alpha_{2})\Gamma(1-\alpha_{2})}e^{\pi\alpha_{1}i}(x_{2}-x_{1})\\
    &=x_{1}+\dfrac{\left\lvert x_{2}-x_{1}\right\rvert\sin\pi\alpha_{2}}{\sin\pi(1-\alpha_{1}-\alpha_{2})}\dfrac{x_{2}-x_{1}}{\left\lvert x_{2}-x_{1}\right\rvert}\dfrac{x_{13}-x_{1}}{x_{2}-x_{1}}\dfrac{\left\lvert x_{2}-x_{1}\right\rvert}{\left\lvert x_{13}-x_{1}\right\rvert}\\
    &=x_{1}+\left\lvert x_{13}-x_{1}\right\rvert\dfrac{x_{13}-x_{1}}{\left\lvert x_{13}-x_{1}\right\rvert}=x_{13}\,.
  \end{align*}
  This holds even if one has $\alpha_{1}+\alpha_{2}>1$. Therefore we finally obtain the power series representation at $\xi<\left\lvert z\right\rvert<1$.
\end{proof}
If we set $w'(z)\coloneqq w(z/\xi)$, then we obtain the Schwarz-Christoffel map $w'$ which maps $0,1,1/\xi,\infty$ to $x_{1},x_{2},x_{3},x_{4}$ respectively. 
\begin{lem}
  \label{sqr-rep2}
  Let $w$ be a Schwarz-Christoffel map from the upper half plane to the same convex quadrilateral $x_{1}x_{2}x_{3}x_{4}$ as in Lemma \ref{sqr-rep1} such that $w$ maps $0,1,\xi,\infty$ to $x_{1},x_{2},x_{3},x_{4}$, respectively. Here, we set $\xi=w^{-1}(x_{3})\in(1,\infty)$. We define $x_{13}$ as in Lemma \ref{sqr-rep1}. Then $w$ can be expressed by the following power series:\\
  $(\rm{i})$\,on $\{z\in\overline{\mathbb{H}}\mid\left\lvert z\right\rvert<1\}$
  \begin{align*}
    &w(z)=x_{1}+\dfrac{\Gamma(\alpha_{1}+\alpha_{2})}{\Gamma(\alpha_{1}+1)\Gamma(\alpha_{2})}\dfrac{x_{2}-x_{1}}{_{2}F_{1}(\alpha_{1},1-\alpha_{3},\alpha_{1}+\alpha_{2};1/\xi)}z^{\alpha_{1}}\\
    &\hspace{6cm}\cdot F_{1}\left(\alpha_{1},1-\alpha_{2},1-\alpha_{3},\alpha_{1}+1;z,\dfrac{z}{\xi}\right),
  \end{align*}
  $(\rm{ii})$\,on $\{z\in\overline{\mathbb{H}}\mid\left\lvert 1-z\right\rvert<\min\{1,\xi-1\}\}$
  \begin{align*}
    &w(z)=x_{2}+\dfrac{\Gamma(\alpha_{1}+\alpha_{2})}{\Gamma(\alpha_{1})\Gamma(\alpha_{2}+1)}\dfrac{x_{1}-x_{2}}{_{2}F_{1}(\alpha_{2},1-\alpha_{3},\alpha_{1}+\alpha_{2};(1-\xi)^{-1})}\\
    &\hspace{3cm}\cdot(1-z)^{\alpha_{2}}F_{1}\left(\alpha_{2},1-\alpha_{1},1-\alpha_{3},\alpha_{2}+1;1-z,\dfrac{z-1}{\xi-1}\right),
  \end{align*}
  $(\rm{iii})$\,on $\{z\in\overline{\mathbb{H}}\mid1<\left\lvert z\right\rvert<\xi\}$
  \begin{align*}
    &w(z)=x_{13}-e^{-\pi\alpha_{2}i}\dfrac{\Gamma(\alpha_{1}+\alpha_{2}-1)}{\Gamma(\alpha_{1})\Gamma(\alpha_{2})}\dfrac{x_{2}-x_{1}}{_{2}F_{1}(\alpha_{1},1-\alpha_{3},\alpha_{1}+\alpha_{2};\xi)}\\
    &\hspace{3cm}\cdot z^{\alpha_{1}+\alpha_{2}-1}G_{2}\left(1-\alpha_{2},1-\alpha_{3},\alpha_{1}+\alpha_{2}-1,1-\alpha_{1}-\alpha_{2};-\dfrac{1}{z},-\dfrac{z}{\xi}\right).
  \end{align*}
\end{lem}
\subsection{The behavior of $z_{4,\epsilon}$ and the corresponding trees.}
We first study the limit value of $z_{4,\epsilon}$, and check the correspondence between trees and degenerations of the unit disk (or the upper half plane) with four marked points. 
\begin{lem}
  \label{square-cm1}
  Let $f_{1},f_{2},f_{3},f_{4}$ be functions which satisfy one of conditions in Table \ref{tree-critpt2}, and denote the critical point of $f_{i+1}-f_{i}$ by $p_{i}$. We put $L_{i}^{\epsilon}=\operatorname{graph}(\epsilon df_{i})$, and denote the intersection point of $L_{i}^{\epsilon}$ and $L_{i+1}^{\epsilon}$ by $x_{i}^{\epsilon}$. Let $w_{\epsilon}$ be the Schwarz-Christoffel map from the upper half plane to the convex quadrilateral $x_{1}^{\epsilon}x_{2}^{\epsilon}x_{3}^{\epsilon}x_{4}^{\epsilon}$ such that $w_{\epsilon}(z_{i})=x_{i}^{\epsilon}$ for $i=1,2,3$. Here, we set $z_{1}=1,z_{2}=\infty,z_{3}=0$, and $z_{4,\epsilon}=w_{\epsilon}^{-1}(x_{4}^{\epsilon})\in(0,1)$. Then, we obtain
  \begin{gather*}
    \lim_{\epsilon\to+0}z_{4,\epsilon}=z_{1}\ \ \ \text{if $(p_{3}-p_{1})(p_{4}-p_{2})<0$,}\\
    \lim_{\epsilon\to+0}z_{4,\epsilon}=z_{4}\ \ \ \text{if $(p_{3}-p_{1})(p_{4}-p_{2})>0$.}
  \end{gather*}
\end{lem}
\begin{proof}
  We first consider the case $(p_{3}-p_{1})(p_{4}-p_{2})<0$. Let $s_{b}(Q_{\epsilon})$ be the infimum of lengths of the curves joining the side $x_{1}^{\epsilon}x_{2}^{\epsilon}$ to the side $x_{3}^{\epsilon}x_{4}^{\epsilon}$ for $Q_{\epsilon}$. Here, $Q_{\epsilon}$ is the convex quadrilateral $x_{1}^{\epsilon}x_{2}^{\epsilon}x_{3}^{\epsilon}x_{4}^{\epsilon}$. Then, we have the limit value of $s_{b}(Q_{\epsilon})$ as follows:
  \[
    \lim_{\epsilon\to+0}s_{b}(Q_{\epsilon})=\min\{\left\lvert p_{4}-p_{2}\right\rvert,\left\lvert p_{4}-p_{3}\right\rvert,\left\lvert p_{1}-p_{2}\right\rvert,\left\lvert p_{1}-p_{3}\right\rvert\}\neq 0.
  \]
  Let $m(Q_{\epsilon})$ be the area of $Q_{\epsilon}$. Clearly we have $m(Q_{\epsilon})\to 0\ (\epsilon\to+0)$. By Rengel's inequality (Lemma \ref{rengel}), we obtain $M(Q_{\epsilon})\rightarrow0(\epsilon\to+0)$. This is equivalent to $z_{4,\epsilon}\rightarrow z_{1}\,(\epsilon\to+0)$. Similarly, if $(p_{3}-p_{1})(p_{4}-p_{2})>0$ holds, then we obtain $z_{4,\epsilon}\rightarrow z_{3}\,(\epsilon\to+0)$.
\end{proof}
Here we identify the upper half plane with the unit disk. If $z_{4,\epsilon}$ tends to $z_{3}$ or $z_{1}$, then the boundary collision happens, and the flat bubble arises from the unit disk $[z_{1},z_{2},z_{3},z_{4,\epsilon}]$ with four marked points. It is known that these boundary collisions correspond to trees as Figure \ref{disktree} (see \cite{Dev12}).
\begin{figure}[tb]
  \centering
  \begin{tikzpicture}
    \draw[red,dashed] (180:1.5)--(270:1.5);
    \fill[red,opacity=0.3] (180:1.5) arc (180:270:1.5)--(270:1.5)--cycle;
    \draw (0,0) circle [radius=1.5];
    \draw (0:1.5) node{$\times$};
    \draw (100:1.5) node{$\times$};
    \draw (210:1.5) node{$\times$};
    \draw (230:1.5) node{$\times$};
    \draw (0:1.5) node[right]{$z_{1}$};
    \draw (100:1.5) node[above]{$z_{2}$};
    \draw (210:1.5) node[left]{$z_{3}$};
    \draw (230:1.5) node[below]{$z_{4,\epsilon}$};
    \begin{scope}[xshift=60]
      \draw[<->] (0.1,0)--(2,0);
    \end{scope}
    \begin{scope}[xshift=190]
      \draw[thin] (2,1)--(1,0);
      \draw[thin] (2,-1)--(1,0);
      \draw[thin] (-1,0)--(1,0);
      \draw[thin] (-2,1)--(-1,0);
      \draw[thin] (-2,-1)--(-1,0);
      \draw (-2,1) node[above left]{$v_{3}$};
      \draw (-2,-1) node[below left]{$v_{4}$};
      \draw (2,-1) node[below right]{$v_{1}$};
      \draw (2,1) node[above right]{$v_{2}$};
    \end{scope}
    \begin{scope}[yshift=-120]
      \draw[red,dashed] (30:1.5)--(-60:1.5);
      \fill[red,opacity=0.3] (-60:1.5) arc (-60:30:1.5)--(30:1.5)--cycle;
      \draw (0,0) circle [radius=1.5];
      \draw (0:1.5) node{$\times$};
      \draw (100:1.5) node{$\times$};
      \draw (210:1.5) node{$\times$};
      \draw (-20:1.5) node{$\times$};
      \draw (0:1.5) node[right]{$z_{1}$};
      \draw (100:1.5) node[above]{$z_{2}$};
      \draw (210:1.5) node[left]{$z_{3}$};
      \draw (-20:1.5) node[below right]{$z_{4,\epsilon}$};
    \end{scope}
    \begin{scope}[xshift=60,yshift=-120]
      \draw[<->] (0.1,0)--(2,0);
    \end{scope}
    \begin{scope}[xshift=190,yshift=-120]
      \draw[thin] (2,1)--(1,0);
      \draw[thin] (2,-1)--(1,0);
      \draw[thin] (-1,0)--(1,0);
      \draw[thin] (-2,1)--(-1,0);
      \draw[thin] (-2,-1)--(-1,0);
      \draw (-2,1) node[above left]{$v_{2}$};
      \draw (-2,-1) node[below left]{$v_{3}$};
      \draw (2,-1) node[below right]{$v_{4}$};
      \draw (2,1) node[above right]{$v_{1}$};
    \end{scope}
  \end{tikzpicture}
  \caption{Correspondence between boundary collisions and trees.}
  \label{disktree}
\end{figure}
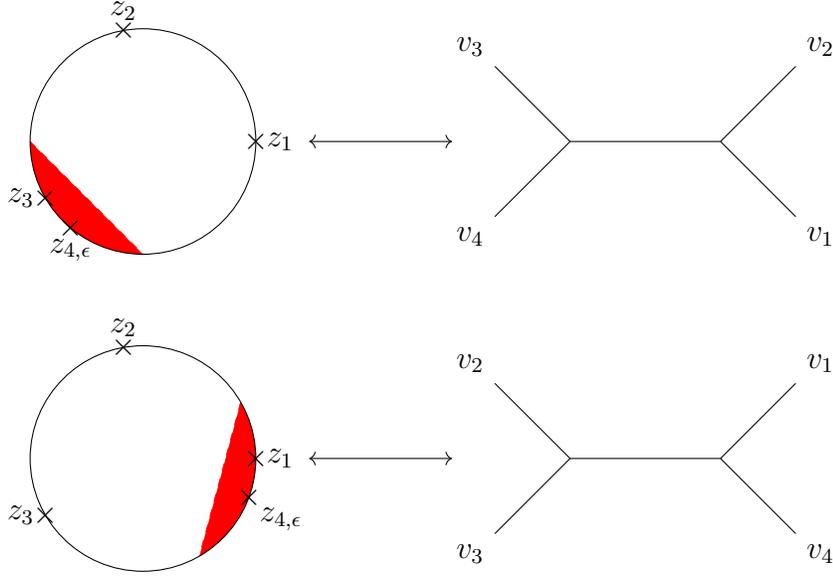
By comparing Table \ref{tree-critpt2} and the results in Lemma \ref{square-cm1}, we obtain the following lemma.
\begin{lem}
  \label{4thpt-tree}
  We set $f_{1},f_{2},f_{3},f_{4}$ and $z_{1},z_{2},z_{3},z_{4,\epsilon}$ the same as these in Lemma \ref{square-cm1}. The tree which corresponds to the deformation of the upper half plane with four points $z_{1},z_{2},z_{3},z_{4,\epsilon}$ by $\epsilon\to+0$ is isomorphic to the tree which is used to construct the gradient tree $I\in\mathcal{M}_{g}(\mathbb{R};\vec{f},\vec{p})$.
\end{lem}
Then we can use the tree which is induced by the limit value of $z_{4,\epsilon}$ to divide the upper half plane into several regions. We next give how to calculate the length of the internal edge.
\begin{lem}
  \label{length-4thpt}
  We assume the same setting in Lemma \ref{square-cm1} and $L_{i}^{\epsilon}\cap L_{j}^{\epsilon}\neq\emptyset,L_{i}^{\epsilon}\neq L_{j}^{\epsilon}$ for $i\neq j$. Then we obtain
  \begin{align*}
    \log z_{4,\epsilon}&\sim -\dfrac{\pi l}{\epsilon}\ (\epsilon\to+0)\ \ \  \text{if $(p_{3}-p_{1})(p_{4}-p_{2})>0$,}\\
    \log(1-z_{4,\epsilon})&\sim -\dfrac{\pi l}{\epsilon}\ (\epsilon\to+0)\ \ \ \text{if $(p_{3}-p_{1})(p_{4}-p_{2})<0$}.
  \end{align*}
  Here, the real number $l>0$ is the length of the interior edge of the tree which constructs the gradient tree $I\in\mathcal{M}_{g}(\mathbb{R};\vec{f},\vec{p})$.
\end{lem}
\begin{proof}
  We first consider the case $(p_{3}-p_{1})(p_{4}-p_{2})>0$. We calculate the ratio of length of two sides $x_{1}^{\epsilon}x_{2}^{\epsilon}$ and $x_{3}^{\epsilon}x_{4}^{\epsilon}$. We have 
  \begin{align*}
    &\left\lvert x_{1}^{\epsilon}-x_{2}^{\epsilon}\right\rvert/\left\lvert x_{3}^{\epsilon}-x_{4}^{\epsilon}\right\rvert\\
    &=\left.\left\lvert \int_{1}^{\infty} \zeta^{\alpha_{3}^{\epsilon}-1}(\zeta-z_{4,\epsilon})^{\alpha_{4}^{\epsilon}-1}(\zeta-1)^{\alpha_{1}^{\epsilon}-1} \,d\zeta \right\rvert\middle/\left\lvert \int_{0}^{z_{4,\epsilon}} \zeta^{\alpha_{3}^{\epsilon}-1}(\zeta-z_{4,\epsilon})^{\alpha_{4}^{\epsilon}-1}(\zeta-1)^{\alpha_{1}^{\epsilon}-1} \,d\zeta \right\rvert\right.\,.
  \end{align*}
  We also have 
  \begin{align*}
    \left\lvert \int_{1}^{\infty} \zeta^{\alpha_{3}^{\epsilon}-1}(\zeta-z_{4,\epsilon})^{\alpha_{4}^{\epsilon}-1}(\zeta-1)^{\alpha_{1}^{\epsilon}-1} \,d\zeta \right\rvert&=\left\lvert \int_{0}^{1} t^{\alpha_{2}^{\epsilon}-1}(1-t)^{\alpha_{1}^{\epsilon}-1}(1-z_{4,\epsilon}t)^{\alpha_{4}^{\epsilon}-1}\,dt\right\rvert\\
    &=\dfrac{\Gamma(\alpha_{1}^{\epsilon})\Gamma(\alpha_{2}^{\epsilon})}{\Gamma(\alpha_{1}^{\epsilon}+\alpha_{2}^{\epsilon})}\,_{2}F_{1}\left(\alpha_{2}^{\epsilon},1-\alpha_{4}^{\epsilon},\alpha_{1}^{\epsilon}+\alpha_{2}^{\epsilon};z_{4,\epsilon}\right)
  \end{align*}
  and 
  \begin{align*}
    \left\lvert \int_{0}^{z_{4,\epsilon}} \zeta^{\alpha_{3}^{\epsilon}-1}(\zeta-z_{4,\epsilon})^{\alpha_{4}^{\epsilon}-1}(\zeta-1)^{\alpha_{1}^{\epsilon}-1} \,d\zeta\right\rvert&=z_{4,\epsilon}^{\alpha_{3}^{\epsilon}+\alpha_{4}^{\epsilon}-1}\left\lvert \int_{0}^{1} t^{\alpha_{3}^{\epsilon}-1}(1-t)^{\alpha_{4}^{\epsilon}-1}(1-z_{4,\epsilon}t)^{\alpha_{1}^{\epsilon}-1}\,dt\right\rvert\\
    &=\dfrac{\Gamma(\alpha_{3}^{\epsilon})\Gamma(\alpha_{4}^{\epsilon})}{\Gamma(\alpha_{3}^{\epsilon}+\alpha_{4}^{\epsilon})}z_{4,\epsilon}^{\alpha_{3}^{\epsilon}+\alpha_{4}^{\epsilon}-1}\,_{2}F_{1}\left(\alpha_{3}^{\epsilon},1-\alpha_{1}^{\epsilon},\alpha_{3}^{\epsilon}+\alpha_{4}^{\epsilon};z_{4,\epsilon}\right).
  \end{align*}
  We therefore obtain
  \[
    \dfrac{\left\lvert x_{1}^{\epsilon}-x_{2}^{\epsilon}\right\rvert}{\left\lvert x_{3}^{\epsilon}-x_{4}^{\epsilon}\right\rvert}=\dfrac{\Gamma(\alpha_{1}^{\epsilon})\Gamma(\alpha_{2}^{\epsilon})}{\Gamma(\alpha_{1}^{\epsilon}+\alpha_{2}^{\epsilon})}\dfrac{\Gamma(\alpha_{3}^{\epsilon}+\alpha_{4}^{\epsilon})}{\Gamma(\alpha_{3}^{\epsilon})\Gamma(\alpha_{4}^{\epsilon})}\left(\dfrac{1}{z_{4,\epsilon}}\right)^{\alpha_{3}^{\epsilon}+\alpha_{4}^{\epsilon}-1}\dfrac{_{2}F_{1}\left(\alpha_{2}^{\epsilon},1-\alpha_{4}^{\epsilon},\alpha_{1}^{\epsilon}+\alpha_{2}^{\epsilon};z_{4,\epsilon}\right)}{_{2}F_{1}\left(\alpha_{3}^{\epsilon},1-\alpha_{1}^{\epsilon},\alpha_{3}^{\epsilon}+\alpha_{4}^{\epsilon};z_{4,\epsilon}\right)},
  \]
  which leads to the following:
  \[
    z_{4,\epsilon}^{\alpha_{3}^{\epsilon}+\alpha_{4}^{\epsilon}-1}=\dfrac{\left\lvert x_{3}^{\epsilon}-x_{4}^{\epsilon}\right\rvert}{\left\lvert x_{1}^{\epsilon}-x_{2}^{\epsilon}\right\rvert}\dfrac{\Gamma(\alpha_{1}^{\epsilon})\Gamma(\alpha_{2}^{\epsilon})}{\Gamma(\alpha_{1}^{\epsilon}+\alpha_{2}^{\epsilon})}\dfrac{\Gamma(\alpha_{3}^{\epsilon}+\alpha_{4}^{\epsilon})}{\Gamma(\alpha_{3}^{\epsilon})\Gamma(\alpha_{4}^{\epsilon})}\dfrac{_{2}F_{1}\left(\alpha_{2}^{\epsilon},1-\alpha_{4}^{\epsilon},\alpha_{1}^{\epsilon}+\alpha_{2}^{\epsilon};z_{4,\epsilon}\right)}{_{2}F_{1}\left(\alpha_{3}^{\epsilon},1-\alpha_{1}^{\epsilon},\alpha_{3}^{\epsilon}+\alpha_{4}^{\epsilon};z_{4,\epsilon}\right)}\,.
  \]
  Since one has $z_{4,\epsilon}\rightarrow 0,\alpha_{3}^{\epsilon}+\alpha_{4}^{\epsilon}\rightarrow 1\,(\epsilon\to+0)$ because of $(p_{3}-p_{1})(p_{4}-p_{2})>0$, we obtain
  \[
    \dfrac{_{2}F_{1}(\alpha_{2}^{\epsilon},1-\alpha_{4}^{\epsilon},\alpha_{1}^{\epsilon}+\alpha_{2}^{\epsilon};z_{4,\epsilon})}{_{2}F_{1}(\alpha_{3}^{\epsilon},1-\alpha_{1}^{\epsilon},\alpha_{3}^{\epsilon}+\alpha_{4}^{\epsilon};z_{4,\epsilon})}\rightarrow 1,\,\dfrac{\Gamma(\alpha_{3}^{\epsilon}+\alpha_{4}^{\epsilon})}{\Gamma(\alpha_{1}^{\epsilon}+\alpha_{2}^{\epsilon})}=\dfrac{\Gamma(\alpha_{3}^{\epsilon}+\alpha_{4}^{\epsilon})}{\Gamma(2-\alpha_{3}^{\epsilon}-\alpha_{4}^{\epsilon})}\rightarrow 1\,(\epsilon\to+0).
  \]
  We here calculate limit values of $(\Gamma(\alpha_{3}^{\epsilon})\Gamma(\alpha_{4}^{\epsilon}))/(\Gamma(\alpha_{2}^{\epsilon})\Gamma(\alpha_{1}^{\epsilon}))$ and ${\left\lvert x_{3}^{\epsilon}-x_{4}^{\epsilon}\right\rvert}/{\left\lvert x_{1}^{\epsilon}-x_{2}^{\epsilon}\right\rvert}$. We consider the case $a_{1},a_{3}\in(a_{2},a_{4})$ only because other cases can be proved in a similar way. In this case, one has $p_{1}<p_{2}<p_{3}<p_{4}$ or $p_{4}<p_{3}<p_{2}<p_{1}$. We first obtain
  \[
    \lim_{\epsilon\to+0}\dfrac{\left\lvert x_{3}^{\epsilon}-x_{4}^{\epsilon}\right\rvert}{\left\lvert x_{1}^{\epsilon}-x_{2}^{\epsilon}\right\rvert}=\dfrac{\left\lvert p_{3}-p_{4}\right\rvert}{\left\lvert p_{1}-p_{2}\right\rvert}
  \]
  since $x_{j}^{\epsilon}\rightarrow p_{j}\,(\epsilon\to+0)$. In this case, by Corollary \ref{angleslim}, we have $\alpha_{1}^{\epsilon},\alpha_{4}^{\epsilon}\rightarrow0\,(\epsilon\to+0)$ and $\alpha_{2}^{\epsilon},\alpha_{3}^{\epsilon}\rightarrow1\,(\epsilon\to+0)$. We therefore have 
  \[
    \lim_{\epsilon\to+0}\dfrac{\Gamma(\alpha_{1}^{\epsilon})\Gamma(\alpha_{2}^{\epsilon})}{\Gamma(\alpha_{3}^{\epsilon})\Gamma(\alpha_{4}^{\epsilon})}=\lim_{\epsilon\to+0}\dfrac{\Gamma(\alpha_{2}^{\epsilon})}{\Gamma(\alpha_{3}^{\epsilon})}\dfrac{\epsilon\Gamma(\alpha_{1}^{\epsilon})}{\epsilon\Gamma(\alpha_{4}^{\epsilon})}=\left|\dfrac{a_{1}-a_{4}}{a_{2}-a_{1}}\right|=\dfrac{a_{4}-a_{1}}{a_{1}-a_{2}},
  \]
  and we obtain
  \[
    \lim_{\epsilon\to+0}z_{4,\epsilon}^{\alpha_{3}^{\epsilon}+\alpha_{4}^{\epsilon}-1}=\dfrac{p_{4}-p_{3}}{p_{2}-p_{1}}\dfrac{a_{4}-a_{1}}{a_{1}-a_{2}}.
  \]
  By $a_{3}\neq a_{1}$, we have 
  \begin{align*}
    (p_{2}-p_{1})(a_{1}-a_{2})&=\left(-\dfrac{b_{3}-b_{2}}{a_{3}-a_{2}}+\dfrac{b_{2}-b_{1}}{a_{2}-a_{1}}\right)(a_{1}-a_{2})\\
    &=\dfrac{-(b_{3}-b_{2})(a_{1}-a_{2})-(b_{2}-b_{1})(a_{3}-a_{2})}{a_{3}-a_{2}}\\
    &=\dfrac{-(b_{3}-b_{1})(a_{3}-a_{2})+(b_{3}-b_{2})(a_{3}-a_{1})}{a_{3}-a_{2}}\\
    &=\left(p_{13}-p_{2}\right)(a_{3}-a_{1}).
  \end{align*}
  Here, $p_{13}$ is the critical point of $f_{3}-f_{1}$. Similarly, we have
  \[
    (p_{4}-p_{3})(a_{4}-a_{1})=\left(p_{13}-p_{3}\right)(a_{3}-a_{1}).
  \]
  Thus, we obtain
  \[
    \lim_{\epsilon\to+0}z_{4,\epsilon}^{\alpha_{3}^{\epsilon}+\alpha_{4}^{\epsilon}-1}=\dfrac{p_{13}-p_{3}}{p_{13}-p_{2}}.
  \]
  On the other hand, by Lemma \ref{k4gradtree}, the gradient tree $I\in\mathcal{M}_{g}(\mathbb{R};\vec{f},\vec{p})$ is constructed by the tree $(b)$ in Figure \ref{k4trees2}. The restriction $I_{int}$ of $I$ to the interior edge satisfies
  \[
    \dfrac{dI_{int}}{dt}=-\grad(f_{3}-f_{1}),I_{int}(0)=p_{2},
  \]
  and the length $l>0$ of the interior edge satisfies
  \[
    e^{-(a_{3}-a_{1})l}=\dfrac{p_{13}-p_{3}}{p_{13}-p_{2}}.
  \]
  By using the results of Lemma \ref{angles}, one has
  \begin{align*}
    \pi\alpha_{3}^{\epsilon}&=\pi-\left\lvert \arctan a_{4}\epsilon-\arctan a_{3}{\epsilon}\right\rvert=\pi-\left(\arctan a_{4}\epsilon-\arctan a_{3}\epsilon\right)\,,\\
    \pi\alpha_{4}^{\epsilon}&=\left\lvert \arctan a_{1}\epsilon-\arctan a_{4}\epsilon\right\rvert= \arctan a_{4}\epsilon-\arctan a_{1}\epsilon\,,
  \end{align*}
  which leads to
  \[
    \lim_{\epsilon\to+0}\dfrac{\pi\alpha_{3}^{\epsilon}+\pi\alpha_{4}^{\epsilon}-\pi}{\epsilon}=\lim_{\epsilon
    \to+0}\dfrac{\arctan a_{3}\epsilon-\arctan a_{1}\epsilon}{\epsilon}=a_{3}-a_{1}.
  \]
  Because of the continuity of the exponential function, we also have 
  \[
    \lim_{\epsilon\to+0}\log z_{4,\epsilon}^{\alpha_{3}^{\epsilon}+\alpha_{4}^{\epsilon}-1}=\log\,e^{-(a_{3}-a_{1})l}=-(a_{3}-a_{1})l.
  \]
  We thus obtain
  \begin{align*}
    &\lim_{\epsilon\to+0}\dfrac{\epsilon}{\pi}\log z_{4,\epsilon}=\lim_{\epsilon\to+0}\dfrac{\epsilon}{\pi(\alpha_{3}^{\epsilon}+\alpha_{4}^{\epsilon}-1)}\log z_{4,\epsilon}^{\alpha_{3}^{\epsilon}+\alpha_{4}^{\epsilon}-1}\\
    &=\left.\left(\lim_{\epsilon\to+0}\log z_{4,\epsilon}^{\alpha_{3}^{\epsilon}+\alpha_{4}^{\epsilon}-1}\right)\middle/\left(\lim_{\epsilon\to+0}\dfrac{\pi(\alpha_{3}^{\epsilon}+\alpha_{4}^{\epsilon}-1)}{\epsilon}\right)\right.=\dfrac{-(a_{3}-a_{1})l}{a_{3}-a_{1}}=-l.
  \end{align*}
  This is equivalent to $\log z_{4,\epsilon}\sim -\pi l/\epsilon\,(\epsilon\to+0)$. We next study the case $(p_{3}-p_{1})(p_{4}-p_{2})<0$. Let $\varphi_{\epsilon}$ be the conformal transformation on the upper half plane such that $\varphi_{\epsilon}(z_{4,\epsilon})=0,\varphi_{\epsilon}(z_{2})=1,\varphi_{\epsilon}(z_{3})=\infty$. This is given by $\varphi_{\epsilon}(z)=(z-z_{4,\epsilon})/z$. Then we have $z_{4,\epsilon}'\coloneqq\varphi_{\epsilon}(z_{1})=1-z_{4,\epsilon}$. By Lemma \ref{square-cm1}, one has $z_{4,\epsilon}'\rightarrow 0\,(\epsilon\to+0)$, and the case $(p_{3}-p_{1})(p_{4}-p_{2})<0$ is reduced to the case $(p_{3}-p_{1})(p_{4}-p_{2})>0$.
\end{proof}
\subsection{Proof of main theorem.}
  We now prove Theorem \ref{k4corr1}. We prove this in the case $a_{4}<a_{1}<a_{3}<a_{2},\,p_{3}<p_{4}<p_{1}<p_{2}$ only since other cases can be proved in a similar way. We first prove \eqref{eq:k4p0-1-1}. We have 
  \begin{align*}
    w_{\epsilon}(z)=x_{3}^{\epsilon}+e^{\pi(2-\alpha_{4}^{\epsilon}-\alpha_{1}^{\epsilon})}z_{4,\epsilon}^{1-\alpha_{3}^{\epsilon}-\alpha_{4}^{\epsilon}}\dfrac{x_{4}-x_{3}}{_{2}F_{1}(\alpha_{3}^{\epsilon},1-\alpha_{1}^{\epsilon},\alpha_{3}^{\epsilon}+\alpha_{4}^{\epsilon};z_{4,\epsilon})}\dfrac{\Gamma(\alpha_{3}^{\epsilon}+\alpha_{4}^{\epsilon})}{\Gamma(\alpha_{3}^{\epsilon})\Gamma(\alpha_{4}^{\epsilon})}\\
    \cdot\int_{0}^{z}\zeta^{\alpha_{3}^{\epsilon}-1}(\zeta-z_{4,\epsilon})^{\alpha_{4}^{\epsilon}-1}(\zeta-1)^{\alpha_{1}^{\epsilon}-1}\,d\zeta
  \end{align*}
  because of 
  \begin{align*}
    &\lim_{\epsilon\to+0}\int_{0}^{z}\zeta^{\alpha_{3}^{\epsilon}-1}(\zeta-z_{4,\epsilon})^{\alpha_{4}^{\epsilon}-1}(\zeta-1)^{\alpha_{1}^{\epsilon}-1}\,d\zeta\\
    &\hspace{3cm}=e^{\alpha_{4}^{\epsilon}+\alpha_{1}^{\epsilon}-2}z_{4,\epsilon}^{\alpha_{3}^{\epsilon}+\alpha_{4}^{\epsilon}-1}\dfrac{\Gamma(\alpha_{3}^{\epsilon})\Gamma(\alpha_{4}^{\epsilon})}{\Gamma(\alpha_{3}^{\epsilon}+\alpha_{4}^{\epsilon})}\,_{2}F_{1}(\alpha_{3}^{\epsilon},1-\alpha_{1}^{\epsilon},\alpha_{3}^{\epsilon}+\alpha_{4}^{\epsilon};z_{4,\epsilon}).
  \end{align*}
  When $z\in\overline{\mathbb{H}}$ satisfies $z\in D(\delta)$ and $\left\lvert z\right\rvert\leq 1$, we obtain
  \begin{align}
    \label{eq4.1-1}
    &\left\lvert w_{\epsilon}(z)-p_{1}\right\rvert\notag\\
    &<\left\lvert x_{3}^{\epsilon}-p_{1}+e^{\pi(2-\alpha_{4}^{\epsilon}-\alpha_{1}^{\epsilon})}z_{4,\epsilon}^{1-\alpha_{3}^{\epsilon}-\alpha_{4}^{\epsilon}}\dfrac{x_{4}-x_{3}}{_{2}F_{1}(\alpha_{3}^{\epsilon},1-\alpha_{1}^{\epsilon},\alpha_{3}^{\epsilon}+\alpha_{4}^{\epsilon};z_{4,\epsilon})}\dfrac{\Gamma(\alpha_{3}^{\epsilon}+\alpha_{4}^{\epsilon})}{\Gamma(\alpha_{3}^{\epsilon})\Gamma(\alpha_{4}^{\epsilon})}\right.\notag\\
    &\hspace{7cm}\left.\cdot\int_{0}^{\delta z/\left\lvert z\right\rvert }\zeta^{\alpha_{3}^{\epsilon}-1}(\zeta-z_{4,\epsilon})^{\alpha_{4}^{\epsilon}-1}(\zeta-1)^{\alpha_{1}^{\epsilon}-1}\,d\zeta\right\rvert\notag\\
    &+\left\lvert z_{4,\epsilon}^{1-\alpha_{3}^{\epsilon}-\alpha_{4}^{\epsilon}}\dfrac{x_{4}-x_{3}}{_{2}F_{1}(\alpha_{3}^{\epsilon},1-\alpha_{1}^{\epsilon},\alpha_{3}^{\epsilon}+\alpha_{4}^{\epsilon};z_{4,\epsilon})}\dfrac{\Gamma(\alpha_{3}^{\epsilon}+\alpha_{4}^{\epsilon})}{\Gamma(\alpha_{3}^{\epsilon})\Gamma(\alpha_{4}^{\epsilon})}\right.\notag\\
    &\hspace{6cm}\left.\cdot\int_{\delta z/\left\lvert z\right\rvert }^{z}\zeta^{\alpha_{3}^{\epsilon}-1}(\zeta-z_{4,\epsilon})^{\alpha_{4}^{\epsilon}-1}(\zeta-1)^{\alpha_{1}^{\epsilon}-1}\,d\zeta\right\rvert.
  \end{align}
  By Lemma \ref{sqr-rep1} $(\rm{iii})$, we have
  \begin{align*}
    &x_{3}^{\epsilon}+e^{\pi(2-\alpha_{4}^{\epsilon}-\alpha_{1}^{\epsilon})}z_{4,\epsilon}^{1-\alpha_{3}^{\epsilon}-\alpha_{4}^{\epsilon}}\dfrac{x_{4}-x_{3}}{_{2}F_{1}(\alpha_{3}^{\epsilon},1-\alpha_{1}^{\epsilon},\alpha_{3}^{\epsilon}+\alpha_{4}^{\epsilon};z_{4,\epsilon})}\dfrac{\Gamma(\alpha_{3}^{\epsilon}+\alpha_{4}^{\epsilon})}{\Gamma(\alpha_{3}^{\epsilon})\Gamma(\alpha_{4}^{\epsilon})}\\
    &\hspace{7cm}\cdot\int_{0}^{\delta z/\left\lvert z\right\rvert }\zeta^{\alpha_{3}^{\epsilon}-1}(\zeta-z_{4,\epsilon})^{\alpha_{4}^{\epsilon}-1}(\zeta-1)^{\alpha_{1}^{\epsilon}-1}\,d\zeta\\
    =&\,x_{13}^{\epsilon}-e^{-\pi\alpha_{4}^{\epsilon}i}\dfrac{\Gamma(\alpha_{3}^{\epsilon}+\alpha_{4}^{\epsilon}-1)}{\Gamma(\alpha_{3}^{\epsilon})\Gamma(\alpha_{4}^{\epsilon})}\dfrac{x_{4}^{\epsilon}-x_{3}^{\epsilon}}{_{2}F_{1}(\alpha_{3}^{\epsilon},1-\alpha_{4}^{\epsilon},\alpha_{3}^{\epsilon}+\alpha_{4}^{\epsilon};z_{4,\epsilon})}\left(\dfrac{\delta z}{z_{4,\epsilon}\left\lvert z\right\rvert }\right)^{\alpha_{3}^{\epsilon}+\alpha_{4}^{\epsilon}-1}\\
    &\hspace{1.5cm}\cdot G_{2}\left(1-\alpha_{4}^{\epsilon},1-\alpha_{1}^{\epsilon},\alpha_{3}^{\epsilon}+\alpha_{4}^{\epsilon}-1,1-\alpha_{3}^{\epsilon}-\alpha_{4}^{\epsilon};-\dfrac{z_{4,\epsilon}\left\lvert z\right\rvert}{\delta z},-\dfrac{\delta z}{\left\lvert z\right\rvert}\right),
  \end{align*}
  where $x_{13}^{\epsilon}$ is the intersection point of $L_{1}^{\epsilon}$ and $L_{3}^{\epsilon}$. Hence, we obtain
  \begin{align}
    \label{eq4.1-2}
    &\left\lvert x_{3}^{\epsilon}-p_{1}+e^{\pi(2-\alpha_{4}^{\epsilon}-\alpha_{1}^{\epsilon})}z_{4,\epsilon}^{1-\alpha_{3}^{\epsilon}-\alpha_{4}^{\epsilon}}\dfrac{x_{4}-x_{3}}{_{2}F_{1}(\alpha_{3}^{\epsilon},1-\alpha_{1}^{\epsilon},\alpha_{3}^{\epsilon}+\alpha_{4}^{\epsilon};z_{4,\epsilon})}\dfrac{\Gamma(\alpha_{3}^{\epsilon}+\alpha_{4}^{\epsilon})}{\Gamma(\alpha_{3}^{\epsilon})\Gamma(\alpha_{4}^{\epsilon})}\right.\notag\\
    &\hspace{7cm}\left.\cdot\int_{0}^{\delta z/\left\lvert z\right\rvert }\zeta^{\alpha_{3}^{\epsilon}-1}(\zeta-z_{4,\epsilon})^{\alpha_{4}^{\epsilon}-1}(\zeta-1)^{\alpha_{1}^{\epsilon}-1}\,d\zeta\right\rvert\notag\\
    &\,\leq\left\lvert x_{13}^{\epsilon}-e^{-\pi\alpha_{4}^{\epsilon}i}z_{4,\epsilon}^{1-\alpha_{3}^{\epsilon}-\alpha_{4}^{\epsilon}}\dfrac{\Gamma(\alpha_{3}^{\epsilon}+\alpha_{4}^{\epsilon}-1)}{\Gamma(\alpha_{3}^{\epsilon})\Gamma(\alpha_{4}^{\epsilon})}\dfrac{x_{4}^{\epsilon}-x_{3}^{\epsilon}}{_{2}F_{1}(\alpha_{3}^{\epsilon},1-\alpha_{4}^{\epsilon},\alpha_{3}^{\epsilon}+\alpha_{4}^{\epsilon};z_{4,\epsilon})}-p_{1}\right\rvert\notag\\
    &\hspace{0.5cm}+\left\lvert -e^{-\pi\alpha_{4}^{\epsilon}i}z_{4,\epsilon}^{1-\alpha_{3}^{\epsilon}-\alpha_{4}^{\epsilon}}\dfrac{\Gamma(\alpha_{3}^{\epsilon}+\alpha_{4}^{\epsilon}-1)}{\Gamma(\alpha_{3}^{\epsilon})\Gamma(\alpha_{4}^{\epsilon})}\dfrac{x_{4}^{\epsilon}-x_{3}^{\epsilon}}{_{2}F_{1}(\alpha_{3}^{\epsilon},1-\alpha_{4}^{\epsilon},\alpha_{3}^{\epsilon}+\alpha_{4}^{\epsilon};z_{4,\epsilon})}\right\rvert \left\lvert \left(\delta\dfrac{z}{\left\lvert z\right\rvert }\right)^{\alpha_{3}^{\epsilon}+\alpha_{4}^{\epsilon}-1}-1\right\rvert\notag\\
    &\hspace{0.5cm}+\left\lvert -e^{-\pi\alpha_{4}^{\epsilon}i}z_{4,\epsilon}^{1-\alpha_{3}^{\epsilon}-\alpha_{4}^{\epsilon}}\dfrac{\Gamma(\alpha_{3}^{\epsilon}+\alpha_{4}^{\epsilon}-1)}{\Gamma(\alpha_{3}^{\epsilon})\Gamma(\alpha_{4}^{\epsilon})}\dfrac{x_{4}^{\epsilon}-x_{3}^{\epsilon}}{_{2}F_{1}(\alpha_{3}^{\epsilon},1-\alpha_{4}^{\epsilon},\alpha_{3}^{\epsilon}+\alpha_{4}^{\epsilon};z_{4,\epsilon})}\right\rvert \left\lvert \left(\delta\dfrac{z}{\left\lvert z\right\rvert}\right)^{\alpha_{3}^{\epsilon}+\alpha_{4}^{\epsilon}-1}\right\rvert\notag\\
    &\hspace{1cm}\cdot\left\lvert G_{2}\left(1-\alpha_{4}^{\epsilon},1-\alpha_{1}^{\epsilon},\alpha_{3}^{\epsilon}+\alpha_{4}^{\epsilon}-1,1-\alpha_{3}^{\epsilon}-\alpha_{4}^{\epsilon};-\dfrac{z_{4,\epsilon}\left\lvert z\right\rvert}{\delta z},-\dfrac{\delta z}{\left\lvert z\right\rvert}\right)-1\right\rvert.
  \end{align}
  We here put 
  \begin{align*}
    M_{0,\epsilon,1}\coloneqq&\,\left\lvert x_{13}^{\epsilon}-e^{-\pi\alpha_{4}^{\epsilon}i}z_{4,\epsilon}^{1-\alpha_{3}^{\epsilon}-\alpha_{4}^{\epsilon}}\dfrac{\Gamma(\alpha_{3}^{\epsilon}+\alpha_{4}^{\epsilon}-1)}{\Gamma(\alpha_{3}^{\epsilon})\Gamma(\alpha_{4}^{\epsilon})}\dfrac{x_{4}^{\epsilon}-x_{3}^{\epsilon}}{_{2}F_{1}(\alpha_{3}^{\epsilon},1-\alpha_{4}^{\epsilon},\alpha_{3}^{\epsilon}+\alpha_{4}^{\epsilon};z_{4,\epsilon})}-p_{1}\right\rvert~,\\
    M_{0,\epsilon,2}\coloneqq&\left\lvert -e^{-\pi\alpha_{4}^{\epsilon}i}z_{4,\epsilon}^{1-\alpha_{3}^{\epsilon}-\alpha_{4}^{\epsilon}}\dfrac{\Gamma(\alpha_{3}^{\epsilon}+\alpha_{4}^{\epsilon}-1)}{\Gamma(\alpha_{3}^{\epsilon})\Gamma(\alpha_{4}^{\epsilon})}\dfrac{x_{4}^{\epsilon}-x_{3}^{\epsilon}}{_{2}F_{1}(\alpha_{3}^{\epsilon},1-\alpha_{4}^{\epsilon},\alpha_{3}^{\epsilon}+\alpha_{4}^{\epsilon};z_{4,\epsilon})}\right\rvert.
  \end{align*}
  We have
  \begin{equation}
    \label{eq4.1-3}
    \left\lvert \left(\delta\dfrac{z}{\left\lvert z\right\rvert}\right)^{\alpha_{3}^{\epsilon}+\alpha_{4}^{\epsilon}-1}-1\right\rvert\leq \left\lvert \delta^{\alpha_{3}^{\epsilon}+\alpha_{4}^{\epsilon}-1}e^{\pi(\alpha_{3}^{\epsilon}+\alpha_{4}^{\epsilon}-1)i}-1\right\rvert.
  \end{equation}
  Here, in the case $n\neq m$ one has 
  \begin{align}
    \label{eq4.1-4}
    (\alpha_{3}^{\epsilon}+\alpha_{4}^{\epsilon}-1)_{n-m}(1-\alpha_{3}^{\epsilon}-\alpha_{4}^{\epsilon})_{m-n}=&\dfrac{\Gamma(\alpha_{3}^{\epsilon}+\alpha_{4}^{\epsilon}-1+n-m)}{\Gamma(\alpha_{3}^{\epsilon}+\alpha_{4}^{\epsilon}-1)}\dfrac{\Gamma(1-\alpha_{3}^{\epsilon}-\alpha_{4}^{\epsilon}+m-n)}{\Gamma(1-\alpha_{3}^{\epsilon}-\alpha_{4}^{\epsilon})}\notag\\
    =&\dfrac{\sin\pi(\alpha_{3}^{\epsilon}+\alpha_{4}^{\epsilon}-1+n-m)}{\sin\pi(\alpha_{3}^{\epsilon}+\alpha_{4}^{\epsilon}-1)}\dfrac{\alpha_{3}^{\epsilon}+\alpha_{4}^{\epsilon}-1}{\alpha_{3}^{\epsilon}+\alpha_{4}^{\epsilon}-1+n-m}\notag\\
    =&(-1)^{n-m}\dfrac{\alpha_{3}^{\epsilon}+\alpha_{4}^{\epsilon}-1}{\alpha_{3}^{\epsilon}+\alpha_{4}^{\epsilon}-1+n-m}.
  \end{align}
  Since we have
  \[
    \left\lvert \dfrac{\alpha}{\alpha+N}\right\rvert\leq\left\lvert \alpha\right\rvert\max\left\{\dfrac{1}{\left\lvert \alpha-1\right\rvert},\dfrac{1}{\alpha+1}\right\}
  \]
  under the conditions $0<\left\lvert \alpha\right\rvert<1$ and $N\in\mathbb{Z}\setminus\{0\}$, we obtain
  \begin{align}
    \label{eq4.1-5}
    \left\lvert (\alpha_{3}^{\epsilon}+\alpha_{4}^{\epsilon}-1)_{n-m}(1-\alpha_{3}^{\epsilon}-\alpha_{4}^{\epsilon})_{m-n}\right\rvert\leq\left\lvert \alpha_{3}^{\epsilon}+\alpha_{4}^{\epsilon}-1\right\rvert \max\left\{\dfrac{1}{\alpha_{3}^{\epsilon}+\alpha_{4}^{\epsilon}},\dfrac{1}{2-\alpha_{3}^{\epsilon}-\alpha_{4}^{\epsilon}}\right\}.
  \end{align}
  We also have $(\alpha_{3}^{\epsilon}+\alpha_{4}^{\epsilon}-1)_{n-m}(1-\alpha_{3}^{\epsilon}-\alpha_{4}^{\epsilon})_{m-n}=1$ in the case $n=m$. By \eqref{eq4.1-3}-\eqref{eq4.1-5}, we have
  \begin{align}
    \label{eq4.1-6}
    &\left\lvert G_{2}\left(1-\alpha_{4}^{\epsilon},1-\alpha_{1}^{\epsilon},\alpha_{3}^{\epsilon}+\alpha_{4}^{\epsilon}-1,1-\alpha_{3}^{\epsilon}-\alpha_{4}^{\epsilon};-\dfrac{z_{4,\epsilon}\left\lvert z\right\rvert}{\delta z},-\dfrac{\delta z}{\left\lvert z\right\rvert}\right)-1\right\rvert\notag\\
    \leq\,& \left(\left\lvert \alpha_{3}^{\epsilon}+\alpha_{4}^{\epsilon}-1\right\rvert \max\left\{\dfrac{1}{\alpha_{3}^{\epsilon}+\alpha_{4}^{\epsilon}},\dfrac{1}{2-\alpha_{3}^{\epsilon}-\alpha_{4}^{\epsilon}}\right\}\right)\sum_{\substack{m,n\geq0\\m\neq n}}\dfrac{(1-\alpha_{4}^{\epsilon})_{m}}{m!}\dfrac{(1-\alpha_{1}^{\epsilon})_{n}}{n!}\left(\dfrac{z_{4,\epsilon}}{\delta }\right)^{m}\delta^{n}\notag\\
    &\hspace{8cm}+\sum_{m=1}^{\infty}\dfrac{(1-\alpha_{4}^{\epsilon})_{m}}{m!}\dfrac{(1-\alpha_{1}^{\epsilon})_{m}}{m!}z_{4,\epsilon}^{m}\notag\\
    =&\left(\left\lvert \alpha_{3}^{\epsilon}+\alpha_{4}^{\epsilon}-1\right\rvert \max\left\{\dfrac{1}{\alpha_{3}^{\epsilon}+\alpha_{4}^{\epsilon}},\dfrac{1}{2-\alpha_{3}^{\epsilon}-\alpha_{4}^{\epsilon}}\right\}\right)\notag\\
    &\hspace{2cm}\cdot\left(F_{1}\left(1,1-\alpha_{4}^{\epsilon},1-\alpha_{1}^{\epsilon},1;\dfrac{z_{4,\epsilon}}{\delta},\delta\right)-\,_{2}F_{1}(1-\alpha_{4}^{\epsilon},1-\alpha_{1}^{\epsilon},1;z_{4,\epsilon})\right)\notag\\
    &\hspace{8cm}+\,_{2}F_{1}(1-\alpha_{4}^{\epsilon},1-\alpha_{1}^{\epsilon},1;z_{4,\epsilon})-1.
  \end{align}
  We denote the rightmost side of the above inequation by $M_{0,\epsilon,3}$. We also have 
  \begin{align}
    \label{eq4.1-7}
    \int_{\delta z/\left\lvert z\right\rvert }^{z}\zeta^{\alpha_{3}^{\epsilon}-1}(\zeta-z_{4,\epsilon})^{\alpha_{4}^{\epsilon}-1}(\zeta-1)^{\alpha_{1}^{\epsilon}-1}\,d\zeta<\,&\int_{\delta}^{\left\lvert z\right\rvert}s^{\alpha_{3}^{\epsilon}-1}\left\lvert \dfrac{z}{\left\lvert z\right\rvert}s-z_{4,\epsilon}\right\rvert^{\alpha_{4}^{\epsilon}-1}\left\lvert \dfrac{z}{\left\lvert z\right\rvert}s-1\right\rvert^{\alpha_{1}^{\epsilon}-1}\,ds\notag\\
    <\,&\delta^{\alpha_{3}^{\epsilon}-1}(\delta-z_{4,\epsilon})^{\alpha_{4}^{\epsilon}-1}\delta^{\alpha_{1}^{\epsilon}-1}(\left\lvert z\right\rvert-\delta)\notag\\
    <\,&\delta^{\alpha_{3}^{\epsilon}-1}(\delta-z_{4,\epsilon})^{\alpha_{4}^{\epsilon}-1}\delta^{\alpha_{1}^{\epsilon}-1}(1-\delta).
  \end{align}
  By \eqref{eq4.1-1}-\eqref{eq4.1-7}, we obtain
  \begin{align}
    \label{eq4.1-8}
    \left\lvert w_{\epsilon}(z)-p_{1}\right\rvert&\,\leq M_{0,\epsilon,1}+M_{0,\epsilon,2}\left(\left\lvert \delta^{\alpha_{3}^{\epsilon}+\alpha_{4}^{\epsilon}-1}e^{\pi(\alpha_{3}^{\epsilon}+\alpha_{4}^{\epsilon}-1)i}-1\right\rvert+\delta^{\alpha_{3}^{\epsilon}+\alpha_{4}^{\epsilon}-1}M_{0,\epsilon,3}\right)\notag\\
    &\hspace{2cm}+M_{0,\epsilon,4}\delta^{\alpha_{3}^{\epsilon}-1}(\delta-z_{4,\epsilon})^{\alpha_{4}^{\epsilon}-1}\delta^{\alpha_{1}^{\epsilon}-1}(1-\delta),
  \end{align}
  where
  \[
    M_{0,\epsilon,4}\coloneqq\left\lvert z_{4,\epsilon}^{1-\alpha_{3}^{\epsilon}-\alpha_{4}^{\epsilon}}\dfrac{x_{4}-x_{3}}{_{2}F_{1}(\alpha_{3}^{\epsilon},1-\alpha_{1}^{\epsilon},\alpha_{3}^{\epsilon}+\alpha_{4}^{\epsilon};z_{4,\epsilon})}\dfrac{\Gamma(\alpha_{3}^{\epsilon}+\alpha_{4}^{\epsilon})}{\Gamma(\alpha_{3}^{\epsilon})\Gamma(\alpha_{4}^{\epsilon})}\right\rvert.
  \]
  We denote the right hand side of \eqref{eq4.1-8} by $M_{0,\epsilon,5}$. By applying the argument in the proof of Lemma \ref{length-4thpt}, we have
  \begin{align}
    \label{eq4.2-1}
    \lim_{\epsilon\to+0}z_{4,\epsilon}^{1-\alpha_{3}^{\epsilon}-\alpha_{4}^{\epsilon}}&=\lim_{\epsilon\to+0}\left(\dfrac{_{2}F_{1}(\alpha_{3}^{\epsilon},1-\alpha_{1}^{\epsilon},\alpha_{3}^{\epsilon}+\alpha_{4}^{\epsilon};z_{4,\epsilon})}{_{2}F_{1}(\alpha_{2}^{\epsilon},1-\alpha_{4}^{\epsilon},\alpha_{1}^{\epsilon}+\alpha_{2}^{\epsilon};z_{4,\epsilon})}\dfrac{\left\lvert x_{1}^{\epsilon}-x_{2}^{\epsilon}\right\rvert}{\left\lvert x_{4}^{\epsilon}-x_{3}^{\epsilon}\right\rvert}\dfrac{\Gamma(\alpha_{3}^{\epsilon})\Gamma(\alpha_{4}^{\epsilon})}{\Gamma(\alpha_{1}^{\epsilon})\Gamma(\alpha_{2}^{\epsilon})}\dfrac{\Gamma(\alpha_{1}^{\epsilon}+\alpha_{2}^{\epsilon})}{\Gamma(\alpha_{3}^{\epsilon}+\alpha_{4}^{\epsilon})}\right)\notag\\
    &=\dfrac{\left\lvert a_{3}-a_{2}\right\rvert}{\left\lvert a_{4}-a_{3}\right\rvert}\dfrac{\left\lvert p_{1}-p_{2}\right\rvert}{\left\lvert p_{4}-p_{3}\right\rvert}=\dfrac{a_{2}-a_{3}}{a_{3}-a_{4}}\dfrac{p_{2}-p_{1}}{p_{4}-p_{3}}.
  \end{align}
  We also have
  \begin{align}
    \label{eq4.2-2}
    &\lim_{\epsilon\to+0}e^{-\pi\alpha_{4}^{\epsilon}i}z_{4,\epsilon}^{1-\alpha_{3}^{\epsilon}-\alpha_{4}^{\epsilon}}\dfrac{\Gamma(\alpha_{3}^{\epsilon}+\alpha_{4}^{\epsilon}-1)}{\Gamma(\alpha_{3}^{\epsilon})\Gamma(\alpha_{4}^{\epsilon})}\dfrac{x_{4}^{\epsilon}-x_{3}^{\epsilon}}{_{2}F_{1}(\alpha_{3}^{\epsilon},1-\alpha_{4}^{\epsilon},\alpha_{3}^{\epsilon}+\alpha_{4}^{\epsilon};z_{4,\epsilon})}\notag\\
    &=\dfrac{a_{2}-a_{3}}{a_{3}-a_{4}}\dfrac{p_{1}-p_{2}}{p_{4}-p_{3}}\dfrac{a_{3}-a_{4}}{a_{3}-a_{1}}(p_{4}-p_{3})=-p_{1}-\dfrac{b_{3}-b_{1}}{a_{3}-a_{1}}=-p_{1}+p_{13}
  \end{align}
  where $p_{13}$ is the critical point of $f_{3}-f_{1}$, and $x_{13}^{\epsilon}\rightarrow p_{13}\,(\epsilon\to+0)$. By Corollary \ref{angleslim}, one has
  \begin{equation}
    \label{eq4.2-3}
    \lim_{\epsilon\to+0}\alpha_{2}^{\epsilon}=\lim_{\epsilon\to+0}\alpha_{3}^{\epsilon}=0,\ \lim_{\epsilon\to+0}\alpha_{1}^{\epsilon}=\lim_{\epsilon\to+0}\alpha_{4}^{\epsilon}=0
  \end{equation}
  and
  \begin{gather}
    \lim_{\epsilon\to+0}\dfrac{\Gamma(\alpha_{3}^{\epsilon}+\alpha_{4}^{\epsilon}-1)}{\Gamma(\alpha_{3}^{\epsilon})\Gamma(\alpha_{4}^{\epsilon})}=\lim_{\epsilon\to+0}\dfrac{\epsilon\Gamma(\alpha_{3}^{\epsilon}+\alpha_{4}^{\epsilon}-1)}{\epsilon\Gamma(\alpha_{3}^{\epsilon})\cdot\Gamma(\alpha_{4}^{\epsilon})}=\dfrac{a_{3}-a_{4}}{a_{3}-a_{1}},\label{eq4.2-4}\\
    \lim_{\epsilon\to+0}\dfrac{\Gamma(\alpha_{3}^{\epsilon}+\alpha_{4}^{\epsilon})}{\Gamma(\alpha_{3}^{\epsilon})\Gamma(\alpha_{4}^{\epsilon})}=\lim_{\epsilon\to+0}\dfrac{\epsilon\Gamma(\alpha_{3}^{\epsilon}+\alpha_{4}^{\epsilon})}{\epsilon\Gamma(\alpha_{3}^{\epsilon})\cdot\Gamma(\alpha_{4}^{\epsilon})}=0.\label{eq4.2-5}
  \end{gather}
  By \eqref{eq4.1-8}-\eqref{eq4.2-5}, we obtain $\lim_{\epsilon\to+0}M_{0,\epsilon,5}=0$. On the other hand, we have 
  \begin{align*}
    w_{\epsilon}(z)&=x_{2}^{\epsilon}-\dfrac{x_{1}^{\epsilon}-x_{2}^{\epsilon}}{_{2}F_{1}(\alpha_{2}^{\epsilon},1-\alpha_{4}^{\epsilon},\alpha_{1}^{\epsilon}+\alpha_{2}^{\epsilon};z_{4,\epsilon})}\dfrac{\Gamma(\alpha_{1}^{\epsilon}+\alpha_{2}^{\epsilon})}{\Gamma(\alpha_{1}^{\epsilon})\Gamma(\alpha_{2}^{\epsilon})}\\
    &\hspace{4cm}\cdot\int_{\infty}^{z}\zeta^{\alpha_{3}^{\epsilon}-1}(\zeta-z_{4,\epsilon})^{\alpha_{4}^{\epsilon}-1}(\zeta-1)^{\alpha_{1}^{\epsilon}-1}\,d\zeta
  \end{align*}
  because of 
  \begin{align*}
    \lim_{z\to1}\int_{\infty}^{z}\zeta^{\alpha_{3}^{\epsilon}-1}(\zeta-z_{4,\epsilon})^{\alpha_{4}^{\epsilon}-1}(\zeta-1)^{\alpha_{1}^{\epsilon}-1}\,d\zeta=-\dfrac{\Gamma(\alpha_{1}^{\epsilon})\Gamma(\alpha_{2}^{\epsilon})}{\Gamma(\alpha_{1}^{\epsilon}+\alpha_{2}^{\epsilon})}\,_{2}F_{1}(\alpha_{2}^{\epsilon},1-\alpha_{4}^{\epsilon},\alpha_{1}^{\epsilon}+\alpha_{2}^{\epsilon};z_{4,\epsilon}).
  \end{align*}
  We next discuss the case $z\in D(\delta)$ and $\left\lvert z\right\rvert\geq1$. In this case, we have
  \begin{align}
    \label{eq4.3-1}
    &\left\lvert w_{\epsilon}(z)-p_{1}\right\rvert\notag\\
    \leq\,&\left\lvert x_{2}^{\epsilon}-p_{1}-\dfrac{x_{1}^{\epsilon}-x_{2}^{\epsilon}}{_{2}F_{1}(\alpha_{2}^{\epsilon},1-\alpha_{4}^{\epsilon},\alpha_{1}^{\epsilon}+\alpha_{2}^{\epsilon};z_{4,\epsilon})}\dfrac{\Gamma(\alpha_{1}^{\epsilon}+\alpha_{2}^{\epsilon})}{\Gamma(\alpha_{1}^{\epsilon})\Gamma(\alpha_{2}^{\epsilon})}\right.\notag\\
    &\hspace{4cm}\left.\cdot\int_{\infty}^{\delta^{-1}z/\left\lvert z\right\rvert }\zeta^{\alpha_{3}^{\epsilon}-1}(\zeta-z_{4,\epsilon})^{\alpha_{4}^{\epsilon}-1}(\zeta-1)^{\alpha_{1}^{\epsilon}-1}\,d\zeta\right\rvert\notag\\
    +&\dfrac{\left\lvert x_{1}^{\epsilon}-x_{2}^{\epsilon}\right\rvert }{_{2}F_{1}(\alpha_{2}^{\epsilon},1-\alpha_{4}^{\epsilon},\alpha_{1}^{\epsilon}+\alpha_{2}^{\epsilon};z_{4,\epsilon})}\dfrac{\Gamma(\alpha_{1}^{\epsilon}+\alpha_{2}^{\epsilon})}{\Gamma(\alpha_{1}^{\epsilon})\Gamma(\alpha_{2}^{\epsilon})}\left\lvert \int_{\delta^{-1}z/\left\lvert z\right\rvert }^{z}\zeta^{\alpha_{3}^{\epsilon}-1}(\zeta-z_{4,\epsilon})^{\alpha_{4}^{\epsilon}-1}(\zeta-1)^{\alpha_{1}^{\epsilon}-1}\,d\zeta\right\rvert.
  \end{align}
  We also have 
  \begin{align*}
    &\int_{\infty}^{\delta^{-1}z/\left\lvert z\right\rvert }\zeta^{\alpha_{3}^{\epsilon}-1}(\zeta-z_{4,\epsilon})^{\alpha_{4}^{\epsilon}-1}(\zeta-1)^{\alpha_{1}^{\epsilon}-1}\,d\zeta\\
    =\,&-\dfrac{\Gamma(\alpha_{2}^{\epsilon})}{\Gamma(\alpha_{2}^{\epsilon}+1)}\left(\delta^{-1}\dfrac{z}{\left\lvert z\right\rvert}\right)^{-\alpha_{2}^{\epsilon}} F_{1}\left(\alpha_{2}^{\epsilon},1-\alpha_{4}^{\epsilon},1-\alpha_{1}^{\epsilon},\alpha_{2}^{\epsilon}+1;\dfrac{z_{4,\epsilon}\left\lvert z\right\rvert }{\delta^{-1}z},\dfrac{\left\lvert z\right\rvert }{\delta^{-1}z}\right).
  \end{align*}
  We thus obtain
  \begin{align}
    \label{eq4.3-2}
    &\left\lvert x_{2}^{\epsilon}-p_{1}-\dfrac{x_{1}^{\epsilon}-x_{2}^{\epsilon}}{_{2}F_{1}(\alpha_{2}^{\epsilon},1-\alpha_{4}^{\epsilon},\alpha_{1}^{\epsilon}+\alpha_{2}^{\epsilon};z_{4,\epsilon})}\dfrac{\Gamma(\alpha_{1}^{\epsilon}+\alpha_{2}^{\epsilon})}{\Gamma(\alpha_{1}^{\epsilon})\Gamma(\alpha_{2}^{\epsilon})}\right.\notag\\
    &\hspace{4cm}\left.\cdot\int_{\infty}^{\delta^{-1}z/\left\lvert z\right\rvert }\zeta^{\alpha_{3}^{\epsilon}-1}(\zeta-z_{4,\epsilon})^{\alpha_{4}^{\epsilon}-1}(\zeta-1)^{\alpha_{1}^{\epsilon}-1}\,d\zeta\right\rvert\notag\\
    <&\left\lvert x_{2}^{\epsilon}-p_{1}+\dfrac{x_{1}^{\epsilon}-x_{2}^{\epsilon}}{_{2}F_{1}(\alpha_{2}^{\epsilon},1-\alpha_{4}^{\epsilon},\alpha_{1}^{\epsilon}+\alpha_{2}^{\epsilon};z_{4,\epsilon})}\dfrac{\Gamma(\alpha_{1}^{\epsilon}+\alpha_{2}^{\epsilon})}{\Gamma(\alpha_{1}^{\epsilon})\Gamma(\alpha_{2}^{\epsilon}+1)}\right\rvert\notag\\
    &+\dfrac{\left\lvert x_{1}^{\epsilon}-x_{2}^{\epsilon}\right\rvert }{_{2}F_{1}(\alpha_{2}^{\epsilon},1-\alpha_{4}^{\epsilon},\alpha_{1}^{\epsilon}+\alpha_{2}^{\epsilon};z_{4,\epsilon})}\dfrac{\Gamma(\alpha_{1}^{\epsilon}+\alpha_{2}^{\epsilon})}{\Gamma(\alpha_{1}^{\epsilon})\Gamma(\alpha_{2}^{\epsilon}+1)} \left\lvert \left(\delta^{-1}\dfrac{z}{\left\lvert z\right\rvert}\right)^{-\alpha_{2}^{\epsilon}}-1\right\rvert\notag\\
    &+\dfrac{\left\lvert x_{1}^{\epsilon}-x_{2}^{\epsilon}\right\rvert }{_{2}F_{1}(\alpha_{2}^{\epsilon},1-\alpha_{4}^{\epsilon},\alpha_{1}^{\epsilon}+\alpha_{2}^{\epsilon};z_{4,\epsilon})}\dfrac{\Gamma(\alpha_{1}^{\epsilon}+\alpha_{2}^{\epsilon})}{\Gamma(\alpha_{1}^{\epsilon})\Gamma(\alpha_{2}^{\epsilon}+1)}\left\lvert \left(\delta^{-1}\dfrac{z}{\left\lvert z\right\rvert}\right)^{-\alpha_{2}^{\epsilon}}\right\rvert\notag\\
    &\hspace{4cm}\cdot \left\lvert F_{1}\left(\alpha_{2}^{\epsilon},1-\alpha_{4}^{\epsilon},1-\alpha_{1}^{\epsilon},\alpha_{2}^{\epsilon}+1;\dfrac{z_{4,\epsilon}\left\lvert z\right\rvert }{\delta^{-1}z},\dfrac{\left\lvert z\right\rvert }{\delta^{-1}z}\right)-1\right\rvert.
  \end{align}
  We here put 
  \begin{align*}
    M_{0,\epsilon,6}&\coloneqq\left\lvert x_{2}^{\epsilon}-p_{1}+\dfrac{x_{1}^{\epsilon}-x_{2}^{\epsilon}}{_{2}F_{1}(\alpha_{2}^{\epsilon},1-\alpha_{4}^{\epsilon},\alpha_{1}^{\epsilon}+\alpha_{2}^{\epsilon};z_{4,\epsilon})}\dfrac{\Gamma(\alpha_{1}^{\epsilon}+\alpha_{2}^{\epsilon})}{\Gamma(\alpha_{1}^{\epsilon})\Gamma(\alpha_{2}^{\epsilon}+1)}\right\rvert,\\
    M_{0,\epsilon,7}&\coloneqq\dfrac{\left\lvert x_{1}^{\epsilon}-x_{2}^{\epsilon}\right\rvert }{_{2}F_{1}(\alpha_{2}^{\epsilon},1-\alpha_{4}^{\epsilon},\alpha_{1}^{\epsilon}+\alpha_{2}^{\epsilon};z_{4,\epsilon})}\dfrac{\Gamma(\alpha_{1}^{\epsilon}+\alpha_{2}^{\epsilon})}{\Gamma(\alpha_{1}^{\epsilon})\Gamma(\alpha_{2}^{\epsilon}+1)}.
  \end{align*}
  We also have
  \begin{equation}
    \label{eq4.3-3}
    \left\lvert \left(\delta^{-1}\dfrac{z}{\left\lvert z\right\rvert}\right)^{-\alpha_{2}^{\epsilon}}-1\right\rvert\leq \left\lvert \delta^{\alpha_{2}^{\epsilon}}e^{-\pi\alpha_{2}^{\epsilon}}-1\right\rvert~, 
  \end{equation}
  \begin{align}
    \label{eq4.3-4}
    &\left\lvert F_{1}\left(\alpha_{2}^{\epsilon},1-\alpha_{4}^{\epsilon},1-\alpha_{1}^{\epsilon},\alpha_{2}^{\epsilon}+1;\dfrac{z_{4,\epsilon}\left\lvert z\right\rvert }{\delta^{-1}z},\dfrac{\left\lvert z\right\rvert }{\delta^{-1}z}\right)-1\right\rvert\notag\\
    &\leq\sum_{\substack{n,m\geq0\\(n,m)\neq(0,0)}}\dfrac{(\alpha_{2}^{\epsilon})_{m+n}(1-\alpha_{4}^{\epsilon})_{m}(1-\alpha_{1}^{\epsilon})_{n}}{(\alpha_{2}^{\epsilon}+1)_{m+n}m!n!}\left\lvert \dfrac{z_{4,\epsilon}\left\lvert z\right\rvert }{\delta^{-1}z}\right\rvert ^{m}\left\lvert \dfrac{\left\lvert z\right\rvert }{\delta^{-1}z}\right\rvert^{n}\notag\\
    &\leq\sum_{\substack{n,m\geq0\\(n,m)\neq(0,0)}}\dfrac{(\alpha_{2}^{\epsilon})_{m+n}(1-\alpha_{4}^{\epsilon})_{m}(1-\alpha_{1}^{\epsilon})_{n}}{(\alpha_{2}^{\epsilon}+1)_{m+n}m!n!}(\delta z_{4,\epsilon})^{m}\delta^{n}\notag\\
    &=F_{1}(\alpha_{2}^{\epsilon},1-\alpha_{4}^{\epsilon},1-\alpha_{1}^{\epsilon},\alpha_{2}^{\epsilon}+1;\delta z_{4,\epsilon},\delta)-1
  \end{align}
  and
  \begin{align}
    \label{eq4.3-5}
    \left\lvert \int_{\delta^{-1}z/\left\lvert z\right\rvert }^{z}\zeta^{\alpha_{3}^{\epsilon}-1}(\zeta-z_{4,\epsilon})^{\alpha_{4}^{\epsilon}-1}(\zeta-1)^{\alpha_{1}^{\epsilon}-1}\,d\zeta\right\rvert\leq&\int_{\left\lvert z\right\rvert}^{\delta^{-1}}t^{\alpha_{3}^{\epsilon}-1}\left\lvert \dfrac{z}{\left\lvert z\right\rvert}t-z_{4,\epsilon}\right\rvert^{\alpha_{4}^{\epsilon}-1}\left\lvert \dfrac{z}{\left\lvert z\right\rvert}t-1\right\rvert^{\alpha_{1}^{\epsilon}-1}\,dt\notag\\
    \leq&(1-z_{4,\epsilon})^{\alpha_{4}^{\epsilon}-1}\delta^{\alpha_{1}^{\epsilon}-1}(\delta^{-1}-1).
  \end{align}
  By \eqref{eq4.3-1}-\eqref{eq4.3-5}, we obtain
  \begin{align}
    \label{eq4.3-6}
    &\left\lvert w_{\epsilon}(z)-p_{1}\right\rvert\notag\\
    &<M_{0,\epsilon,6}+M_{0,\epsilon,7}\left[\left\lvert \delta^{\alpha_{2}^{\epsilon}}e^{-\pi\alpha_{2}^{\epsilon}i}-1\right\rvert+\delta^{\alpha_{2}^{\epsilon}}\left(F_{1}(\alpha_{2}^{\epsilon},1-\alpha_{4}^{\epsilon},1-\alpha_{1}^{\epsilon},\alpha_{2}^{\epsilon}+1;\delta z_{4,\epsilon},\delta)-1\right)\right]\notag\\
    &\hspace{2cm}+M_{0,\epsilon,7}\alpha_{2}^{\epsilon}(1-z_{4,\epsilon})^{\alpha_{4}^{\epsilon}-1}\delta^{\alpha_{1}^{\epsilon}-1}(\delta^{-1}-1).
  \end{align}
  We denote the right hand side by $M_{0,\epsilon,8}$. Since one has
  \[
    \lim_{\epsilon\to+0}\dfrac{\Gamma(\alpha_{1}^{\epsilon}+\alpha_{2}^{\epsilon})}{\Gamma(\alpha_{1}^{\epsilon})\Gamma(\alpha_{2}^{\epsilon}+1)}=1
  \]
  because of Corollary \ref{angleslim}, we obtain $\lim_{\epsilon\to+0}M_{0,\epsilon,8}=0$. When we set $M_{0,\epsilon}\coloneqq\max\{M_{0,\epsilon,5},M_{0,\epsilon,8}\}$, we finally obtain \eqref{eq:k4p0-1-1}. We can also prove \eqref{eq:k4p0-1-2} in a similar way by using Lemma \ref{sqr-rep2}. We next prove \eqref{eq:k4ext-1}, but we can prove it in a similar way as Theorem \ref{k3corr}. Hence we mention only what formulas we use. In the case $i=1,2$, we can prove it by using Lemma \ref{sqr-rep1} $(\rm{ii})$ and $(\rm{iv})$ respectively. On the other hand, Lemma \ref{sqr-rep2} $(\rm{i})$ and $(\rm{ii})$ lead to \eqref{eq:k4ext-1} in the case $i=3,4$ respectively. We next prove \eqref{eq:k4int-1}. We first have
  \begin{align*}
    w_{\epsilon}(z)&=x_{13}^{\epsilon}-e^{-\pi\alpha_{4}^{\epsilon}i}\dfrac{\Gamma(\alpha_{3}^{\epsilon}+\alpha_{4}^{\epsilon}-1)}{\Gamma(\alpha_{3}^{\epsilon})\Gamma(\alpha_{4}^{\epsilon})}\dfrac{x_{4}^{\epsilon}-x_{3}^{\epsilon}}{_{2}F_{1}(\alpha_{3}^{\epsilon},1-\alpha_{4}^{\epsilon},\alpha_{3}^{\epsilon}+\alpha_{4}^{\epsilon};z_{4,\epsilon})}\\
    &\hspace{1cm}\cdot\left(\dfrac{z}{z_{4,\epsilon}}\right)^{\alpha_{3}^{\epsilon}+\alpha_{4}^{\epsilon}-1}G_{2}\left(1-\alpha_{4}^{\epsilon},1-\alpha_{1}^{\epsilon},\alpha_{3}^{\epsilon}+\alpha_{4}^{\epsilon}-1,1-\alpha_{3}^{\epsilon}-\alpha_{4}^{\epsilon};-\dfrac{z_{4,\epsilon}}{z},-z\right)
  \end{align*}
  when $z_{4,\epsilon}/\delta<\left\lvert z\right\rvert<\delta $ holds. By definition of gradient trees, we also have $I_{int}(\tau)=p_{13}+(p_{1}-p_{13})\exp[(a_{1}-a_{3})\tau]$. We thus obtain
  \begin{align}
    \label{eq4.4-1}
    &\left\lvert w_{\epsilon}(z)-I_{int}(\epsilon\tau)\right\rvert\leq\left\lvert x_{13}^{\epsilon}-p_{13}\right\rvert\notag\\
    &\hspace{2cm}+\left\lvert -e^{-\pi\alpha_{4}^{\epsilon}i}z_{4,\epsilon}^{1-\alpha_{3}^{\epsilon}-\alpha_{4}^{\epsilon}}\dfrac{\Gamma(\alpha_{3}^{\epsilon}+\alpha_{4}^{\epsilon}-1)}{\Gamma(\alpha_{3}^{\epsilon})\Gamma(\alpha_{4}^{\epsilon})}\dfrac{x_{4}^{\epsilon}-x_{3}^{\epsilon}}{_{2}F_{1}(\alpha_{3}^{\epsilon},1-\alpha_{4}^{\epsilon},\alpha_{3}^{\epsilon}+\alpha_{4}^{\epsilon};z_{4,\epsilon})}\right\rvert\left\lvert z\right\rvert^{\alpha_{3}^{\epsilon}+\alpha_{4}^{\epsilon}-1}\notag\\
    &\hspace{3cm}\cdot\left\lvert G_{2}\left(1-\alpha_{4}^{\epsilon},1-\alpha_{1}^{\epsilon},\alpha_{3}^{\epsilon}+\alpha_{4}^{\epsilon}-1,1-\alpha_{3}^{\epsilon}-\alpha_{4}^{\epsilon};-\dfrac{z_{4,\epsilon}}{z},-z\right)-1\right\rvert\notag\\
    &\hspace{2cm}+\left\lvert -e^{-\pi\alpha_{4}^{\epsilon}i}z_{4,\epsilon}^{1-\alpha_{3}^{\epsilon}-\alpha_{4}^{\epsilon}}\dfrac{\Gamma(\alpha_{3}^{\epsilon}+\alpha_{4}^{\epsilon}-1)}{\Gamma(\alpha_{3}^{\epsilon})\Gamma(\alpha_{4}^{\epsilon})}\dfrac{x_{4}^{\epsilon}-x_{3}^{\epsilon}}{_{2}F_{1}(\alpha_{3}^{\epsilon},1-\alpha_{4}^{\epsilon},\alpha_{3}^{\epsilon}+\alpha_{4}^{\epsilon};z_{4,\epsilon})} z^{\alpha_{3}^{\epsilon}+\alpha_{4}^{\epsilon}-1}\right.\notag\\
    &\hspace{10cm}\left.-(p_{1}-p_{13})e^{(a_{1}-a_{3})\epsilon\tau}\right\rvert.
  \end{align}
  We here put
  \begin{align*}
    M_{int,\epsilon,1}&\coloneqq\left\lvert x_{13}^{\epsilon}-p_{13}\right\rvert,\\
    M_{int,\epsilon,2}&\coloneqq\left\lvert -e^{-\pi\alpha_{4}^{\epsilon}i}z_{4,\epsilon}^{1-\alpha_{3}^{\epsilon}-\alpha_{4}^{\epsilon}}\dfrac{\Gamma(\alpha_{3}^{\epsilon}+\alpha_{4}^{\epsilon}-1)}{\Gamma(\alpha_{3}^{\epsilon})\Gamma(\alpha_{4}^{\epsilon})}\dfrac{x_{4}^{\epsilon}-x_{3}^{\epsilon}}{_{2}F_{1}(\alpha_{3}^{\epsilon},1-\alpha_{4}^{\epsilon},\alpha_{3}^{\epsilon}+\alpha_{4}^{\epsilon};z_{4,\epsilon})}\right\rvert.
  \end{align*}
  Since we have
  \begin{align*}
    \pi(\alpha_{3}^{\epsilon}+\alpha_{4}^{\epsilon}-1)&=\left\lvert \arctan a_{4}\epsilon-\arctan a_{3}\epsilon\right\rvert-\left\lvert \arctan a_{1}\epsilon-\arctan a_{4}\epsilon\right\rvert\\
    &=\arctan a_{3}\epsilon-\arctan a_{1}\epsilon>0
  \end{align*}
  by Lemma \ref{angles}, we obtain 
  \begin{equation}
    \label{eq4.4-2}
    \left\lvert \phi_{int,l,\delta}^{\epsilon}(\tau,\sigma)\right\rvert^{\alpha_{3}^{\epsilon}+\alpha_{4}^{\epsilon}-1}<\delta^{\alpha_{3}^{\epsilon}+\alpha_{4}^{\epsilon}-1}.
  \end{equation}
  We also have
  \begin{align}
    \label{eq4.4-3}
    &\left\lvert G_{2}\left(1-\alpha_{4}^{\epsilon},1-\alpha_{1}^{\epsilon},\alpha_{3}^{\epsilon}+\alpha_{4}^{\epsilon}-1,1-\alpha_{3}^{\epsilon}-\alpha_{4}^{\epsilon};-\dfrac{z_{4,\epsilon}}{z},-z\right)-1\right\rvert\notag\\
    \leq\,& \left(\left\lvert \alpha_{3}^{\epsilon}+\alpha_{4}^{\epsilon}-1\right\rvert \max\left\{\dfrac{1}{\alpha_{3}^{\epsilon}+\alpha_{4}^{\epsilon}},\dfrac{1}{2-\alpha_{3}^{\epsilon}-\alpha_{4}^{\epsilon}}\right\}\right)\sum_{\substack{m,n\geq0\\m\neq n}}\dfrac{(1-\alpha_{4}^{\epsilon})_{m}}{m!}\dfrac{(1-\alpha_{1}^{\epsilon})_{n}}{n!}\left(\dfrac{z_{4,\epsilon}}{\left\lvert z\right\rvert }\right)^{m}\left\lvert z\right\rvert ^{n}\notag\\
    &\hspace{8cm}+\sum_{m=1}^{\infty}\dfrac{(1-\alpha_{4}^{\epsilon})_{m}}{m!}\dfrac{(1-\alpha_{1}^{\epsilon})_{m}}{m!}z_{4,\epsilon}^{m}\notag\\
    \leq&\left(\left\lvert \alpha_{3}^{\epsilon}+\alpha_{4}^{\epsilon}-1\right\rvert \max\left\{\dfrac{1}{\alpha_{3}^{\epsilon}+\alpha_{4}^{\epsilon}},\dfrac{1}{2-\alpha_{3}^{\epsilon}-\alpha_{4}^{\epsilon}}\right\}\right)\notag\\
    &\hspace{2cm}\cdot\left(F_{1}\left(1,1-\alpha_{4}^{\epsilon},1-\alpha_{1}^{\epsilon},1;\delta,\delta\right)-\,_{2}F_{1}(1-\alpha_{4}^{\epsilon},1-\alpha_{1}^{\epsilon},1;\delta^{2})\right)\notag\\
    &\hspace{8cm}+\,_{2}F_{1}(1-\alpha_{4}^{\epsilon},1-\alpha_{1}^{\epsilon},1;z_{4,\epsilon})-1
  \end{align}
  when $z_{4,\epsilon}/\delta<\left\lvert z\right\rvert<\delta$ holds. We denote the right-most side of the above inequation by $M_{int,\epsilon,3}$. On the other hand, the following holds: 
  \begin{align}
    \label{eq4.4-4}
    &\left\lvert -e^{-\pi\alpha_{4}^{\epsilon}i}z_{4,\epsilon}^{1-\alpha_{3}^{\epsilon}-\alpha_{4}^{\epsilon}}\dfrac{\Gamma(\alpha_{3}^{\epsilon}+\alpha_{4}^{\epsilon}-1)}{\Gamma(\alpha_{3}^{\epsilon})\Gamma(\alpha_{4}^{\epsilon})}\dfrac{x_{4}^{\epsilon}-x_{3}^{\epsilon}}{_{2}F_{1}(\alpha_{3}^{\epsilon},1-\alpha_{4}^{\epsilon},\alpha_{3}^{\epsilon}+\alpha_{4}^{\epsilon};z_{4,\epsilon})} z^{\alpha_{3}^{\epsilon}+\alpha_{4}^{\epsilon}-1}-(p_{1}-p_{13})e^{(a_{1}-a_{3})\tau}\right\rvert\notag\\
    &\leq\left\lvert -e^{-\pi\alpha_{4}^{\epsilon}i}z_{4,\epsilon}^{1-\alpha_{3}^{\epsilon}-\alpha_{4}^{\epsilon}}\dfrac{\Gamma(\alpha_{3}^{\epsilon}+\alpha_{4}^{\epsilon}-1)}{\Gamma(\alpha_{3}^{\epsilon})\Gamma(\alpha_{4}^{\epsilon})}\dfrac{x_{4}^{\epsilon}-x_{3}^{\epsilon}}{_{2}F_{1}(\alpha_{3}^{\epsilon},1-\alpha_{4}^{\epsilon},\alpha_{3}^{\epsilon}+\alpha_{4}^{\epsilon};z_{4,\epsilon})}-(p_{1}-p_{13})\right\rvert \left\lvert z\right\rvert^{\alpha_{3}+\alpha_{4}^{\epsilon}-1}\notag\\
    &\hspace{2cm}+\left\lvert p_{1}-p_{13}\right\rvert \left\lvert z^{\alpha_{3}^{\epsilon}+\alpha_{4}^{\epsilon}-1}-e^{(a_{1}-a_{3})\epsilon\tau}\right\rvert.
  \end{align}
  We here define $M_{int,\epsilon,4}$ as follows:
  \[
    M_{int,\epsilon,4}\coloneqq\left\lvert -e^{-\pi\alpha_{4}^{\epsilon}i}z_{4,\epsilon}^{1-\alpha_{3}^{\epsilon}-\alpha_{4}^{\epsilon}}\dfrac{\Gamma(\alpha_{3}^{\epsilon}+\alpha_{4}^{\epsilon}-1)}{\Gamma(\alpha_{3}^{\epsilon})\Gamma(\alpha_{4}^{\epsilon})}\dfrac{x_{4}^{\epsilon}-x_{3}^{\epsilon}}{_{2}F_{1}(\alpha_{3}^{\epsilon},1-\alpha_{4}^{\epsilon},\alpha_{3}^{\epsilon}+\alpha_{4}^{\epsilon};z_{4,\epsilon})}-(p_{1}-p_{13})\right\rvert.
  \]
  One has
  \begin{align}
    \label{eq4.4-5}
    &\left\lvert (\phi_{int}(\tau,\sigma))^{\alpha_{3}^{\epsilon}+\alpha_{4}^{\epsilon}-1}-e^{(a_{1}-a_{3})\epsilon\tau}\right\rvert\notag\\
    &\leq~\delta^{\alpha_{3}^{\epsilon}+\alpha_{4}^{\epsilon}-1}\left\lvert e^{-\pi(\alpha_{3}^{\epsilon}+\alpha_{4}^{\epsilon}-1)i}-1\right\rvert +\left\lvert e^{-\pi(\alpha_{3}^{\epsilon}+\alpha_{4}^{\epsilon}-1)\tau}-e^{-(a_{3}-a_{1})\epsilon\tau}\right\rvert.
  \end{align}
  By $\alpha_{3}^{\epsilon}+\alpha_{4}^{\epsilon}-1>0$, we have $\left\lvert \phi_{int}(\tau,\sigma)\right\rvert^{\alpha_{3}^{\epsilon}+\alpha_{4}^{\epsilon}-1}<\delta^{\alpha_{3}+\alpha_{4}^{\epsilon}-1}$. We also have 
  \begin{align}
    \label{eq4.4-6}
    &\left\lvert e^{-\pi(\alpha_{3}^{\epsilon}+\alpha_{4}^{\epsilon}-1)\tau}-e^{(a_{1}-a_{3})\epsilon\tau}\right\rvert\notag\\
    &\leq~\left\lvert  \exp\left[\dfrac{-\pi(\alpha_{3}^{\epsilon}+\alpha_{4}^{\epsilon}-1)}{\pi(\alpha_{3}^{\epsilon}+\alpha_{4}^{\epsilon}-1)-(a_{3}-a_{1})\epsilon}\log\dfrac{\pi(\alpha_{3}^{\epsilon}+\alpha_{4}^{\epsilon}-1)}{(a_{3}-a_{1})\epsilon}\right]\right.\notag\\
    &\hspace{1cm}\left.-\exp\left[\dfrac{-(a_{3}-a_{1})\epsilon}{\pi(\alpha_{3}^{\epsilon}+\alpha_{4}^{\epsilon}-1)-(a_{3}-a_{1})\epsilon}\log\dfrac{\pi(\alpha_{3}^{\epsilon}+\alpha_{4}^{\epsilon}-1)}{(a_{3}-a_{1})\epsilon}\right]\right\rvert.
  \end{align}
  We denote the right hand side of the above inequation by $M_{int,\epsilon,5}$. When we set $M_{int,\epsilon}$ as
  \begin{align*}
   M_{int,\epsilon}\coloneqq\,& M_{int,\epsilon,1}+M_{int,\epsilon,2}\delta^{\alpha_{3}^{\epsilon}+\alpha_{4}^{\epsilon}-1}M_{int,\epsilon,3}\\
   &+M_{int,\epsilon,4}\delta^{\alpha_{3}^{\epsilon}+\alpha_{4}^{\epsilon}-1}+\left\lvert p_{1}-p_{13}\right\rvert \left(\delta^{\alpha_{3}^{\epsilon}+\alpha_{4}^{\epsilon}-1}\left\lvert e^{-\pi(\alpha_{3}^{\epsilon}+\alpha_{4}^{\epsilon}-1)i}-1\right\rvert+M_{int,\epsilon,5}\right),
  \end{align*}
  we obtain $\sup_{(\tau,\sigma)\in\Theta_{int}^{\epsilon}(\delta)}\left\lvert w_{\epsilon}\circ\phi_{int}(\tau,\sigma)\right\rvert<M_{int,\epsilon}$. As we calculated before, we have
  \begin{gather}
    \lim_{\epsilon\to+0}M_{int,\epsilon,5}=\left\lvert e^{-1}-e^{-1}\right\rvert=0\label{eq4.4-7}\\
    \lim_{\epsilon\to+0}\left\lvert -e^{-\pi\alpha_{4}^{\epsilon}i}z_{4,\epsilon}^{1-\alpha_{3}^{\epsilon}-\alpha_{4}^{\epsilon}}\dfrac{\Gamma(\alpha_{3}^{\epsilon}+\alpha_{4}^{\epsilon}-1)}{\Gamma(\alpha_{3}^{\epsilon})\Gamma(\alpha_{4}^{\epsilon})}\dfrac{x_{4}^{\epsilon}-x_{3}^{\epsilon}}{_{2}F_{1}(\alpha_{3}^{\epsilon},1-\alpha_{4}^{\epsilon},\alpha_{3}^{\epsilon}+\alpha_{4}^{\epsilon};z_{4,\epsilon})}-(p_{1}-p_{13})\right\rvert=0.\label{eq4.4-8}
  \end{gather}
  By \eqref{eq4.4-1}-\eqref{eq4.4-8}, we obtain $\lim_{\epsilon\to+0}M_{int,\epsilon}=0$, and we finally prove Theorem \ref{k4corr1} in the case $a_{4}<a_{1}<a_{3}<a_{2},p_{3}<p_{4}<p_{1}<p_{2}$.

  \vspace{0.5cm}
  Theorem \ref{k4corr1} seems to hold even if the quadrilateral is not generic, for instance, there exist parallel pairs of edges of the quadrilateral $x_{1}^{\epsilon}x_{2}^{\epsilon}x_{3}^{\epsilon}x_{4}^{\epsilon}$. However, in this case, we can not apply the connection formula of $F_{1}$ and $G_{2}$ (Lemma \ref{F1conn1}) directly to the proof. We also expect that the method in Theorem \ref{k4corr1} can be applied to the case $k\geq5$. However, functions which describe the Schwarz-Christoffel map differ from hypergeometric functions appearing in this paper. This makes proving convergence of holomorphic disks in the general case more harder. We will discuss these cases in the future. 

\bibliographystyle{plain} 
  
\end{document}